\documentclass[11pt, letterpaper,reqno]{amsart}
\usepackage[T1]{fontenc}
\usepackage[utf8]{inputenc}
\usepackage{amsmath}
\usepackage{amsthm,amssymb,bbm}
 \usepackage[colorlinks=true, linkcolor=blue, citecolor=red, pdfborder={0 0 0}]{hyperref}
 \usepackage{cleveref}
\usepackage[usenames,dvipsnames]{color}
\usepackage{graphicx}
\usepackage{comment}
\usepackage{enumerate}
\usepackage{enumitem}
\usepackage[normalem]{ulem}
\usepackage{relsize}

\makeatletter
\theoremstyle{plain}
\newtheorem{thm}{\protect\theoremname}
\theoremstyle{plain}
\newtheorem{lem}[thm]{\protect\lemmaname}
\theoremstyle{plain}
\newtheorem{cor}[thm]{\protect\corname}
\theoremstyle{definition}
\newtheorem{example}[thm]{\protect\examplename}
\theoremstyle{definition}
\newtheorem{defn}[thm]{\protect\definitionname}
\theoremstyle{plain}
\newtheorem{prop}[thm]{\protect\propositionname}
\theoremstyle{plain}
\newtheorem*{lem*}{\protect\lemmaname}
\theoremstyle{remark}
\newtheorem{claim}[thm]{\protect\claimname}
\theoremstyle{remark}
\newtheorem{remark}[thm]{\protect\remarkname}

\newcommand{\abbr}[1]{{\sc\lowercase{#1}}}

\makeatother

\providecommand{\claimname}{Claim}
\providecommand{\definitionname}{Definition}
\providecommand{\examplename}{Example}
\providecommand{\lemmaname}{Lemma}
\providecommand{\corname}{Corollary}
\providecommand{\propositionname}{Proposition}
\providecommand{\theoremname}{Theorem}
\providecommand{\remarkname}{Remark}

\numberwithin{equation}{section}
\numberwithin{thm}{section}

\newcommand\nc\newcommand
\nc\dmo\DeclareMathOperator
\nc\bs\boldsymbol

\nc{\red}[1]{{\color{red} #1}}
\nc{\blue}[1]{{\color{blue} #1}}
\nc{\green}[1]{{\color{green} #1}}
\nc{\cyan}[1]{{\color{cyan} #1}}
\definecolor{purple}{rgb}{0.9,0,0.8}
\nc{\purple}[1]{{\color{purple} #1}}
\definecolor{gray}{rgb}{0.5,0.5,0.5}
\nc{\gray}[1]{{\color{gray} #1}}
\nc{\note}[1]{{\blue{\textup{[#1]}}}}

\nc{\A}{\mathbb{A}}
\nc{\B}{\mathbb{B}}
\nc{\D}{\mathbb{D}}
\nc{\E}{\mathbb{E}}
\nc{\F}{\mathbb{F}}
\nc{\bI}{\mathbb{I}}
\nc{\J}{\mathbb{J}}
\nc{\K}{\mathbb{K}}
\nc{\M}{\mathbb{M}}
\nc{\N}{\mathbb{N}}
\nc{\bO}{\mathbb{O}}

	\renewcommand{\P}{\mathbb{P}}
\nc{\Q}{\mathbb{Q}}
\nc{\R}{\mathbb{R}}
\nc{\bS}{\mathbb{S}}
\nc{\T}{\mathbb{T}}
\nc{\V}{\mathbb{V}}
\nc{\W}{\mathbb{W}}
\nc{\X}{\mathbb{X}}
\nc{\Y}{\mathbb{Y}}
\nc{\Z}{\mathbb{Z}}

\nc{\cA}{\mathcal{A}}
\nc{\cB}{\mathcal{B}}
\nc{\cC}{\mathcal{C}}
\nc{\cD}{\mathcal{D}}
\nc{\cE}{\mathcal{E}}
\nc{\cF}{\mathcal{F}}
\nc{\cG}{\mathcal{G}}
\nc{\cH}{\mathcal{H}}
\nc{\cI}{\mathcal{I}}
\nc{\cJ}{\mathcal{J}}
\nc{\cK}{\mathcal{K}}
\nc{\cL}{\mathcal{L}}
\nc{\cM}{\mathcal{M}}
\nc{\cN}{\mathcal{N}}
\nc{\cO}{\mathcal{O}}
\nc{\cP}{\mathcal{P}}
\nc{\cQ}{\mathcal{Q}}
\nc{\cR}{\mathcal{R}}
\nc{\cS}{\mathcal{S}}
\nc{\cT}{{\cA}}
\nc{\cU}{{\mathcal{U}}}
\nc{\cV}{\mathcal{V}}
\nc{\cW}{\mathcal{W}}
\nc{\cX}{\mathcal{X}}
\nc{\cY}{\mathcal{Y}}
\nc{\cZ}{\mathcal{Z}}

\nc{\eps}{\varepsilon}
\nc{\epp}{\epsilon}
\nc{\lam}{\lambda}
\nc{\ls}{\lesssim}
\nc{\gs}{\gtrsim}
\def \lf {\lfloor}
\def \rf {\rfloor}
\dmo{\lls}{\,\ls\,}
\dmo{\ggs}{\,\gs\,}

\dmo{\adj}{adj}
\dmo{\Hom}{Hom}
\dmo{\injo}{inj}
\dmo{\proj}{proj}
\dmo{\tr}{tr}
\dmo{\Tr}{Tr}
\nc{\Span}{\operatorname{span}}
\def \tran {\mathsf{T}}

\dmo{\rank}{rank}
\nc\wt\widetilde
\nc\wh\widehat

\dmo{\pr}{\mathbb{P}}
\dmo{\e}{\mathbb{E}}
\dmo{\Var}{Var}
\dmo{\Corr}{Corr}
\dmo{\Cov}{Cov}
\dmo{\1}{\mathbf{1}}
\nc{\eqd}{\stackrel{\text{\tiny $d$}}{=}}

\nc{\iid}{\abbr{iid}}
\nc{\ldp}{\abbr{ldp }}
\nc{\ldps}{\abbr{ldp}s }
\dmo{\Ber}{Ber}
\dmo{\Bern}{Bernoulli}
\dmo{\Bin}{Bin}
\dmo{\Poi}{Poi}
\nc{\bvec}{{\underline{b}}}
\nc{\ind}{\mathbb{I}}

\dmo\indhom{ind}
\nc\Verts{\mathsf{V}}
\nc\Edges{\mathsf{E}}
\nc\verts{\mathsf{v}}
\nc\edges{\mathsf{e}}
\nc\ER{Erd\H{o}s--R\'enyi }

\dmo\UT{UT}
\dmo\LT{LT}
\nc\bA{{\bs{A}}}
\nc\bG{{\bs{G}}}
\nc{\Tens}{{\mathcal{Z}}}
\nc{\Tensn}{{\Tens_n}}
\nc{\Tensnr}{{\Tens_{n,r}}}			
\nc{\tens}{{Z}}
\nc{\Syms}{{\mathcal{S}}}
\nc{\Symsn}{{\Syms_n}}
\nc{\Symsnr}{{\Syms_{n,r}}}		
\nc{\syms}{{S}}
\nc{\usym}{{\underline{\syms}}}
\nc{\uS}{{\underline{\syms}}}
\nc{\cAn}{{\cA_n}}
\nc{\cAnr}{{\cA_{n,r}}}		
\nc{\cQn}{{\cQ_n}}
\nc{\cQnr}{{\cQ_{n,r}}}			
\nc{\Test}{{\mathcal{T}}}
\nc{\test}{{T}}
\nc{\ttest}{{\tau}}
\nc{\htest}{{\wh{\test}}}
\nc{\pnorm}{{\Delta,p}}
\nc{\ba}{\bs a}
\nc{\base}{{\mathsf{f}}}
\nc{\Base}{{\mathsf{F}}}		
\nc{\Basys}{{\mathsf{B}}}
\nc{\BBasys}{{\mathbbmss{B}}}
\nc{\dom}{{\mathsf{f}}}
\nc{\Dom}{{\mathsf{F}}}
\nc{\nbr}{\mathcal{N}}
\nc{\bdy}{\partial}
\nc{\dd}{{d}}
\nc{\ddstar}{\dd_\star}	
\nc{\ddprime}{\dd'}
\nc{\DDprime}{{\Delta'}}
\nc{\DDedge}{{\Delta_\star}}
\nc{\cover}{\mathcal{I}}
\nc{\convex}{{\mathcal{K}}}
\nc{\badset}{{\mathcal{E}_\star}}
\nc{\uX}{\underline{X}}
\nc{\uY}{\underline{Y}}
\nc{\uZ}{\underline{Z}}
\nc{\uA}{\underline{A}}
\nc{\uB}{\underline{B}}
\nc{\uP}{\underline{P}}
\nc{\uQ}{\underline{\smash{Q}}}
\nc{\hQ}{{\wh Q}}
\dmo{\hull}{hull}
\nc{\HG}{{G}}
\nc{\Lg}{{K}}
\nc{\resid}{{R}}
\nc{\tA}{{\wt{A}}}
\nc{\tB}{{\wt{B}}}
\nc{\tZ}{{\wt{Z}}}
\nc{\utB}{\wt{\uB}}
\nc{\utA}{\wt{\uA}}
\nc{\utZ}{\wt{\uZ}}
\nc{\jay}{{J_{n,r}}}
\dmo\str{{str}}
\dmo\rand{{rand}}
\dmo{\ME}{{ME}}
\nc{\nbhd}{{\cU}}

\dmo{\CN}{{CN}}
\dmo{\cn}{{cn}}

\nc{\growing}{{w}}
\nc{\rate}{{R}}
\nc{\PUP}{{(1+\delta)^{-1/\edges(H)}}}
\dmo{\DKL}{{D}}
\dmo{\eye}{{I}}
\nc{\uH}{{\uline{H}}}
\nc{\udelta}{{\uline{\delta}}}
\nc{\uL}{{\uline{L}}}

\nc{\nick}[1]{#1}
\nc{\purp}[1]{{\color{purple} #1}}
\nc{\Lp}{\big|\log(p \wedge (1-p))\big|}

\begin{document}
\title[Regularity method and LDP for random hypergraphs]{Regularity method 
and large deviation principles for the Erd\H{o}s--R\'enyi hypergraph}

\author[N.\ A.\ Cook]{Nicholas A.\ Cook$^\star$}
 \address{$^\star$Department of Mathematics, Duke University, 
120 Science Dr, Durham, NC 27710}\email{nickcook@math.duke.edu}
\author[A.\ Dembo]{Amir Dembo$^{\ddagger}$}\thanks{$
{}^{\ddagger}$Partially supported by NSF grant DMS-1954337.}
\address{$^{\ddagger}$Department of Mathematics, Stanford University, 
Stanford, CA 94305}\email{adembo@stanford.edu}
\author[H.\ Pham]{
Huy Tuan Pham$^{\mathsection}$}
\address{$^{\mathsection}$Department of Mathematics, Stanford University, 
Stanford, CA 94305}\email{huypham@stanford.edu}
\date{\today}


\maketitle
\begin{abstract}
We develop a quantitative large deviations theory for random 
hypergraphs, which rests on tensor decomposition and counting lemmas 
under a novel family of cut-type norms. As our main application, we obtain
sharp asymptotics for joint upper and lower tails of homomorphism counts in the $r$-uniform Erd\H{o}s--R\'enyi hypergraph for any fixed $r\ge 2$, generalizing and improving on previous results for the Erd\H{o}s--R\'enyi graph ($r=2$). The theory is sufficiently quantitative to allow the density of the hypergraph to vanish at a polynomial rate, and additionally yields tail asymptotics for other nonlinear functionals, such as induced homomorphism counts.
\end{abstract}

\section{Introduction}

\subsection{Overview}
\label{sec:overview}

For a fixed integer $r\ge2$ and (large) integer $n$, let $\cQ_{n,r}=[0,1]^{{[n]\choose r}}$ denote the set of $[0,1]$-valued functions 
on $r$-sets $I=\{i_1,\dots, i_r\}\subset[n]:=\{1,\dots,n\}$.
We associate elements of $\cQ_{n,r}$ with edge-weighted $r$-uniform hypergraphs over $[n]$ with edge weights $Q(I)$, $I\in{{[n]\choose r}}$.
The set $\cQ_{n,r}$ also parametrizes the collection of inhomogeneous \ER measures $\mu_Q$ over \emph{unweighted} $r$-uniform hypergraphs (\emph{$r$-graphs}), where for a random $r$-graph with distribution $\mu_Q$, 
each $r$-set $I$ is included as an edge independently with probability 
$Q(I)$. 
For the case that $Q\equiv p$ for some $p\in (0,1)$ we have that $\mu_Q=:\mu_p$ is the distribution of the \ER $r$-graph $\bG=\bG_{n,p}^{(r)}$. 

Our aim is to establish precise estimates, at exponential scale, for probabilities of rare events for $\bG$, and in particular to justify asymptotics (in the large $n$ limit) of the form
\begin{equation}	\label{LDP.top}
\log\mu_p(\cE) \sim - \inf\big\{ \eye_p(Q): Q\in \cE'\big\}
\end{equation}
for general sets of hypergraphs $\cE$ (viewed as subsets of the discrete cube $\{0,1\}^{{[n]\choose r}}$), where the infimum is taken over an appropriate ``approximation'' $\cE'$
of $\cE$ in the solid cube $\cQ_{n,r}$.
Here, $\eye_p(Q):=\DKL(\mu_Q\|\mu_p)$ 
is the relative entropy
of $\mu_Q$ with respect to $\mu_p$ (see \eqref{def:eyepQ}).

Of particular interest are tail estimates for the number of occurrences of a fixed $r$-graph $H$ as a sub-hypergraph of $\bG$, which have been the subject of intense activity in recent years, mainly for the case $r=2$ (we review the literature in $\mathsection$\ref{sec:previously} below). 
Writing $\Verts(H), \Edges(H)$ for the vertex and edge sets of an $r$-graph $H$, and $\verts(H), \edges(H)$ for their respective cardinalities, we recall the \emph{homomorphism density} of $H$ in 
a weighted hypergraph $Q\in\cQ_{n,r}$ is 
\begin{equation}	\label{def:tQ}
t(H,Q) = \frac1{n^{\verts(H)}} \sum_{\phi: \Verts(H)\to [n]} \prod_{e\in \Edges(H)} Q(\phi(e))\,,
\end{equation}
where we have extended $Q$ symmetrically to a function on ordered $r$-tuples, taking value zero when the coordinates are not all distinct.
For the case that $Q=A_G$ is the 0--1 adjacency tensor of an $r$-graph $G$, this
is the probability that a uniform random mapping of the vertices of $H$ into the vertices of $G$  maps the edges of $H$ onto edges of $G$. In this case we abusively write $t(H,G):=t(H,A_G)$. 

As an application of our main results -- namely, the quantitative \ldp of \Cref{thm:LDP} (a consequence of a tensor \emph{decomposition lemma} (\Cref{thm:reg})) and a \emph{counting lemma} (\Cref{thm:count}) -- we obtain the following instances of \eqref{LDP.top} for intersections of super/sub-level sets of functionals \eqref{def:tQ}. 
For a sequence of $r$-graphs $\uH=(H_1,\dots, H_m)$ and $\udelta=(\delta_1,\dots,\delta_m)\in\R_+^m$, define the joint upper-tail rate and corresponding entropic optimization problem
\begin{small}
\begin{align}
\UT_{n,p}(\uH,\udelta) &:= -\log\P\big(\,t(H_k,\bG)\ge (1+\delta_k)p^{\edges(H_k)}\,,\,1\le k\le m\big)\,, 	\label{def:UT}\\
\Phi_{n,p}(\uH,\udelta) &:= \inf_{Q\in\cQ_{n,r}}\big\{\eye_p(Q): \, t(H_k,Q) \ge (1+\delta_k)p^{\edges(H_k)}\,,\,1\le k\le m\big\}\,, 	\label{def:Phi}
\end{align}
\end{small}
and for $\udelta\in (0,1)^m$ the analogous joint lower-tail quantities
\begin{small}
\begin{align}
\LT_{n,p}(\uH,\udelta) &:= -\log\P\big(\,t(H_k,\bG)\le (1-\delta_k)p^{\edges(H_k)}\,,\, 1\le k\le m\big)\,, 	\label{def:LT}\\
\Psi_{n,p}(\uH,\delta) &:= \inf_{Q\in\cQ_{n,r}}\big\{\eye_p(Q): \, t(H_k,Q) \le (1-\delta_k)p^{\edges(H)}\,,\,1\le k\le m\big\}\,.
\label{def:Psi}
\end{align}
\end{small}
The scaling by $p^{\edges(H)}$ is natural as one checks that $\E\, t(H,\bG) \sim p^{\edges(H)}$ for the range of $p$ considered below. 
For an $r$-graph $H$ we write $\Delta(H)$ for its max-degree -- that is, the maximum over $v\in\Verts(H)$ of $\deg_H(v)=|\{e\in \Edges(H): v\in e\}|$. In the following we additionally refer to a hypergraph parameter $\DDprime(H)$ whose definition is deferred to \eqref{eq:def-delta'},
only noting here that it always lies in the range
\begin{equation}	\label{DDprime.range}
\tfrac1r(\Delta(H)+1)\le \DDprime(H) \le \Delta(H)+1\,
\end{equation}
with the lower bound attained (for instance) by stars, and the upper bound by cliques.
For our conventions on asymptotic notation see $\mathsection$\ref{sec:notation}.

\begin{thm}\label{thm:tails}
Fix $r$-graphs $H_1,\dots, H_m$. 
Let $\Delta_{\max}=\max_k\Delta(H_k)$ and $\Delta'_{\max}=\max_k\DDprime(H_k)$. 
\begin{itemize}[leftmargin=.2cm,topsep=.3cm,itemsep=.2cm]
\item[] \uline{Joint upper tail}:
\nick{If} $n^{-1/\Delta'_{\max}}\ll p<1$, then
for any fixed $\delta_1,\dots,\delta_m>0$, 
\begin{equation}	\label{UT.LB}
\UT_{n,p}(\uH,\udelta) \ge 
(1-o(1)) \Phi_{n,p}(\uH, \udelta-o(1)) 
\end{equation}
and if 
$n^{-1/\Delta_{\max}}\ll p<1$, 
then 
\begin{equation}	\label{UT.UB}
\UT_{n,p}(\uH,\udelta) \le 
(1+o(1)) \Phi_{n,p}(\uH, \udelta+o(1)) \,.
\end{equation}

\item[] \uline{Joint lower tail}: 
If $n^{-1/\Delta'_{\max}}\log n\ll p< 1$, then for any fixed 
$\delta_1,\dots, \delta_m\in(0,1)$,
\begin{equation}	\label{LT.UBLB}
 \LT_{n,p}(\uH,\udelta) =
(1+o(1)) \Psi_{n,p}(\uH, \udelta+o(1)) \,.
\end{equation}
\end{itemize}
\end{thm}

\begin{remark}
\label{rmk:WLOG.p}
For the proofs of \eqref{UT.LB}--\eqref{UT.UB} we may assume $p\le\min_k (1+\delta_k)^{-1/\edges(H_k)}$, as otherwise the bounds hold vacuously.
In \Cref{sec:UTUB} we establish \eqref{UT.UB} by an alternative argument (similar to the one for the upper bound in \eqref{LT.UBLB}) which yields a wider range of $p$ for certain graphs; we refrain from pursuing the widest range of $p$ that can be obtained by our arguments under various assumptions on $H$.
We remark that in the graphs setting ($r=2$), where asymptotic formulas for $\Phi_{n,p}(\uH,\udelta)$ have been established \cite{BGLZ,BhDe}, the upper bound \eqref{UT.UB} is easily obtained by computing the probability of specific events that saturate the upper tail. However, in the general $r$-graph setting such formulas have only been obtained in a few cases (see \Cref{cor:LiZh}). 
\end{remark}

\begin{remark}\label{rmk:minmax.degree}
A similar result holds for mixed upper and lower tails, for $p\gg n^{-1/\Delta'_{\max}}\log n$.
However, when $p=o(1)$
the answer is just the lower tail problem \eqref{def:LT} for the sub-collection of the $H_k$ in the lower tail, together with any in the upper tail for which $\Delta(H_k)=1$. 
This is due to the different speeds for lower- versus upper-tail deviations (of order $n^rp$ versus $n^rp^{\Delta(H)}$).
For similar reasons, it turns out that for $p=o(1)$ the \abbr{rhs} in \eqref{UT.LB}--\eqref{UT.UB} is asymptotically equal to $\Phi_{n,p}(\pi(\uH),\pi(\udelta))$, where $\pi$ denotes restriction to the entries $k$ for which $\Delta(H_k) = \Delta_{\min}=\min_\ell\Delta(H_\ell)$ -- see \cite{BhDe} for the case $r=2$ (the idea is the same for general $r$) or the proof of \Cref{prop:lower}.
Consequently, the assumption for \eqref{UT.LB} can be relaxed to $1>p\gg n^{-1/\Delta'_\star}$ for $\Delta'_\star=\max\{ \Delta'(H_k): \Delta(H_k)=\Delta_{\min}\}$.
\end{remark}

\begin{remark}	\label{rmk:inhomogeneous}
By straightforward modifications of our arguments, \Cref{thm:tails} extends to random $r$-graphs drawn from inhomogeneous \ER measures $\mu_P$, provided $P(I)\in [cp,Cp]$ for all $I\in {[n]\choose r}$ and some fixed $0<c<C<\infty$. (One replaces $\eye_p(Q)$ with $\DKL(\mu_Q\|\mu_P)$ in \eqref{def:Phi}, \eqref{def:Psi}.)
\end{remark}

Various cases of \Cref{thm:tails} have been established before, mainly with $r=2$ and/or $m=1$, with some results holding in a wider range of $p$; we review the literature in $\mathsection$\ref{sec:previously} below.
We note that treating joint tail events ($m\ge2$) is important for applications to the analysis of general exponential random graph models, a class of Gibbs distributions on graphs that is widely applied in the social sciences literature -- see \cite{ChDi,ChDe,ElGr:ergm,CoDe:ergm}.

Our aim is more general than joint tail estimates: in this work we initiate a quantitative large deviations theory for random hypergraphs.
In particular,
in \Cref{thm:LDP} we establish versions of the approximation \eqref{LDP.top} for general sets $\cE$ at (large) fixed $n$, 
which amount to quantitative Large Deviation Principles (\abbr{LDP}s) for the \ER measure on $r$-graphs, extending the qualitative \ldp of Chatterjee and Varadhan \cite{ChVa} for the case $r=2$ and fixed $p$. 
The approximating sets $\cE'$ are defined under a new family of tensor norms $\|\cdot\|_\Basys^*$ that generalize the matrix cut norm. 
The main technical ingredient for establishing \Cref{thm:LDP} is a \emph{decomposition lemma} (\Cref{thm:reg}) for sparse tensors that generalizes the classic Frieze--Kannan decomposition for matrices \cite{FrKa1}. The role of the decomposition lemma is analogous to that of Szemer\'edi's regularity lemma in \cite{ChVa}.
Combining with an accompanying \emph{sparse counting lemma} (\Cref{thm:count}) -- a deterministic result establishing sharp Lipschitz control on the functionals $t(H,\cdot)$ under the $\Basys^*$-norms -- we obtain the upper- and lower-tail bounds for homomorphism counts as contractions of the general \abbr{LDP}s. 

We expect that our results could be applied or extended to other natural distributions on $r$-graphs. For instance, apart from inhomogeneous \ER $r$-graphs (see \Cref{rmk:inhomogeneous} above) one may apply the results of this work to  random regular hypergraphs, in a similar way to how large deviations results for the case $r=2$ from \cite{CoDe} were applied to random regular graphs in \cite{BhDe, Gunby}. 

The $\Basys^*$-norms are the main innovation of this work. (The work \cite{CoDe} relied on the spectral norm, which is unavailable for tensors.)
There are several novel features of these norms and the associated decomposition and counting lemmas. First, they are constructed to adapt to the level of sparsity under consideration. Second, as opposed to typical decomposition lemmas in extremal graph theory, our decomposition lemma attains a better quantitative bound that is crucial to obtain \Cref{thm:LDP}, by (necessarily) excluding an exceptional set whose probability can be made arbitrarily small. Finally, both our tensor norm and decomposition lemma make explicit use of the Boolean nature of the test tensors in order to obtain the nearly optimal quantitative bound -- in particular, in the case $r=2$ we improve on the result from \cite{CoDe} for counts of general graphs $H$. 

In extremal graph theory, the combination of decomposition lemmas (and closely related \emph{regularity lemmas}) with counting lemmas is known as a \emph{regularity method}, and our results for the $\Basys^*$-norms thus comprise a new regularity method for sparse hypergraphs, which we expect will have applications outside of large deviations theory.

Within the context of large deviations, the regularity method approach is quite flexible, and we demonstrate this with an application to the upper tail for \emph{induced} homomorphism counts in \Cref{thm:induced-main}. We further obtain strong results for the lower tail of counts of Sidorenko hypergraphs in \Cref{thm:LT}. 


In $\mathsection$\ref{sec:LDPreg} we review the connections between large deviations problems, graph limits and the regularity method, highlighting \nick{a special case of} one of our key technical results, the sparse counting lemma. 
In $\mathsection$\ref{sec:previously} we give an overview of previous works on upper and lower tails for random graphs, and in $\mathsection$\ref{sec:discussion} we discuss the potential scope and limitations of the regularity method approach to quantitative \abbr{ldp}s.

\subsection{Large deviation principles and the regularity method}
\label{sec:LDPreg}

On a conceptual level, the most important antecedent for our results \nick{is} the seminal work of Chatterjee and Varadhan establishing an \ldp for the \ER graph (the case $r=2$) \cite{ChVa}. Their result is a true \ldp in the classical sense, in that it establishes asymptotics of the form \eqref{LDP.top} for subsets $\cE$ of a fixed topological space $\cQ$, where $\cE'$ is an open/closed approximation of $\cE$. It is perhaps unclear how such an \ldp could be formulated in this context, as the \ER measures are on a \emph{sequence} of spaces $\cQ_{n,2}$ of growing dimension, but an appropriate setting is provided by the topological space of \emph{graphons}, which is in some sense the completion of the collection of all finite graphs of all sizes under a topology induced by the cut norm. This is the appropriate topology for studying homomorphism densities $t(H,\cdot)$, as these extend to continuous functionals on graphon space -- a consequence of the classic counting lemma. 
The key ingredient for the \ldp is the compactness of graphon space, which is a consequence of Szemer\'edi's regularity lemma (in fact the Frieze--Kannan \emph{weak regularity lemma} \cite{FrKa1} suffices for their purposes). 
Indeed, graphon theory gives a topological perspective on the classic regularity method in extremal graph theory, which is based on the regularity and counting lemmas.

We note that while large deviations theory was first formulated at the (in some sense ``correct'') level of generality of a topological theory by Varadhan in the 1960s \cite{Varadhan66}, the topological theory of dense graph limits was developed much more recently by Lov\'asz, Szegedy and coauthors \cite{LoSz06,BCLSSV,BCLSV1,BCLSV2}. We refer to the books \cite{Chatterjee:book,Lovasz:book} for further background on graph limits and the regularity method.

Unfortunately, graphon theory is largely unsuitable for the study of sparse graphs, such as \ER graphs with $p= n^{-c}$ for any positive constant $c>0$. 
In \cite{Chatterjee:survey}, Chatterjee poses the problem of developing a sparse graph limit theory that is powerful enough to prove upper-tail asymptotics for sparse \ER graphs.
While there are by now several sparse graph limit theories (see e.g.\ \cite{BCCL19} and references therein), we do not know of any that are generally suitable for the study of large deviations. 

The present work bypasses the development of an appropriate sparse (hyper)graph limit theory by instead developing a sparse hypergraph regularity method at finite $n$. As with sparse graph limits, existing sparse regularity tools are unsuitable for the study of \nick{upper-tail} large deviations (we review the literature in $\mathsection$\ref{sec:sparsereg}).
The main challenge in this context is \emph{localization phenomena}: that the underlying mechanisms for upper-tail deviations in the sparse setting are the appearance of dense configurations of $o(n^2)$ edges, which are invisible to the cut-norm topology. Such localized structures are a general problem for the development of sparse extensions of the regularity and counting lemmas, and hence for a sparse graph limit theory. We remark that such localization phenomena do not occur in the corresponding problem for lower-tail deviations. As such, asymptotics for extremes of the lower tail (the probability of containing no copy of a certain graph $H$) have been obtained previously in beautiful works around the K\L R conjecture and the hypergraph container method \cite{Luc, SaTo:containers, BMS:containers}; see also \Cref{rmk:KLR}. 

In place of the cut norm, we introduce a family of tensor norms designed to detect localization phenomena. 
\nick{A} key result is a (deterministic) sparse counting lemma giving optimal Lipschitz control on homomorphism counts 
\nick{of fixed $r$-graphs $H$ in large sparse $r$-graphs $G$,}
which we expect could be useful for other extremal problems where localization plays an important role.
We highlight here a 
\nick{special case of our sparse counting lemma for homomorphism counts of $K_4^{(3)}$, the complete $3$-graph on 4 vertices (thus $K_4^{(3)}$ contains all 4 possible edges of size 3).}
Recall that the symmetric adjacency tensor for an $r$-graph $G$ is denoted $A_G$. 
\nick{We say that $H'$ is a proper sub-hypergraph of $H$ if $\Verts(H')\subseteq \Verts(H)$ and $\Edges(H')$ is a strict subset of $\Edges(H)$.

\begin{thm}[Sparse counting lemma for $K_4^{(3)}$ counts]
\label{thm:count.K43}
Let $p\in (0,1)$ be arbitrary, and let $G_1,G_2$ be two $3$-graphs over the common vertex set $[n]$ such that
\begin{align}
&\max_{I,J,K\subseteq[n]^2}
\bigg| 
\sum_{(i,j,k)\in I^{(2,3)}\cap J^{(1,3)}\cap K^{(1,2)}} A_{G_1}(i,j,k)- A_{G_2}(i,j,k) 
 \bigg| 	\notag\\
&\le \eps p \cdot \bigg(n^3p^3+ n p^2 ( |I|+|J|+|K|)
+ 
|I^{(2,3)}\cap J^{(1,3)}\cap K^{(1,2)}| 
 \bigg)	\label{Bstar.K43}
\end{align}
for some $\eps\in (0,1]$, where for $I\subset[n]^2$ and $a,b\in \{1,2,3\}$ we write $I^{(a,b)}:=\{(i_1,i_2,i_3): (i_a,i_b)\in I\}\subseteq[n]^3$. 
Assume further that
\begin{equation}	\label{assu:count0.K43}
t(H,G_1)\le L p^{\edges(H')}
\end{equation}
for some $L\ge1$ and
all proper sub-hypergraphs $H'\subset K_4^{(3)}$. 
Then 
\begin{equation}	\label{conc:count0.K43}
|t(K_4^{(3)},G_1)-t(K_4^{(3)},G_2)| \ls \eps Lp^{4}.
\end{equation}
\end{thm}

The left hand side of \eqref{Bstar.K43} is the maximal edge discrepancy between $G_1$ and $G_2$ over sets of the special form $I^{(2,3)}\cap J^{(1,3)}\cap K^{(1,2)}$, which play an analogous role to the cut sets $I\times J\subseteq [n]^2$ in the definition of the cut norm for 2-graphs. For homomorphism densities of general $r$-graphs $H$ we consider edge discrepancies across structured sets with more general shapes, with carefully chosen $(n,p)$-dependent weights as on the right hand side of \eqref{Bstar.K43}, which will be crucial to get accurate control when $G_1,G_2$ are sparse. The shapes of structured sets and the weights are summarized by a \emph{weighted base system} $\BBasys$, which leads to the definition of a norm $\|\cdot\|_{\BBasys}^*$. In that setup, the bound \eqref{Bstar.K43} is equivalent up to constant factors to a bound of the form $\|A_{G_1}-A_{G_2}\|_{\BBasys}^*\le \eps p$ for certain base system $\BBasys$; see Examples \ref{ex:base.K43} and \ref{ex:base.K43b}.

The general sparse counting lemma of \Cref{thm:count} roughly states that for a given $r$-graph $H$, if $G_1,G_2$ are two (large) $r$-graphs such that
\begin{equation}	\label{assu:count0}
t(H',G_i) = O( p^{\edges(H')} )\,,\quad i=1,2
\end{equation}
for all proper sub-hypergraphs $H'\subset H$
(in particular $G_1,G_2$ are $O(p)$-sparse, by the case that $H'$ is a single edge)
and $\|A_{G_1}-A_{G_2}\|_{\BBasys}^*\le \eps p$ for an appropriate choice of weighted base system $\BBasys$ depending on $H$, then
\begin{equation}	\label{conc:count0}
|t(H,G_1)-t(H,G_2)| \ls \eps p^{\edges(H)}
\end{equation}
where the implicit constant depends only on $H$.}

The full definition of the $\Basys^*$-norms is a bit notationally involved (as is common in hypergraph regularity theory), so we first motivate them in $\mathsection$\ref{sec:sparsereg} with a special instance for matrices. 
The key point is that the \abbr{rhs} in \eqref{conc:count0} can be made small compared to the typical value $\sim p^{\edges(H)}$, even when $p=o(1)$ (the result is non-asymptotic so $p$ may depend in an arbitrary way on $n$).

Sparse counting lemmas for the cut norm have been a subject of intense study ever since a sparse extension of Szemer\'edi's regularity lemma was observed by Kohayakawa \cite{Kohayakawa} and R\"odl. 
Existing sparse counting lemmas establish \eqref{conc:count0} under different hypotheses, generally assuming $G_1$ and $G_2$ are both contained in a sparse pseudorandom ``host'' graph -- effectively ruling out localization phenomena, which are treated as a nuisance -- while only assuming they are close in the cut metric, which is sensitive to differences in edge counts only at a macroscopic scale (over a constant proportion of the vertices). 
While such versions are effective for obtaining sparse Ramsey/Tur\'an theorems or certain extreme cases of the lower tails,
they are unsuitable for our purposes of controlling upper tails, which are \emph{governed} by localization phenomena. Our assumption \eqref{assu:count0} is weaker than a pseudorandom host condition, while closeness 
under a $\Basys^*$-norm
is (necessarily) stronger, as these are designed to be sensitive to localization.
We discuss these points further in $\mathsection$\ref{sec:sparsereg}.

The accuracy of the $\Basys^*$-norms is only useful for large deviations if the space $\cQ_{n,r}$ is sufficiently compact under these norms, in a quantitative \nick{(metric entropy) sense. 
Indeed, a typical approach to derive large deviation upper bound is to first derive an upper bound on certain special sets, and combine them by constructing a covering of (most of) the space by these special sets. As encountered later in Theorem \ref{lem:upperLDP.convex}, by a straightforward consequence of the minimax theorem, one has a non-asymptotic large deviation upper bound 
\nick{taking the form of the right hand side of \eqref{LDP.top}} for the measure of convex sets $\cE'$.
Thus, 
\nick{one obtains large deviation upper bounds for more general sets $\cE$ by covering them with convex sets and applying the union bound, leading to}
an error term given by the metric entropy \nick{of the set $\cE$ (i.e.\ the logarithm of the covering number)}.
}

\nick{Suitable control on the metric entropy} is established by the decomposition lemma (\Cref{thm:reg}), which allows general sets $\cE$ to be efficiently covered by small balls centered on ``structured'' tensors. \Cref{thm:reg} generalizes the Frieze--Kannan decomposition for matrices, and crucially provides more efficient decompositions after the (optional) removal of a small set of exceptional tensors. 

\nick{As an example, in the context of $K_4^{(3)}$ counts as in \Cref{thm:count.K43} above, \Cref{thm:reg} implies that the set 
of all $3$-graphs $G$ over $[n]$ is covered by the $\eps p$-neighborhood (under the norm $\|\cdot\|_{\BBasys}^*$) of a small collection of ``structured'' weighted $3$-graphs,
together with a set $\cE$ of ``exceptional'' $3$-graphs of measure $\P(\bG_{n,p}\in \cE) \le p^{Ln^3p^3}$.
The structured weighted $r$-graphs have adjacency tensors $Q$ that are linear combinations of at most $O(L\eps^{-2}p^{-2})$ Boolean ``test tensors'' of the form $1_{(i,j,k)\in I^{(2,3)}\cap J^{(1,3)}\cap K^{(1,2)}}$, with notation as in \Cref{thm:count.K43}. The parameter $L\ge1$ is free to be chosen according to one's needs; note there is a tradeoff between the measure of the exceptional set $\cE$ and the size of the covering.

The decomposition lemma thus allows us to cover super- and sub-level sets for homomorphism densities $t(H,\cdot)$,  by a small number of sets of diameter $O(\eps p)$ in the appropriate $\Basys^*$-norm, together with an exceptional set whose measure can be made small compared to the large deviation rate. The counting lemma then shows that $t(H,\cdot)$ can only change by $O(\eps p^{\edges(H)})$ on these sets (recall that \eqref{Bstar.K43} is equivalent to such a bound) which allows us to justify the approximation of the upper and lower tails by \eqref{def:Phi} and \eqref{def:Psi}, respectively.}

\subsection{Previous works}
\label{sec:previously}

The past decade has seen several results of the form of \Cref{thm:tails} established for various ranges of sparsity $p$, mainly for the case $r=2$ (the \ER graph) and $m=1$, and often focusing only on the upper or lower tail. In the present work we aim for broader \abbr{LDP}-type statements as in \eqref{LDP.top}, which can only be expected to hold in a proper subset of the range of $p$ for which the asymptotics \eqref{UT.LB}--\eqref{LT.UBLB} are expected to hold -- we discuss this point further in $\mathsection$\ref{sec:discussion} below. 

Many works have obtained asymptotics for $\UT_{n,p}(H,\delta)$ and $\LT_{n,p}(H,\delta)$ holding up to constants depending on $\delta$. 
For the lower tail this is accomplished by Janson's inequality \cite{Janson:LT,JaWa:LT}.
For the ``infamous'' upper tail, following works \cite{KiVu,JOR} obtaining upper and lower bounds matching up to a factor $\log(1/p)$, the sharp dependence on $n$ and $p$ was obtained in a wide range of $p$ for triangles \cite{Chatterjee:triangles, DeKa:triangles}, general cliques \cite{DeKa:cliques}, cycles \cite{Raz:cycles}, and stars \cite{SiWa:stars}. 

Following the \ldp of \cite{ChVa} for fixed $p$ and $r=2$, 
the first results establishing sharp asymptotics for $\UT_{n,p}(H,\delta)$ allowing $p=n^{-c}$ took a rather different route from the one taken here, proceeding through a general study of Gibbs measures on the hypercube. This began with the influential work of Chatterjee and Dembo \cite{ChDe} introducing a new \emph{nonlinear large deviations} paradigm, further developed in \cite{Eldan, ElGr:decomp, Augeri, Austin:Gibbs, Yan:NLDT}, with the focus of establishing sufficient conditions for validity of the na\"ive mean-field approximation for the partition function, a problem of independent interest in statistical physics.  Large deviation estimates were deduced through (lossy) approximation arguments, and hence these works could only permit a small sparsity exponent $c$. 

The more direct approach to the large deviations problem  
through an appropriate finite-$n$ regularity method was 
introduced by the first two authors in \cite{CoDe} for the case $r=2$, where improved ranges for $p$ were obtained by replacing the cut norm with the spectral norm. 
The sparse counting lemma step in that work was only sharp for cycle counts, which is ultimately due to the fact that these can be expressed as moments of the spectral distribution of the adjacency matrix.
For the case of cycle counts, similar results (and superior in the case of triangles) were independently obtained by Augeri \cite{Augeri}.
The method was further applied to edge eigenvalues of the adjacency matrix in \cite{CoDe}, with a formula for the corresponding entropic optimization problem obtained in \cite{BhGa}; results on edge eigenvalues of sparser \ER graphs have more recently appeared in \cite{BBG,Basak:perron}. 

The lack of a spectral theory for tensors motivated the development of the $\Basys^*$-norm regularity method, which is the main technical contribution of this work. 

In \cite{HMS}, the upper tail asymptotic \eqref{UT.LB}--\eqref{UT.UB} was extended to an essentially sharp range of $p$ for the case ($r=2$, $m=1$, non-bipartite $\Delta$-regular $H$), with an optimal result for the bipartite case subsequently obtained in \cite{BaBa}. (These works consider counts of \emph{embeddings}, which only allow injective maps $\phi$ in \eqref{def:tQ}; while the difference is negligible in the ranges of $p$ considered here,  embedding counts have significantly different behavior from homomorphism counts when $p\ll n^{-1/\Delta}$.)
We comment further on the method of \cite{HMS, BaBa} in $\mathsection$\ref{sec:discussion} below. 
Very recently (after the first version of this paper appeared on arXiv) the same method was further developed to obtain upper-tail asymptotics for \emph{induced} homomorphism counts of $H=C_4$ in the \ER graph ($r=2$) in an essentially optimal range of $p$ \cite{Cohen22}. 

Upper tails for other random graph models besides the \ER distribution have been studied: $G(n,m)$ (uniformly random with $n$ vertices and $m$ edges) \cite{DeLu}, random regular graphs \cite{DhSe19, BhDe, Gunby}, and sparse inhomogenous \ER graphs, such as stochastic block models \cite{BhDe}. 
For the case of fixed $p$, extensions of the Chatterjee--Varadhan \ldp to inhomogeneous \ER graphs have been established in \cite{BCGPS,GrPi21, Markering}. 

There are only a few works considering hypergraphs with $r\ge3$.
The upper tail asymptotic 
\eqref{UT.LB}--\eqref{UT.UB} was established in \cite{LuZh:dense} for the case of $p$ fixed, $m=1$ and $H$ a \emph{linear} hypergraph (see \Cref{ex:linear} below), and more recently in \cite{LiZh} for general $H$ and $n^{-c(H)}\ll p\ll 1$ for sufficiently small $c(H)>0$, using general nonlinear large deviation tools from \cite{Eldan} (one checks their proof allows $c(H)=1/(6\edges(H))+o(1)$). 
Very recently, the lower tail asymptotic \eqref{LT.UBLB} for the case $m=1$ was established in an optimal range of sparsity in \cite{KoSa} by a beautiful entropy argument (as in \cite{HMS,BaBa} they consider embeddings rather than homomorphisms). 

There is a parallel line of works establishing asymptotic formulas for the entropic optimization problems \eqref{def:Phi}, \eqref{def:Psi}. For $r=2$, $m=1$ and fixed $p$ this was done in \cite{ChVa,LuZh:dense} for the upper tail in a certain region of the $(p,\delta)$-plane. 
The latter work extended \cite{ChVa} to counts of linear hypergraphs in dense \ER hypergraphs,
and further characterized the regime of $(p,\delta)$ for which the infimizer is the constant $Q\equiv p$ in this more general context. 
In \cite{MuBh} such a regime is provided for the variational problem corresponding to counts of general hypergraphs.
For $r=2$ and $1\gg p\gg n^{-1/\Delta}$ an asymptotic formula was obtained for the upper tail for all $\delta>0$ in \cite{LuZh:sparse} ($m=1$, $H$ a clique), \cite{BGLZ} ($m=1$, general $H$) and \cite{BhDe} (general $m$ and $H$). 
In \cite{Zhao:LT} Zhao obtains lower tail formulas with $p$ fixed or decaying as slowly as $n^{-c(H)}$ for a small $c(H)>0$, and certain ranges of $\delta$.
For general $r$, $m=1$ and $1\gg p\gg n^{-1/\Delta}$ an asymptotic formula for the upper tail is obtained in \cite{LiZh} for the case $H$ is a clique or the $3$-graph depicted in \Cref{fig:LiZh} -- see \Cref{cor:LiZh}.

Beyond establishing asymptotic formulas for (joint) upper and lower tails, there is the refined problem of describing the typical structure of $\bG$ \emph{conditioned} on the tail event. This has been addressed for $p$ fixed in some cases in \cite{ChVa,LuZh:dense}, and for the full range of $p=o(1)$ in \cite{HMS} for the upper tail with $r=2$, $m=1$ and $H$ a clique.
\nick{More recently (after this paper was posted to the arXiv) the first two authors established the conditional structure of \ER graphs conditional on general joint upper tail events as in \eqref{def:UT}, with $p=o(1)$ allowed to decay at a certain (suboptimal) polynomial rates, by combining large deviations results of \cite{CoDe} and the present work with a stability analysis for solutions of the entropic optimization problem \eqref{def:Phi} established in \cite{BGLZ,BhDe}. In \cite{CoDe:ergm} the results on  the conditional structure of \ER graphs were used to establish the \emph{typical} structure of sparse exponential random graph models.}
We mention also the line of works \cite{KRRS17a,KRRS17b,KRRS18,RRS} on the related problem of determining the typical structure of dense random graphs with constrained edge and $H$ counts for various choices of $H$.

\subsection{Discussion}
\label{sec:discussion}

Our focus in this work is on the development of quantitative \abbr{LDP}-type statements as in \eqref{LDP.top} applying to general subsets of $\cQ_{n,r}$ at large, fixed $n$, and to translate these to joint tail asymptotics as in \Cref{thm:tails} using a sparse counting lemma.
This approach has the advantage of being quite robust, applying to any functional enjoying a counting lemma under an appropriate $\Basys^*$-norm -- examples include non-monotone functionals such as \emph{induced} homomorphism counts, as well as non-polynomial functions such as the $\Basys^*$-norms themselves (or compositions of these with affine maps, such as centering), which can be viewed as weighted generalizations of the max-cut functional. 
The method also applies \emph{almost}\footnote{We say ``almost'' as there is a technical issue in applying the counting lemma for lower-tail estimates, stemming from the necessity of the crude upper bound \eqref{assu:count0} for counts of subgraphs. For upper tails this can be enforced by arguing inductively over the number of edges in $H$, so that we can restrict to the high-probability event that such a bound holds for all smaller graphs. However, for the lower tail there is the issue that the bad event that \eqref{assu:count0} fails is of upper-tail type and hence is much larger than the event we want to estimate. For the proof of \eqref{LT.UBLB} we get around this by using the FKG inequality to restrict to the event that \eqref{assu:count0} holds. This relies on monotonicity of homomorphism counts, which we do not have for induced homomorphism counts, and hence we do not have a result on the lower tail for the latter. We hope that an alternative argument for restriction to \eqref{assu:count0} can be found that avoids the use of monotonicity.} equally well to upper- and lower-tail events. 

However, \abbr{ldp}s applying to general sets $\cE$ can only be expected to hold in a limited range of sparsity. For instance, under the version of the $\Basys^*$-norms that is needed to analyze clique counts, for which $\Delta(H)={{\verts(H)-1}\choose {r-1}}$, the \ldp only yields \eqref{UT.LB}--\eqref{UT.UB} for $n^{-1/(\Delta(H)+1)}\ll p<1$, whereas the asymptotic should hold for all $p\gg n^{-r/\Delta(H)}$ (up to poly-logarithmic factors). 
For general $\uH$ we believe our method could be sharpened to yield the joint upper and lower tail asymptotics \eqref{UT.LB}--\eqref{LT.UBLB} in the range $1>p\gg n^{-1/\Delta_{\max}}$; this would follow in particular from relaxing the condition \eqref{def:WB} in the decomposition lemma by a factor $p$ (see also \Cref{rmk:pLB}). While we can push further than $n^{-1/\Delta_{\max}}$ for certain $\uH$ for which particularly efficient choices of $\Basys^*$-norm suffice for an accurate counting lemma, in general we believe the upper and lower tails for $t(H,\bG)$ should require very different arguments when $p\ll n^{-1/\Delta(H)}$. 

Thus, for certain sets $\cE$, arguments establishing \eqref{LDP.top} in the optimal range of $p$ will have to exploit special properties of $\cE$ once $p$ is below a certain threshold. For the case of super-level sets for counts of a fixed regular graph $H$ (the case 
of upper tails with $r=2$, $m=1$ and $\Delta$-regular $H$ in \Cref{thm:tails}), this has been accomplished in the optimal sparsity range by a beautiful truncated moment method argument developed in \cite{HMS} and further improved for the bipartite case in \cite{BaBa}. The general argument succeeds in covering the upper tail event by events on which the discrete gradient of the subgraph-counting functional is essentially supported on a small set of edges that they call a ``core'', reducing the problem to the (quite technical) task of counting the possible locations of cores, which they accomplish by exploiting special structure of subgraph-counting polynomials. 
In \cite{HMS} they also apply their general method to the upper tail of counts of $k$-term arithmetic progressions in sparse subsets of $[n]$. While the method is simplest for upper tails of polynomials with non-negative coefficients, it extends to certain non-monotone polynomials including induced subgraph counts (see \cite[Theorem 9.1]{HMS}, which is proved in the recent work \cite{Cohen22} treating the upper tail for induced $C_4$-counts). 

For lower tails, a beautiful entropic method was recently introduced in \cite{KoSa}, where they obtain the asymptotic \eqref{LT.UBLB} (for embedding counts rather than homomorphism counts) in the optimal sparsity range. This approach makes use of the monotonicity of sub-level sets for embedding counts. 

Finally, we note that whereas \eqref{UT.LB} and \eqref{LT.UBLB} are obtained by the sparse regularity method, we obtain the upper bound \eqref{UT.UB} for joint upper tails in the range $p\gg n^{-1/\Delta_{\max}}$ by applying a careful tilting
argument, using the Efron--Stein inequality to derive concentration for homomorphism
counts (as well as induced homomorphism counts) of a random tensor sampled from sparse product measures.
In $\mathsection$\ref{sec:UTUB} we give an alternative argument, more along the lines of the proof of \eqref{UT.LB} and \eqref{LT.UBLB} and holding in a different range of $p$, which may be better or worse depending on $\uH$.

\subsection{Organization}
In $\mathsection$\ref{sec:sparsereg} we discuss previous extensions of the regularity method for sparse graphs, introduce the $\Basys^*$ tensor norms (first in the matrix case), and state our general decomposition and counting lemmas.
$\mathsection$\ref{sec:LDPs} contains our general quantitative \ldps and some corollaries of \Cref{thm:tails} obtained by combining with earlier works on the upper-tail optimization problem $\Phi_{n,p}(H,\delta)$. 
$\mathsection\mathsection$\ref{sec:count.proof}--\ref{sec:LDP.proof} contain the proofs of the counting lemma (\Cref{thm:count}), decomposition lemma (\Cref{thm:reg}) and quantitative \ldps (\Cref{thm:LDP}). 
For the proof of \Cref{thm:tails}, we establish \eqref{UT.LB} in $\mathsection$\ref{sec:tails.upper}, \eqref{LT.UBLB} in $\mathsection$\ref{sec:LT}, and \eqref{UT.UB} in $\mathsection$\ref{sec:tails.lower}.
In $\mathsection$\ref{sec:other} we give extensions of \Cref{thm:tails} to induced homomorphism counts and the lower tail for counts of Sidorenko hypergraphs.

\subsection{Notational conventions}
\label{sec:notation}

We use $C,c,c'$, etc.\ to denote constants that may change from line to line, understood to be absolute if no dependence on parameters (such as $r$) is indicated.
For a (set of) parameter(s) $P$ we write $C(P)$ for a constant depending only on $P$. 

\subsubsection*{(Standard) asymptotic notation:}

For quantities $f,g$ depending on other parameters such as $n$ or $H$, we write $f=O(g)$, $f\ls g$ and $g\gs f$ to mean $|f|\le Cg$, and $f=\Theta(g)$ to mean $f\ls g\ls f$. We indicate dependence of the implied constant on parameters $P$ by writing e.g.\ $f=O_P(g), f\ls_Pg$. 
Notation $o(\cdot),\omega(\cdot), \gg, \ll$ is with respect to the limit $n\to \infty$, with $f=o(g)$, $g=\omega(f)$, $f\ll g$ and $g\gg f$ being synonymous to the statement $f/g\to 0$. 

\subsubsection*{Tensors:} 

Throughout we consider $r$ fixed independently of $n$. 
Denote by $\Tensnr$ the set of order-$r$ tensors of size $n$ (\emph{$r$-tensors}),
which we view as mappings $\tens:[n]^{r}\to\mathbb{R}$.
We equip $\Tensnr$ with the usual $\ell_p$ norms $\|Z\|_p^p = \sum_{i_1,\dots, i_r\in [n]} |Z(i_1,\dots, i_r)|^p$.
The Euclidean inner product on $\Tensnr$ for any $r$ (including $\Tens_{n,1} \cong \R^n$) is denoted $\langle\cdot,\cdot\rangle $.
The orthogonal projection to a subspace $W$ is denoted $P_W$. 
For a set $\cE\subset\Tensnr$ we write $\hull(\cE)$ for its convex hull.

An $r$-tensor is symmetric if it is invariant under permutation of the
 arguments. We write $\Symsnr\subset\Tensnr$ for the set of symmetric $r$-tensors
supported on entries with $r$ distinct coordinates, $\cAnr\subset\Symsnr$ for the subset of Boolean tensors, which are naturally associated to $r$-graphs, and
$\cQnr:=\hull(\cAnr)$, i.e.\ the set of $Q\in \Symsnr$ with all entries lying in $[0,1]$.
For $\syms\in \Symsnr$ we often abusively view its argument as an unordered set, writing $\syms(I):=\syms(i_1,\dots, i_r)$ for $I=\{i_1,\dots, i_r\}$.

Recall the distributions $\mu_Q$ introduced at the start of $\mathsection$\ref{sec:overview}, which we view as measures on $\cAnr$. We generally deal with random hypergraphs through their adjacency tensors in $\cAnr$.
Unless otherwise stated, $\P$ is a probability measure under which $\bA$ has distribution $\mu_p$, so that $\bA$ is the adjacency matrix for the \ER hypergraph $\bG=\bG_{n,p}^{(r)}$, and $\E$ is the associated expectation.
For $Q\in\cQnr$ we write $\P_Q,\E_Q$ for probability and expectation under which $\bA$ has the distribution $\mu_Q$. 
The relative entropy between the $\textrm{Bernoulli}(p)$ and the $\textrm{Bernoulli}(x)$ measures on $\{0,1\}$ is denoted
\begin{equation}	\label{def:eyep}
\eye_p(x) := \DKL(\mu_x\|\mu_p)= x\log\frac xp + (1-x)\log\frac{1-x}{1-p}\,, \qquad x\in [0,1]
\end{equation}
(extended continuously from $(0,1)$ to $[0,1]$).
With some abuse we use the same notation for the relative entropy of $\mu_Q$ with respect to $\mu_p$ on $\cAnr$, defining
\begin{equation}	\label{def:eyepQ}
\eye_p:\cQnr\to [0,\infty)\,,\qquad \eye_p(Q) = \sum_{1\le i_1<\cdots< i_r\le n} \eye_p(Q(i_1,\dots, i_r))\,.
\end{equation}

Note that $\E\bA= p\jay$ for the adjacency tensor $\jay$ of the complete $r$-graph on $n$ vertices.
That is, $\jay(i_1,\dots, i_r)=1$ if the indices are all distinct and zero otherwise.

\subsubsection*{Hypergraphs:}

All $r$-graphs are assumed to be finite and simple (i.e.\ with edge sets having no repeated elements). We often refer to $r$-uniform hypergraphs, sub-hypergraphs, and hyperedges simply as $r$-graphs, subgraphs, and edges, respectively.
For hypergraphs $H=(\Verts,\Edges)$ and $H'= (\Verts',\Edges')$, we 
 say $H'\subseteq H$ if $\Verts'\subseteq\Verts$ and $\Edges'\subseteq\Edges$, and $H'\subset H$ if $\Verts'\subseteq \Verts$ and $\Edges'\subset \Edges$.
We write $\verts(H):=|\Verts(H)|$, $\edges(H):=|\Edges(H)|$, and $\Delta(H)$ for
the maximum degree of $H$, that is, the maximum number of edges
sharing a common vertex $v\in \Verts(H)$. 
For $U\subset \Verts(H)$ we write 
\begin{equation}	\label{def:deg}
\bdy^H U  := \{ e\in \Edges(H): e\ne U, e\cap U\ne \emptyset\}, \qquad  \dd^H(U) := |\bdy^H U|
\end{equation}
for the edge boundary of $U$ and its cardinality, respectively (excluding $U$ itself when it is an edge). 
For $\dom\subset U\subseteq \Verts$ we denote the \emph{$\base$-dominated} boundary and degree of $U$ by
\begin{equation}	\label{def:deg.base}
\bdy_\dom^H U:= \{ e'\in \Edges(H): \emptyset\ne e'\cap U\subseteq \dom\}
= \bdy^HU \setminus \bdy^H(U\setminus \dom)
\,,\qquad
\dd_\dom^H( U) := |\bdy_\dom^H U|.
\end{equation}
This is a subset of $\bdy^HU$ consisting of edges whose overlap with $U$ is contained in $\base$.
We additionally set $\bdy_\emptyset^H U := \emptyset$, $\dd^H_\emptyset (U) := 0$.
We will usually drop the superscript $H$ from all notation, but in some places there will be more than one hypergraph in play and it will be necessary to clarify.

\subsubsection*{Homomorphism counts:}

As in several previous works on upper tails (e.g.\ \cite{ChVa,ChDe,BGLZ}) we count subgraphs in the sense of hypergraph homomorphisms.
Recall that a homomorphism between $r$-graphs $H$ and $G$ is a mapping $\phi:\Verts(H)\to \Verts(G)$ such that the image of every edge in $H$ is an edge in $G$. We do not require that $\phi$ be injective -- in particular, distinct edges of $H$ may be mapped to a common edge in $G$.
We write $\hom(H,G)$ for the number of homomorphisms from $H$ to $G$, so that $t(H,G)$ from \eqref{def:tQ} is $n^{-\verts(H)}\hom(H,G)$.
We extend this to a function on $\Symsnr$ as
\begin{equation}	\label{def:hom}
 \hom(H,\syms) :=\sum_{\phi:\Verts(H)\to [n]}\prod_{e\in \Edges(H)}\syms(\phi(e))
\end{equation}
so that for a graph $G$ over $[n]$ with adjacency tensor $A_G$ we have $\hom(H,G)=\hom(H,A_G)$.
Here, with slight abuse we interpret $\syms(\phi(e))$ for $e=\{v_{1},\dots,v_{r}\}$ as $\syms(\phi(v_{1}),\dots,\phi(v_{r}))$ when $\phi$ is injective on $e$ and 0 otherwise.
We additionally denote the normalized quantities
\begin{equation}	\label{def:ttp}
t(H,\syms) := \frac{\hom(H,\syms)}{n^{\verts(H)}}\,, \qquad t_p(H,\syms):= t(H,\syms/p) = \frac{\hom(H,\syms)}{n^{\verts(H)}p^{\edges(H)}}\,,
\end{equation}
often writing $t(H,G):=t(H,A_G)$ and $t_p(H,G):= t_p(H,A_G)$.\\

\section{Novel cut-type norms and a sparse regularity method}
\label{sec:sparsereg}

Our general approach reduces the problem of large deviations for nonlinear functionals of \ER hypergraphs to the development of a sparse hypergraph regularity method under appropriate extensions of the cut norm. 
This is a problem of general interest in extremal graph theory that goes beyond applications to large deviations, and
there is already a large body of literature on sparse extensions of the regularity method. In this section we begin with a brief overview of such results and explain why their assumptions make them unsuitable for our purposes. 
Then we discuss a special instance of the norms and decompositions in the matrix setting, in order to motivate the more complicated statements for general hypergraphs (deferred to $\mathsection\mathsection$\ref{sec:results.Bstar}--\ref{sec:results.reg}).

\subsection{Previous work on sparse regularity}
\label{sec:sparsereg.background}

Much of the literature on sparse regularity is with an eye towards sparse extensions of classical Tur\'an-type theorems. 
These show that sufficiently dense subsets $G$ of a large set $\Gamma$ are guaranteed to contain some small structure -- specifically, a set from a distinguished class $\cS \subset{\Gamma\choose k}$ of $k$-sets for some fixed $k$. For instance, if $\Gamma$ is the edge set of the complete graph $K_n$ on $n$-vertices, and $\cS$ is the set of ${r+1\choose 2}$-tuples of edges forming a copy of $K_{r+1}$, then Tur\'an's theorem guarantees that $G\subset \Gamma$ contains some element of $\cS$ when $|G|/|\Gamma|$ exceeds $1-\frac1r$ \cite{Turan}. A second example is Szemer\'edi's theorem \cite{Szemeredi}, where $\Gamma=[n]$, $\cS$ is the collection of $k$-term arithmetic progressions, and $G$ must contain some $S\in \cS$ if $|G|\ge\delta|\Gamma|$ for any fixed positive $\delta$ and $n$ sufficiently large. 

Sparse Tur\'an-type theorems establish the same statements when the ``host set'' $\Gamma$ is instead taken to be a sparse pseudorandom subset of the host set $\Gamma_0$ from the corresponding classical theorem.
An example is the Green--Tao theorem establishing existence of arithmetic progressions of arbitrary length in the primes, which proceeded through a ``relative Szemer\'edi theorem'' for a certain set $\Gamma$ of almost-primes that is a sparse pseudorandom subset of $\Gamma_0=[n]$ \cite{GrTa}. 
In the realm of graph theory, Tur\'an's theorem (and more generally, the Erd\H os--Stone theorem and Simonovits's stability theorem) has been transferred to host graphs $\Gamma$ such as sparse \ER graphs \cite{CoGo,Schacht} (see \Cref{rmk:KLR} below for a discussion of related results) and graphs satisfying certain pseudorandomness conditions \cite{CFZ:sparsereg}.

These results can be proved by mimicking proofs of corresponding results for the dense setting, for instance via sparse versions of the hypergraph removal lemma, which in turn are obtained from sparse extensions of hypergraph regularity and counting lemmas. 
Let us briefly recall these in the graph setting (2-uniform hypergraphs).
Recall the normalized matrix cut norm 
\begin{equation}	\label{def:cutnorm.matrix}
\|M\|_\Box = \frac1{n^2}\max_{I,J\subseteq[n]} |\langle M, \1_I\otimes \1_J\rangle|\,,\qquad M\in \R^{n\times n}.
\end{equation}
Here, $\1_I\otimes \1_J$ is the rank-1 matrix $\1_{I}\1_J^\tran$, and $\langle\cdot,\cdot\rangle$ is the Euclidean (Hilbert--Schmidt) inner product on $\R^{n\times n}$.
This extends to a metric $d_\Box$ on the set $\cG_n$ of graphs over the vertex set $[n]$ as $d_\Box(G_1,G_2) = \|A_{G_1}-A_{G_2}\|_\Box$, with $A_{G_i}$ the adjacency matrix for $G_i$.
Thus, we trivially have $d_\Box(G_1,G_2)\le 1$, and a bound $d_\Box(G_1,G_2)\le\eps<1$ provides uniform control on the discrepancy between $G_1,G_2$ of edge counts in vertex subsets of $[n]$ of \emph{macroscopic size}, i.e.\ linear in $n$.

A result of Frieze and Kannan \cite{FrKa1} (from which their weak regularity lemma is easily deduced) states that for any graph $G$, there is a decomposition of its adjacency matrix as 
\begin{equation}	\label{FK:decomp}
A_G = A_{\str} + A_{\rand}
\end{equation}
where the \emph{structured} piece $A_{\str}$ is a linear combination of $O(1/\eps^2)$ \emph{cut matrices} $\1_{I_k}\otimes \1_{J_k}$, and the \emph{pseudorandom} piece $A_{\rand}$ satisfies $\|A_{\rand}\|_\Box \le \eps$. 
The cut-norm counting lemma says that the homomorphism density functionals $t(H,\cdot)$ (recall \eqref{def:ttp})
are $O_H(1)$-Lipschitz in the cut metric. 
For graphs $G_1,G_2$ of density $p=o(1)$ we trivially have $d_\Box(G_1,G_2) \ls p$, so for a sparse regularity lemma we seek a decomposition as in \eqref{FK:decomp} with $\|A_{\rand}\|_\Box \le \eps p$. 
In general such a decomposition requires a growing number of cut matrices, but straightforward modifications of the Frieze--Kannan argument yield sparse decomposition lemmas with $O_\eps(1)$ cuts under additional ``no dense spots'' assumptions on $G$ \cite{CCF:sparsereg}.

One may similarly hope for a sparse counting lemma saying that $|t_p(H,G_1)-t_p(H,G_2)|\ls_H \eps$ (recall the notation \eqref{def:ttp}) if $d_\Box(G_1,G_2) \le \eps p$, but this too is false without additional assumptions: consider for instance the case that $G_1$ and $G_2$ agree on all edges outside a set of vertices $V_0$ of size $\Theta(np)$, where $G_1$ is empty and $G_2$ is full. Since these only differ on $O(n^2p^2)$ edges we have $d_\Box(G_1,G_2) \ls p^2=o(p)$, whereas for triangle counts (say) we have $|t_p(K_3,G_1)-t_p(K_3,G_2)|\gs 1$.
However, in the applications to sparse Tur\'an-type theorems described above, $G_1$ and $G_2$ are both contained in a pseudorandom (or truly random) host graph $\Gamma$, and sparse counting lemmas have been established under various ``(pseudo)random container'' assumptions \cite{GMS:prob-counting,CGSS:KLR,CFZ:sparsereg,CFZ:relative,ABSS}.

\subsection{A modified cut norm for sparse graphs}
\label{sec:Bstar.motivation}

Unfortunately, none of the sparse regularity or counting lemmas just described are useful for us, as the ``no dense spots'' and ``pseudorandom container'' hypotheses rule out the localization phenomena we are trying to detect. 
In \cite{LuZh:sparse,BGLZ} two phenomena are identified as the dominant mechanisms for upper-tail deviations of $t_p(H,\bG)$ in the \ER graph for $n^{-1/\Delta(H)}\ll p\ll 1$: 
the appearance of an almost-clique (of density close to 1) on $\Theta(np^{\Delta(H)/2})$ vertices,\footnote{In fact the almost-clique mechanism only contributes to large deviations when $H$ is a regular graph.}  or of an almost-complete bipartite graph on $J\times [n]\setminus J$ for $|J|=\Theta(n p^{\Delta(H)})$. 
Both types of subgraphs contain $\Theta(n^2p^{\Delta})=o(n^2)$ edges when $p=o(1)$ and are hence invisible to the cut norm; moreover, the cuts that correlate with $\bG$ on these events, namely $\1_{I_0}\otimes\1_{I_0}$ and $\1_{J_0}\otimes \1_{[n]\setminus J_0}$, have factors occurring at three separate scales.

Further localization phenomena have been described in the setting of regular graphs \cite{BhDe,Gunby}, and the possibilities are more numerous in the hypergraph setting \cite{LiZh}.

Our approach is to develop generalizations $\|\cdot\|_{\Basys}^*$ of the cut norm that are sensitive to localization phenomena at all scales. In the general hypergraph setting this is done in terms of a (user-specified) set system $\Base$ over $[r]$, and the class of cut matrices is replaced by a class of \emph{test tensors} $\test$ that are entrywise product of tensors $\ttest_\base,\base\in \Base$ varying only on coordinates in $\base$. 
We defer the general definitions to $\mathsection$\ref{sec:results.Bstar} and discuss here a particular instance of these norms in the case of $2$-graphs.
\nick{(See also \Cref{thm:count.K43} and the discussion that follows it for an example for $3$-graphs, stated there in terms of sets of edges rather than the functional formulation given here.)}

Denote by $\Test=\Test_n$ the class of Bernoulli cut matrices $T=\1_I\otimes \1_J$ with $I,J\subseteq[n]$.
Given a graph $H=(\Verts,\Edges)$ of maximum degree $\Delta$, we set a cutoff scale $n_0:=np^{\Delta-1}$ and for $T=\1_I\otimes \1_J\in \Test$ denote
\begin{equation}	\label{def:B.mat}
\|T\|_{\Delta,2} = (|I|\vee n_0)(|J|\vee n_0).
\end{equation}
(This can be extended to a norm on $\R^{n\times n}$ but we only apply it to cut matrices.)
Now for $M\in \R^{n\times n}$ let
\begin{equation}	\label{def:Bstar.mat}
\|M\|_{\Delta,2}^* = \sup_{T\in \Test} \frac{|\langle M, T\rangle |}{\|T\|_{\Delta,2}} = \max_{I,J\subseteq[n]}\frac{|\1_I^\tran M\1_J|}{(|I|\vee n_0)(|J|\vee n_0)}.
\end{equation}
Note that $\|\cdot\|_{\Delta,2}$ and $\|\cdot\|_{\Delta,2}^*$ depend on $p$, but we suppress this from the notation.
The $\|\cdot\|_{\Delta,2}^*$ norm specializes to the normalized cut norm \eqref{def:cutnorm.matrix} upon taking $p=1$, but for smaller $p$ the $\|\cdot\|_{\Delta,2}^*$ norm is sensitive to changes in density at smaller scales. 
The counting lemma of \Cref{thm:count} specializes to this setting to say that if 2-graphs $G_1,G_2$ over $[n]$ satisfy $\|A_{G_1}-A_{G_2}\|_{\Delta,2}^*\le \eps p$ and $t_p(H',G_1)\le L$ for every proper subgraph $H'\subset G_1$, then $|t_p(H,G_1)-t_p(H,G_2)|\ls_H L\eps$. 
We note that the full version \Cref{thm:count} generalizes this to multilinear homomorphism functionals, and \Cref{thm:count.gen} further extends to signed homomorphisms, which includes induced homomorphisms. 

Note that some form of cutoff scale \nick{$n_0$} is necessary, as otherwise the norm would be too sensitive to changes on single entries. The specific choice $np^{\Delta-1}$ is motivated by the proof of the counting lemma, where it is a critical threshold for the influence of an endpoint of a single edge of $G$ on $t(H,G)$. 
\nick{Indeed, by 
a telescoping decomposition based on the edges of $H$, one can 
express $t_p(H,G_1)-t_p(H,G_2)$ as a sum of terms indexed by edges $e = \{u,v\} \in \Edges(H)$ and embeddings $\psi:\Verts(H)\setminus \{u,v\} \to [n]$. Each term in the sum can be expressed in the form $\langle A_{G_1}-A_{G_2}, \1_I\otimes \1_J\rangle$ where $I$ and $J$ are the common neighbors of $\psi(\partial^H(u)))$ and $\psi(\partial^H(v))$ (recall our notation \eqref{def:deg}). 
In a 
random graph, $|I|$ and $|J|$ are typically of order $np^{\deg_H(u)-1}$ and $np^{\deg_H(v)-1}$ respectively, which are 
at least $np^{\Delta-1}$. This is the motivation for the cutoff $n_0 = np^{\Delta-1}$ in the definition of $\|\cdot \|_{\Delta,2}$.}

The following is a special case of our tensor decomposition lemma (\Cref{thm:reg}). Recall that $\bA=\bA_{n,p}^{(2)}$ is the adjacency matrix for the \ER graph.

\begin{thm}[Decomposition lemma, special case]		\label{thm:reg.2graph}
There exist absolute constants $C_0,c_0>0$ such that the following holds. 
Let $\kappa,\eps>0$ and assume $n$ and $p\in (n^{-2}, 1)$ are such that
\begin{equation}	
np^{\Delta+1}
\ge\frac{C_0\log n}{\eps^2\log(1/p)}\,. \label{assump-n}
\end{equation}
Then there exists a (possibly empty) exceptional set $\badset(\kappa,\eps)\subseteq\cA_{n,2}$ with
$
\mathbb{P}(\bA\in\badset(\kappa,\eps))\le p^{c_0\kappa n^2}
$
such that 
for each $A\in \cA_{n,2}\setminus \badset(\kappa,\eps)$ 
there is a decomposition
\begin{equation}	\label{decomposition.mat}
A = A_{\str}+A_{\rand} 
\end{equation}
where 
\begin{equation*}	\label{Arand.control.mat}
A_{\str}=pJ_{n,2} + 
\sum_{i=1}^k \alpha_i T_i \qquad\text{ and }\qquad \; \|A_{\rand}\|_{\Delta,2}^* \le \eps p
\end{equation*}
for real numbers $\alpha_1,\dots, \alpha_k$, and cut matrices $\test_1,\dots, \test_k$
such that
\begin{equation}
\sum_{i=1}^k \|\test_i\|_{\Delta,2} \le \kappa \nick{\eps}^{-2} p^{-2} n^2\,. \label{complex}
\end{equation}
\end{thm}

\begin{remark}
\Cref{thm:reg.2graph} contains the Frieze--Kannan decomposition lemma \eqref{FK:decomp} as a special case:  taking $p=1/2$ (say) makes $n^2\|\cdot\|_{\Delta,2}^*$ equivalent to the cut norm, and then taking $\kappa$ to be a sufficiently large absolute constant makes $\cE_\star(\kappa,\eps)=\emptyset$; from the lower bound $\|T_i\|_{\Delta,2}\ge n_0^2 \gs n^2$ it follows that $k=O(1/\eps^2)$. 
However, the option to remove the exceptional set of tensors $\cE_\star(\kappa,\eps)$ is important for our application to upper tails, as taking a smaller value of $\kappa$ reduces the complexity of the approximation $A_{\str}$, effectively reducing the dimension of the space of tensors. We will set $\kappa$ just large enough that $-\log\P(\cE_\star(\kappa,\eps))$ is above the large deviation rate (for instance, for the upper tail of $t(H,\bG)$ this is $\kappa=Kp^{\Delta(H)}$ for a sufficiently large constant $K$).
A further key difference from the Frieze--Kannan decomposition lemma is that in \eqref{complex}
the complexity is measured in terms of the total size of the cut matrices.
\end{remark}

\begin{remark}	\label{rmk:pLB}
We believe that the right hand side of \eqref{complex} can be improved to $\kappa \epsilon^{-2}p^{-1}n^2$, which would allow us to replace the left hand side of \eqref{assump-n} with $np^{\Delta}$. An assumption of $np^\Delta\gg1$ would be essentially optimal: when $np^\Delta\ls1$,  the row-sums in submatrices of size $np^{\Delta-1}\times np^{\Delta-1}$ (the smallest scale controlled by the norm $\|\cdot\|_{\Delta,2}$) cease to concentrate, and  we can no longer have uniform control on densities in such submatrices holding with high probability. 
\nick{Relaxing the left hand side in \eqref{assump-n} to $np^\Delta$, and more generally saving a factor $p$ in the analogous assumption \eqref{def:WB} for the general decomposition lemma, would immediately 
\nick{imply that \Cref{thm:tails} holds with $\Delta'_{\max}$ replaced by $\Delta_{\max}$ in all cases.}
For more precise details on the improvement in the case $r=2$ as well as in the general case, we refer to Remark \ref{rmk:pLB'}.}
\end{remark}

In the standard way one can deduce a weak regularity lemma-type statement in terms of the partition of $[n]$ generated by the factors of the test tensors $T_i$, but this is not needed for our applications. 

\Cref{thm:reg.2graph} takes the typical form of a decomposition lemma from graph theory and additive combinatorics, in that the summands in the expansion of the structured piece are controlled in some norm $\|\cdot\|$, while the pseudorandom piece is small in the dual norm $\|\cdot\|^*$, a general perspective that was explored by Gowers in \cite{Gowers:decomp}. 
Another common form of decomposition lemma obtains finer control on the pseudorandom piece, making $\|A_{\rand}\|^*$ small relative to the ``complexity'' $k$ of the structured piece, by separating out a further piece $A_{\text{small}}$ that is small in another norm such as $\ell_2$ (the original regularity lemma of Szemer\'edi is of this type).
This comes at the cost of a much larger value of $k$ than in weak regularity lemmas, and we will not need such control.

Before moving on to the general definition of the $\Basys^*$-norms for $r$-tensors, let us mention one other case of the norms when $r=2$, which is to take the smaller class $\Test=\{ \1_I\otimes \1: I\subseteq[n]\}$, with
\begin{equation}	\label{def:B.mat1}
\|T\|_{\Delta,1} = n\cdot(|I|\vee n_0)\quad \text{for } \quad T=\1_I\otimes \1
\end{equation}
and
\begin{equation}	\label{def:Bstar.mat1}
\|M\|_{\Delta,1}^* = \sup_{T\in \Test} \frac{|\langle M, T\rangle |}{\|T\|_{\Delta,1}} = \max_{I\subseteq[n]}\frac{|\1_I^\tran M\1|}{n(|I|\vee n_0)}.
\end{equation}
Alternatively, letting $d_M(i) :=\frac1n \sum_{j=1}^n M(i,j)$ be the vector of normalized row sums for $M$, 
\[
\|M\|_{\Delta,1}^* = \max_{I\subseteq[n]} \frac{\langle d_M, \1_I\rangle}{|I|\vee n_0}\,.
\]
This norm turns out to be effective for studying star homomorphism densities $t(K_{1,\Delta},G)$, which is perhaps unsurprising since star homomorphism counts are determined by the degree sequence of $G$ -- indeed, we have  $t(K_{1,\Delta},G) = \frac1n\|d_{A_G}\|_{\ell_\Delta}^\Delta$. In this case our general decomposition lemma is approximating the degree sequence $d_{A_G}$ of a graph by a short weighted combination of indicators $\1_{I_k}$, which can be done more efficiently than approximating the whole matrix $A_G$ in the norm $\|\cdot\|_{\Delta,2}^*$. As a result, \Cref{thm:tails} gives tail asymptotics for star homomorphism counts in a wider range of $p$ than for, say, clique counts.
The form of \eqref{def:Bstar.mat1} as compared to \eqref{def:Bstar.mat} illustrates a key point of the $\Basys^*$-norms defined below: that to have an effective counting lemma for $H$-counts, we only need to use test tensors that are non-constant on coordinates corresponding to vertices where edges overlap. 

\begin{remark}	\label{rmk:KLR}
It would be remiss to not mention work on the K{\L}R conjecture on an embedding lemma for subgraphs of the sparse \ER graph $\bG_{n,p}^{(2)}$  \cite{KLR97}, which was ultimately proved using the hypergraph container method in \cite{SaTo:containers,BMS:containers}. 
An embedding lemma is weaker than a counting lemma, only providing the existence of at least one appearance of some subgraph $H$, whereas a counting lemma provides roughly the expected number of copies based on edge densities between parts of a vertex partition. 
Stronger ``probabilistic counting lemmas'' for $\bG_{n,p}^{(2)}$ were obtained in \cite{GMS:prob-counting, CGSS:KLR}, motivated in particular by Tur\'an-type theorems (as discussed in $\mathsection$\ref{sec:sparsereg.background}).
We note an interesting contrast: whereas these works make use of the deterministic Kohayakawa--R\"odl sparse regularity lemma (with a no-dense-spots condition) and establish a counting lemma that holds for dense subgraphs of $\bG_{n,p}^{(2)}$ with high probability, here we combine a deterministic counting lemma with a decomposition lemma holding (with acceptable complexity for our application) with high probability.
\end{remark}

\subsection{The $\Basys^*$-norms}
\label{sec:results.Bstar}

There are several natural generalizations of the cut norm to $r$-tensors. 
For instance, one can take cut tensors of the form $\1_{I_1}\otimes\cdots\otimes \1_{I_r}$ for $I_1,\dots, I_r\subset[n]$, as was done for the sparse Frieze--Kannan decomposition proved in \cite{CCF:sparsereg}. However, it turns out that the resulting norm is only useful for controlling homomorphism densities for linear hypergraphs $H$ (see \cite{LuZh:dense}). 

Instead, we will consider the wider class of Boolean tensors formed by entrywise products of tensors varying on strict subsets of the $r$ coordinates. 
We have the following:

\begin{defn}[Weighted base]	\label{def:Base}
Given a set $e$ of size $r$, a \emph{base (over $e$)} is a collection $\Base$ of proper subsets of $e$ with $\emptyset\in \Base$ and such that any two nonempty elements $\base_1,\base_2\in \Base$ are incomparable, i.e.\ $\base_1\not\subset\base_2$.
A \emph{weighted base} over $e$ is a tuple $\Basys=(r,e,\iota,\Base,\ddstar,\{\dd_\base\}_{\base\in\Base})$, with
\begin{itemize}
\item $\iota:[r]\to e$ a bijective mapping;
\item $\Base$ a base over $e$;
\item non-negative integer weights $\dd_\star$ and $\dd_\base$ satisfying $\dd_\base\le \ddstar$ for each $\base\in \Base$, and $\dd_\emptyset:=0$. 
\end{itemize}
A \emph{base system} over an $r$-graph $H$ is a collection $\BBasys=\{\Basys(e)\}_{e\in \Edges(H)}$ with each $\Basys(e)$ a weighted base over $e$. 
\end{defn}

In our applications to $H$-counts, the choice of integer weights $\dd_\star,\dd_\base$ will generally be determined by the degrees of an edge and its subsets.

For a weighted base $\Basys$ we define an
associated set of \emph{test tensors} $\Test_\Basys\subset\Tensnr$ 
consisting of all nonzero Boolean tensors $\test:[n]^r\to \{0,1\}$ of the form
\begin{equation}	\label{test.base}
\test(i_1,\dots,i_r) 
=\prod_{\base\in\Base}\ttest_\base\circ\pi_\base (i_1,\dots,i_r)\,,\qquad i_1,\dots,i_r\in[n]
\end{equation}
for general Boolean functions $\ttest_\base:[n]^\base\to \{0,1\}$, where we denote by
$
\pi_\base : [n]^r\to [n]^\base
$ 
the natural projections
$
(i_1,\dots, i_r) \mapsto (i_v)_{v\in \iota^{-1}(\base)}\,.
$
We always take $\ttest_\emptyset$ to be the constant tensor $\ttest_\emptyset(i_1,\dots, i_r)\equiv 1$.
We quantify the size of the factors of a test tensor as in \eqref{test.base} with the rescaled $\ell_1([n]^\base)$-norm:
\begin{equation}	\label{norm.factors}
\|\test\|_{\base}:= n^{r-|\base|}p^{\ddstar-\dd_\base} \|\ttest_\base\|_1
\end{equation}
(in particular $\|\test\|_{\emptyset}= n^rp^{\ddstar}$),
and define
\begin{equation}	\label{def:size}
\|\test\|_{\Basys} :=\max \big\{ \, \|\test\|_{1}\,,\;\max_{\base\in\Base}\|\test\|_{\base}\,\big\}.
\end{equation}
(This can be extended to a norm on $\Tensnr$ but we only apply it to test tensors.) Note that the inclusion of $\emptyset\in \Base$ means we always have $\|\test\|_{\Basys} \ge n^rp^{\ddstar}$. 
We define a
seminorm $\|\cdot\|_{\Basys}^*$ on $\Tensnr$ via duality:
\begin{equation}	\label{def:dualnorm}
\|\tens\|_{\Basys}^* :=\max_{\test\in\Test_\Basys}\frac{\left|\langle \tens,\test\rangle \right|}{\|\test\|_{\Basys}}.
\end{equation}
When $\Base$ covers $[r]$ this defines a genuine norm on $\Tensnr$, but we do not enforce this in general.
We note that these seminorms additionally depend on all components of the weighted base $\Basys$ (not just the base $\Base$) as well as $p$, but we suppress this dependence from the notation. 
If $\BBasys=\{\Basys(e)\}_{e\in\Edges}$ is a finite collection of weighted bases $\Basys(e)$ over the respective $r$-sets $e$ (such as a base system  over an $r$-graph $H$, with $\Edges=\Edges(H)$) we denote the seminorm
\begin{equation}	\label{def:BBnorm}
\|\tens\|_\BBasys^*:= \max_{e\in\Edges} \|\tens\|_{\Basys(e)}^*.
\end{equation}

\begin{example}
\label{ex:base.2}
For the case $r=2$, we recover the matrix norm \eqref{def:B.mat} by taking the maximal base $\Base = \{ \emptyset, \{1\}, \{2\}\}$ over $\{1,2\}$, $\ddstar=2\Delta-2$ and $\dd_{\{1\}}=\dd_{\{2\}}=\Delta-1$, and this further specializes to the normalized cut norm $\frac1{n^2}\|\cdot\|_{\Box}$ upon taking $p=1$. 

\end{example}

\nick{For general $r\ge2$ it is useful to consider the maximal base $\Base = {[r]\choose r-1}\cup\{\emptyset\}$ with appropriate degree parameters. 
For instance:}

\begin{example}
\label{ex:base.K43}
\nick{In the $K_4^{(3)}$-counting lemma presented in \Cref{thm:count.K43}, the bound \eqref{Bstar.K43} is equivalent up to constant factors to the bound $\|A_{G_1}-A_{G_2}\|_{\Basys}^*\le \eps p$, where for the weighted base $\Basys$ we take $e=\{1,2,3\}$, $\iota$ the identity, $\Base=\{\emptyset, \{1,2\},\{2,3\}, \{1,3\}\}$, $\ddstar=3$ and $\dd_{\{1,2\}}= \dd_{\{2,3\}}=\dd_{\{1,3\}}=1$. 
Indeed, for this choice of base a test tensor takes the form $\1_{I^{(2,3)}\cap J^{(1,3)}\cap K^{(1,2)}}$ for sets $I,J,K\subset[n]^2$.
(The equivalence is up to a constant because we take a sum on the right hand side in \eqref{Bstar.K43} rather than a maximum as in \eqref{def:size}.)}
\end{example}

For general $r$ and $\Base$ and with $p=1$ we have $\|\test\|_{\Basys}=n^r$ and we recover the family of generalized cut norms considered by Conlon and Lee in \cite{CoLe} (various special cases of which had been considered earlier, such as the case $\Base = {[r]\choose r-1}\cup\{\emptyset\}$  by Gowers in  \cite{Gowers:3graphs, Gowers:hypergraphs}). It was shown in \cite{CoLe} for the unweighted setting ($p=1$) that $\|\cdot\|_{\Basys}^*$ is polynomially equivalent to certain generalized Gowers norms, though the polynomial loss appears to make the latter ineffective in the sparse setting (this is related to how various definitions of quasirandomness for dense graphs cease to be equivalent for sparse graphs). 

For our application to $H$-homomorphism counts (through the counting lemma, \Cref{thm:count} below), for each edge $e\in\Edges(H)$  we will select a weighted base over $e$ satisfying the following property:

\begin{defn}[Dominating base]	\label{def:dom}
For an $r$-graph $H$ and $e\in \Edges(H)$, a base $\Base$ over $e$ is \emph{$H$-dominating} if every edge overlap $e\cap e'$ with $e'\ne e$ is contained in some $\base\in\Base$. A weighted base is $H$-dominating if its base is $H$-dominating.
A base system $\BBasys=\{\Basys(e)\}_{e\in\Edges(H)}$ over $H$ is $H$-dominating if each of the weighted bases $\Basys(e)$ is $H$-dominating.
\end{defn}

\nick{Observe that if the weighted bases $\Basys=(r,e,\iota,\Base,d_*,\{d_\base\}_{\base \in \Base})$ and $\Basys'=(r,e,\iota,\Base',d_*,\{d_\base'\}_{\base' \in \Base'})$ over bases $\Base$ and $\Base'$ respectively are such that $d_* = d_*'$, any $\base' \in \Base'$ is contained in $\Base$ and $d_{\base'}=d_{\base}$, then $\|T\|_{\Basys}\ge \|T\|_{\Basys'}$. Thus $\|Z\|_{\Basys}^* \le \|Z\|_{\Basys'}^*$.}

With a fixed choice of dominating base $\Base(e)$ for each $e\in\Edges(H)$, and some arbitrary choice of bijections $\iota_e:[r]\to e$, we take the $H$-dominating base system $\BBasys=\{\Basys(e)\}_{e\in\Edges(H)}$ with 
\begin{equation}	\label{Deg.usual}
\Basys(e)= \big(r,e,\iota_e, \Base(e), \dd^H(e), \{\dd_\base^H(e)\}_{\base\in\Base(e)}\big)
\end{equation}
and with degree parameters as defined in \eqref{def:deg}--\eqref{def:deg.base}. The choice of integer weights is motivated by the proof of the counting lemma, where the prefacfors in \eqref{norm.factors} combined with the hypothesis of a crude upper bound on counts of subgraphs of $H$ (as in \eqref{assu:count0}) will allow us to close an induction on the number of edges in $H$. 
(We only depart from 
the choice of weights \eqref{Deg.usual} in the \emph{proof} of \Cref{thm:count}, via the generalization \Cref{thm:count.gen}, where we take $H$-dominating bases with weights $\dd^{H_+}(e),\dd_\base^{H_+}(e)$ taken according to a subgraph $H_+$ of $H$, but in the statement of the theorem, and hence in all of its applications, we take weights as in \eqref{Deg.usual}.)

\nick{
\begin{example}\label{ex:base.K43b}
Continuing the example of \Cref{thm:count.K43}, the base of \Cref{ex:base.K43} can be made into a base system $\BBasys$ over the edges of $K_4^{(3)}$ by taking the weighted base $\Basys$ for each edge (identifying each with $\{1,2,3\}$), and one verifies the weights $\ddstar,\dd_\base$ given in \Cref{ex:base.K43} are chosen as in \eqref{Deg.usual}. 
\end{example}
}

\nick{While the maximal} base $\Base=\{\emptyset\}\cup {e\choose r-1}$ 
\nick{(as in Examples \ref{ex:base.2} and \ref{ex:base.K43})} is always dominating, 
\nick{for certain $H$ one can use bases with a smaller number of smaller sets, which leads to better quantitative estimates.}
For instance\nick{:} 

\begin{example}
\nick{W}hen $H$ is a \emph{sunflower}, where pairwise overlaps of all edges are equal to a common \emph{kernel} $V_0\subseteq\Verts(H)$ \nick{(thus $\Delta(H)=\edges(H)$)}, then for any $e\in\Edges(H)$ one can take the dominating base $\Base=\{\emptyset,V_0\}$, 
\nick{with weights as in \eqref{Deg.usual} being $\ddstar=\dd_{V_0}=\Delta(H)-1$. 
}
\nick{W}ith this choice \Cref{thm:reg} gives a more efficient approximation of the adjacency tensor $\bA$ by test tensors, leading to a less restrictive decay condition on $p$ (see \eqref{def:WB}).
\end{example}

\begin{example}
For the case of 2-graphs, for an edge $e=\{u,v\}$ one can always take $\Base(e)=\{\emptyset,\{u\},\{v\}\}$ \nick{as in \Cref{ex:base.2}}. If $u$ (resp.\ $v$) has degree 1 then one can take \nick{the smaller dominating base} $\Base(e)=\{\emptyset,\{v\}\}$ (resp.\ $\{\emptyset, \{u\}\}$). 
In this case, the adjacency matrices for two graphs are close in the $\Basys(e)^*$-\nick{(semi-)}norm when they have approximately the same degree sequence. Such an approximation is indeed sufficient for approximating $\hom(H,G)$ when $H$ is a star (for which we can take such a base for every edge) as these are moments of the degree distribution. 
\end{example}

\begin{example}
\nick{When $H$ is a linear hypergraph, for which $|e\cap e'|\in \{0,1\}$ for all distinct $e,e'\in \Edges(H)$, then for each edge $e$ the base}
of singletons $\nick{\Base(e)=}\{\emptyset\}\cup\{\{v\}:v\in e\}$ is dominating, and the class of test tensors thus reduces to the class of cuts $\1_{I_1}\otimes\cdots\otimes\1_{I_r}$\,.
\end{example}

\nick{We return to these and other examples in \Cref{sec:intro.tails} where we state results for tails of homomorphism counts in $\bG_{n,p}$.}

\subsection{$\Basys^*$ decomposition and counting lemmas}
\label{sec:results.reg}

The following is our general decomposition lemma,
showing that, under the Erd\H{o}s--R\'enyi measure, most symmetric Boolean tensors $A\in \cAnr$ can be decomposed into a structured piece $A_{\str}$ that is a combination of a small number of test tensors of controlled size under the $\Basys$-norm, and a pseudorandom piece $A_{\rand}$ that is small in the dual $\Basys^*$-norm. The theorem allows for a tradeoff between the measure of the set of tensors to be excluded and the complexity of the resulting decomposition (in particular one may choose to make no exclusion).
Recall our notation $\jay$ for the symmetric Boolean $r$-tensor with $\jay(i_1,\dots, i_r) =1$ if and only if all of the arguments $i_1,\dots, i_r$ are distinct (so that for the Erd\H{o}s--R\'enyi tensor we have $\e \bA = p\jay$).

\begin{thm}[Decomposition lemma]	\label{thm:reg}
\nick{The following holds for} $C_0,c_0>0$ depending only on $r$.  
Fix a weighted base
 $\Basys=(r,e,\iota,\Base,\ddstar,\{\dd_\base\}_{\base\in\Base})$ and let $\kappa,\eps>0$. Assuming $n$ and $p\in (n^{-2}, 1)$ are such that
\begin{equation}	\label{def:WB}
\growing_{n,p}(\Basys) := \min_{\base\in\Base} \big\{ n^{r-|\base|}p^{\ddstar-\dd_\base+2} \big\}
\ge\frac{C_0\log n}{\eps^2(1\vee\log(1/p))}\,,
\end{equation}
then there exists a (possibly empty) exceptional set $\badset(\kappa,\eps)\subseteq\cAnr$ with
$
\mathbb{P}(\bA\in\badset(\kappa,\eps))\le p^{c_0\kappa n^{r}}
$
such that 
for each $A\in \cAnr\setminus \badset(\kappa,\eps)$ there is a decomposition
\begin{equation}	\label{decomposition}
A = A_{\str}+A_{\rand} 
\end{equation}
where
\begin{equation}	\label{AstrArand}
A_{\str} = p\jay + \sum_{i=1}^k \alpha_i T_i \qquad\text{ and } \qquad\;
\|A_{\rand}\|_{\Basys}^* \le \eps p
\end{equation}
for real numbers $\alpha_1,\dots, \alpha_k$ and test tensors $\test_1,\dots, \test_k\in \Test_\Basys$ satisfying
\begin{equation}	\label{eq:costD}
\sum_{i=1}^k \|\test_i\|_{\Basys} \le \kappa n^r/(\eps p)^2\,.
\end{equation}
Furthermore,  for each $1\le j\le k$,
\nick{the tensor} $T_j$ is separated from 
the span of $\{T_1,\dots, T_{j-1}\}$
by Euclidean distance at least $\eps p^{1+\ddstar} n^{r/2}$.
\end{thm}

\begin{remark}
The final statement on Euclidean distances will be useful for bounding covering numbers of $\cAnr\setminus \badset(\kappa,\eps)$ under the $\Basys^*$-norm.
\end{remark}

The following is our general sparse counting lemma, which we state for a multilinear generalization of homomorphism counts. For a collection
$\usym$ of symmetric tensors $\{\syms^{e}\}_{e\in \Edges(H)}$ in $\Symsnr$, we define 
\[
t_p(H,\usym) = \frac{\hom(H,\usym)}{n^{\verts(H)}p^{\edges(H)}} := \frac{1}{n^{\verts(H)}p^{\edges(H)}}
\sum_{\phi:\Verts(H)\to[n]}\,\,\prod_{e\in \Edges(H)}\syms^{e}(\phi(e))\,.
\]
For $\usym$ with $S^e\equiv S_0$ for some $S_0\in \Symsnr$ the above expression reduces to the previous definition $t_p(H, \usym) = t_p(H, S_0)$ from \eqref{def:ttp}. This multilinear generalization of homomorphism counts allows us to capture other functionals of interest such as induced subgraph counts. It also naturally appears in the proof of the counting lemma, which interpolates between two (weighted) hypergraphs that are close under the $\|\cdot\|_\BBasys^*$ seminorm.

Recall the notion of a dominating base system from \Cref{def:dom}. 

\begin{thm}[Counting lemma]
\label{thm:count}
Let $p\in (0,1)$ and let $H$ be an $r$-graph.
Let $\BBasys$ be an $H$-dominating base system as in \eqref{Deg.usual} and let $\|\cdot\|_\BBasys^*$ be the associated
seminorm on $\Tensnr$
as defined in \eqref{def:BBnorm}.
Let $\cC\subset\cAnr$ be a set of diameter at most $\eps p$ under 
$\|\cdot\|_\BBasys^*$ for some $\eps\in(0,1]$, and assume further that
there exist $Q_0\in \hull(\cC)$ and $L\ge1$ such that 
\begin{equation}	\label{assu:count2.0}
t_p(H',Q_0)\le L 
\end{equation}
for all proper subgraphs $H'\subset H$. 
Then for any 
$\uP=\{P^{e}\}_{e\in \Edges(H)}, \uQ=\{Q^{e}\}_{e\in \Edges(H)}$ with each $P^e,Q^e\in\hull(\cC)$, we have
\[
t_p(H,\uP)-t_p(H,\uQ) \ls_H L\eps\,.
\]
\end{thm}

The above is a consequence of \Cref{thm:count.gen} giving a counting lemma for the broader class of \emph{signed-homomorphism} functionals interpolating  between homomorphism counts and \emph{induced} homomorphism counts.

\section{Quantitative \ldps}
\label{sec:LDPs}

\subsection{Quantitative $\Basys^*$-norm LDPs}

As a consequence of \Cref{thm:reg} we obtain quantitative \ldps for the measure space $(\cAnr,\mu_{p})$ at large fixed $n$. 
For comparison, the classical \ldp for a sequence of measures $\nu_n$ on a topological space $\cQ$ states that for any $\cE\subset\cQ$, 
\begin{equation}	\label{classical-LDP}
-\inf_{x\in \cD} I(x)\le 
\liminf_{n\to \infty}\frac1{\rate_n}\log \nu_n(\cE) \le \limsup_{n\to \infty}\frac1{\rate_n}\log \nu_n(\cE)  \le -\inf_{x\in \cF} I(x)
\end{equation}
for any open $\cD\subseteq\cE$ and closed $\cF\supseteq\cE$, where $\rate_n$ is the speed and $I(\cdot)$ the \abbr{LDP} rate function.

For our quantitative result, the rate function is the relative entropy $\eye_p(Q)=\DKL(\mu_Q\|\mu_p)$ (see \eqref{def:eyepQ}). 
Now we specify our notions of outer and inner approximations of a set $\cE\subset\cAnr$.
Let $\Edges$ be an arbitrary finite collection of $r$-sets, and let $\BBasys=\{\Basys(e)\}_{e\in \Edges}$ be a collection of weighted bases $\Basys(e)$ over $e$.
Recalling the associated seminorm defined in \eqref{def:BBnorm}, for $Q\in\cQnr,\delta>0$ we denote the 
$\delta$-neighborhood of $Q$ in $\cAnr$ by
\begin{equation}	\label{def:Bnbhd}
\nbhd_{\BBasys}(Q,\delta) := \big\{ A\in \cAnr: 
\|Q-A\|_{\BBasys}^*\le \delta\big\}\,.
\end{equation}
For $\cE\subseteq\cAnr$ we denote the outer approximation
\begin{equation}	\label{def:Bnbhd-outer}
(\cE)_{\BBasys,\delta} := \bigcup_{A\in \cE} \hull(\nbhd_{\BBasys}(A,\delta))
\end{equation}
and the inner approximation
\begin{equation}	\label{def:Bnbhd-inner}
(\cE)^\circ_{\BBasys,\delta}:= \{ Q\in\cQnr: \nbhd_{\BBasys}(Q,\delta) \subseteq \cE\}\,.
\end{equation}
N.B.: $(\cE)_{\BBasys,\delta}$ and $(\cE)^\circ_{\BBasys,\delta}$ are subsets of the solid cube $\cQnr$.
We also emphasize that the convex hulls in \eqref{def:Bnbhd-outer} are different from (and generally proper subsets of) the balls $\{Q\in\cQ_{n,r}:\|Q-A\|_\BBasys^*\le \delta\}$. 

\begin{thm}[Quantitative \abbr{LDP}]
\label{thm:LDP}
Let $\BBasys$ be a collection of weighted bases as above.
\begin{enumerate}
\item[(a)] (LDP upper bound).
With $C_0,c_0$ as in \Cref{thm:reg}, let $p\in (n^{-2},1)$, 
$\kappa,\eps>0$, and assume $\growing_{n,p}(\BBasys):=  \min_{e\in \Edges}\growing_{n,p}(\Basys(e))$ satisfies the lower bound in \eqref{def:WB}.
Then for any $\cE\subseteq\cAnr$, 
\begin{equation}	\label{LDP.upper}
\log\P( \bA\in \cE) \le- \min\Big( \,\rate_\star \;,\;\, \inf\big\{\eye_p(Q) : Q\in (\cE)_{\BBasys,\eps p} \big\} - \rate_{\ME}
\,\Big)\,
\end{equation}
for some \emph{cutoff rate} 
\begin{equation}	\label{Rstar}
\rate_\star =c_0\kappa n^r\log(1/p) - O_{\BBasys}(1)
\end{equation}
and \emph{metric entropy rate} 
\begin{equation}	\label{RME}
\rate_{\ME} \ls_{\BBasys} 
\frac{\kappa n^r \log n}{\eps^2\growing_{n,p}(\BBasys)}\,.
\end{equation}
\item[(b)] 
(LDP lower bound).
If $\growing_{n,p}(\BBasys) \ge C_0'\eps^{-2}\log n$ for a sufficiently large constant $C_0'(r)>0$, then for any $\cE\subseteq\cAnr$,
\begin{equation}	\label{LDP.lower}
\log\P(\bA\in \cE) \ge -\inf\big\{ \eye_p(Q): Q\in (\cE)^\circ_{\BBasys,\eps p}\big\} - O\Big(1+n^{r/2}
\Lp
\Big)
\end{equation}
where the implied constant is absolute.
\end{enumerate}
\end{thm}

In applications we take $\kappa$ such that $\rate_\star$ exceeds the rate of the rare event of interest -- for the upper tail of $t_p(H,\bG)$ this means taking $\kappa=Kp^\Delta$ for a sufficiently large constant $K$. 
From our assumption on $\growing_{n,p}(\BBasys)$ we have $\rate_{\ME}=O(\rate_\star)$,
but we must further ensure that $\growing_{n,p}(\BBasys)\gg1$ to have $\rate_{\ME}$ be negligible compared to the main term. This amounts to a lower bound constraint on $p$, and there is generally flexibility (within the requirements of a counting lemma) to choose the weighted bases to lighten this constraint.  

The upper bound of \Cref{thm:LDP} follows from \Cref{thm:reg} by a straightforward covering argument, combined with the non-asymptotic bound \eqref{upperLDP.convex}.
The term $\rate_{\ME}$ is the sum of log-covering numbers of $\cAnr\setminus\badset(Kp^\Delta,\eps)$ by $\eps p$-balls in the $\Basys(e)^*$-norms. 
(The refinement to convex hulls of their intersections is important when combining \Cref{thm:LDP} with \Cref{thm:count}.)
To obtain bounds for upper tails of functionals $f:\cQnr\to \R$ 
we apply \Cref{thm:LDP} to $\cE= \{f\ge t\}$, 
and then use \Cref{thm:count} to show that $\{f\ge t+\eta\}\subset (\cE)^\circ_{\BBasys,\eps p}\subset(\cE)_{\BBasys, \eps p}\subset\{f\ge t-\eta\}$ for some $\eta=o_{\eps\to0}(1)$.

\subsection{Upper and lower tails for hypergraph counts}
\label{sec:intro.tails}

Our main application of \Cref{thm:LDP}, in combination with the counting lemma (\Cref{thm:count}),
is \Cref{thm:tails} on joint upper and lower tails for homomorphism counts.
In this subsection we state some corollaries of \Cref{thm:tails} for specific classes of hypergraphs and define the parameter $\DDprime(H)$. 
Further applications of Theorems \ref{thm:LDP} and \ref{thm:count} are given in $\mathsection$\ref{sec:other}.

For the case $m=1$ and $p=o(1)$, the optimization problem $\Phi_{n,p}(H,\delta)$ was recently analyzed in \cite{LiZh} for certain $H$ -- specifically, complete hypergraphs and the 3-graph depicted in Figure \ref{fig:LiZh} -- where they deduced upper tail asymptotics for $p\gg n^{-1/(6 
\nick{\edges(H)})}\log n$ via the general framework from \cite{Eldan}, which required bounding the Gaussian width of gradients for the homomorphism counting functionals. Combining \Cref{thm:tails} with \cite[Theorem 2.3]{LiZh} we obtain the following:

\begin{cor}	\label{cor:LiZh}
For fixed $r,k$, and $K_k^{(r)}$ the $r$-uniform clique on $k$ vertices, we have
\begin{align}	
\log & \; \mathbb{P}\big(\, t_p(K_k^{(r)},\bG)\ge 1+\delta\,\big) \nonumber \\
& = 
 -(1+o(1)) \min\Big\{ \frac{\delta^{r/k}}{r!}, \frac{\delta}{(r-1)!k} \Big\} n^r p^{{k-1\choose r-1}} \log(1/p) 
 \label{LiZh:asymp1}
\end{align}
if $n^{-c(r,k)} \ll p\ll 1$ with $c(r,k) = 1/({k-1\choose r-1} + 1)$. 
Furthermore, the lower bound holds for the wider range $n^{-1/{k-1\choose r-1}}\ll p\ll1$.

Moreover, with $H$ the 3-graph depicted in Figure \ref{fig:LiZh}, for $n^{-1/2}\ll p\ll1$, we have
\begin{align}	
\log & \;  \mathbb{P}\big(\, t_p(H,\bG)\ge 1+\delta\,\big) \nonumber \\
& =
 -\Big(\frac16+o(1)\Big) \min\Big\{ \sqrt{9+3\delta}-3, \sqrt{\delta} \Big\} n^3 p^2 \log(1/p) \,.
 \label{LiZh:asymp2}
\end{align}
\end{cor}

The ranges of $p$ for the upper bounds follow from our computation of the parameters $\DDprime(K_k^{(r)})$ and $\DDprime(H)$ in Examples \ref{ex:cliques} and \ref{ex:linear} below. 
Analogously to the case $r=2$, the asymptotic \eqref{LiZh:asymp1} for cliques matches the probability of appearance of higher-rank analogues of the ``clique'' and ``hub'' structures that were first described in \cite{LuZh:sparse}. 
One may view $H$ from Figure \ref{fig:LiZh} as the 3-graph obtained by transposing the incidence matrix of the complete 2-graph $G=K_4$.
The interest in this particular hypergraph is that the mechanism for large deviations of $\hom(H, \bG)$ is more intricate than the simple appearance of a clique or hub structure as is the case for $H=K_k^{(r)}$ -- see \cite{LiZh} for further discussion.

\begin{figure}
\includegraphics[width=5cm]{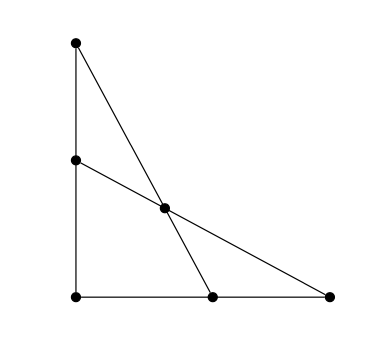}
\caption{\label{fig:LiZh}The 3-graph considered in \cite{LiZh} (see \Cref{cor:LiZh}) has six vertices (dots)
with four edges denoted by straight lines. (Reproduced with permission from \cite{LiZh}.) }
\end{figure}

We next highlight consequences of \Cref{thm:tails} combined with a result from \cite{BGLZ} providing an asymptotic for $\Phi_{n,p}(H,\delta)$ for the for the case $r=2$, $m=1$ and $p=o(1)$.

\begin{cor}[The case of $2$-graphs]
\label{cor:tails2}
Let $H$ be a fixed 2-graph of maximal degree $\Delta\ge2$. 
\begin{itemize}
\item[(a)]
For any fixed $\delta>0$, assuming $n^{-1/(\Delta+1)}\ll p\le\PUP$, 
\begin{equation}	\label{UT.2-graphs} 
\log \mathbb{P}\big(\, t_p(H,\bG)\ge 1+\delta\,\big)
=  -
(1+o(1))\Phi_{n,p}(H,\delta+o(1)) 
\end{equation}
and for fixed $\delta\in (0,1)$ and $n^{-1/(\Delta+1)}\log n\ll p<1$,
\begin{equation}	\label{LT.2-graphs} 
\log \mathbb{P}\big(\, t_p(H,\bG)\le 1-\delta\,\big)
=  -
(1+o(1))\Psi_{n,p}(H,\delta+o(1)) \,.
\end{equation}

\item[(b)]
Furthermore, \eqref{UT.2-graphs} (resp.\ \eqref{LT.2-graphs}) holds in the range $n^{-1/\Delta}\ll p\le \PUP$ (resp.\ $p\gg n^{-1/\Delta}\log n$) whenever each vertex of $H$ of degree $\Delta$ is contained in an isolated star. 

\item[(c)]
Further specializing to the case that $H=K_{1,\Delta}$ (the $\Delta$-armed star) and $n^{-1/\Delta}\ll p\ll1$, 
\begin{equation}	\label{tail:stars}
\log\mathbb{P}\big(\, t_p(K_{1,\Delta},\bG)\ge 1+\delta\,\big) = 
 -(1+o(1)) \delta n^2 p^{\Delta} \log(1/p) .
\end{equation}
\end{itemize}
\end{cor}

\begin{remark}
For general $r$ the asymptotic \eqref{UT.2-graphs} holds in the range\\
$n^{-1/\Delta(H)}\ll p\le\PUP$ for instance when $H$ is a sunflower with at least $2/(r-|V_0|)$ petals, where $V_0\subset\Verts(H)$ is the kernel of $H$ -- see \Cref{ex:sunflower}.
\end{remark}

\begin{remark}
The assumption on $p$ in Part (c) is sharp, as the upper tail rate is known to be of size $\Theta_\delta(n^{1+1/\Delta}p\log(1/p))$ for $n^{-(1+1/\Delta)}(\log n)^{1/(\Delta-1)}\ls p\ls n^{-1/\Delta}$ \cite{SiWa:stars}.
\end{remark}

\begin{proof}
Part (a) is immediate from \Cref{thm:tails}. For (b) we only need to verify that $\DDprime(H)\le\Delta(H)$ for such 2-graphs, which we do in \Cref{ex:2graphs} below.
Part (c) follows from \Cref{thm:tails} and \cite[Theorem 1.5]{BGLZ}, where we note that the independence polynomial of the graph $H^*$ defined in that work is $P_{H^*}(\theta)=1+\theta$ whenever $H$ is a star.
\end{proof}

We now define the parameter $\DDprime(H)$ appearing in \Cref{thm:tails}.
Roughly speaking, it measures how efficiently one can cover the edge overlaps of $H$, with a small value for $r$-graphs with overlaps concentrated on a small number of small vertex sets (such as sunflowers or sparse linear hypergraphs) and a large value for cliques.
A key point is that it depends only on the neighborhood structure of single edges, similarly to how $\Delta(H)$ depends only on the neighborhood of single vertices. Thus it is a \emph{local} hypergraph parameter that is independent of the size of $H$ (as quantified by $\verts(H)$ or $\edges(H)$). 

Recall the notion of a dominating weighted base from \Cref{def:dom}.
Given a dominating base $\Base$ over an edge $e\in\Edges(H)$, 
recalling the edge degree parameters from \eqref{Deg.usual}, we set
\begin{align}	\label{def:dprime.dom}
\ddprime_\Dom(e) &= \max\{ \ddprime_\dom(e):\dom\in \Dom\} \,, \nonumber \\
\ddprime_\dom(e) 
&= 
\frac{\dd^H(e) - \dd^H_\base(e)+2}{|e\setminus \base|}
=
\begin{cases}
\frac{\dd^H(e\setminus\dom)+1 }{|e\setminus \dom|}	& \dom\ne \emptyset\\
\frac{\dd^H(e) +2}{r} & \dom = \emptyset 
\end{cases}	
\;.
\end{align}
(Recall that $\dd^H(e)$ does not count $e$ itself.)
Note that $\ddprime_\dom(e)$ is a normalized count of the edges overlapping $e\setminus \base$, i.e.\ those that are \emph{not} dominated by $\dom$. (If the 2 were replaced by 1 in the first expression for $\ddprime_\base(e)$ then it would be the average degree of vertices in $e\setminus\base$.)
We define
\begin{equation}\label{eq:def-delta'}
\ddprime(e) = \min_{\Dom}\ddprime_\Dom(e)\,,\qquad\quad
\DDprime(H) =\max_{e\in \Edges(H)}\ddprime(e)
\end{equation}
where the minimum is taken over all dominating bases $\Dom$ for $e$.
That is, $\DDprime(H)$ is the smallest number $\Delta'$ such that every edge $e\in \Edges(H)$ has a dominating base $\Base(e)$ 
such that 	
\begin{equation}	\label{def:DDprime}
|\{e'\in \Edges(H): e'\cap (e\setminus\base)\ne\emptyset\}|\le \Delta'|e\setminus\base|-1
\end{equation}
for every $\base\in\Base(e)$ (N.B.: the left hand side counts the edge $e$). 
From this definition, recalling the notation of \eqref{def:WB} and \Cref{thm:LDP}, we have that for any $r$-graph $H$, whenever $np^{\DDprime(H)}\ge1$, there exists 
an $H$-dominating base system $\BBasys=\{\Basys(e)\}_{e\in\Edges(H)}$ 
with weights as in \eqref{Deg.usual} such that
\begin{equation}	\label{growing.LB}
\growing_{n,p}(\BBasys) =\min_{e\in\Edges(H)} \growing_{n,p}(\Basys(e))
=  \min_{e\in \Edges(H)} \min_{\base\in \Base(e)} (np^{\ddprime_\base(e)})^{r-|\base|} 
 \ge np^{\DDprime(H)}.
\end{equation}

We record some general bounds on $\DDprime(H)$ by considering specific dominating bases. 
We drop the superscript $H$ from all notation for the remainder of this section.
We write
\[
\Delta_\star(H) = \max_{e\in \Edges(H)} \dd(e)
\]
for the maximal edge degree (which does not count the edge $e$ itself). 
For $1\le s\le r$ let 
\[
\Delta_s(H) = \max_{e\in \Edges(H)} \max_{U\in {e\choose s}}  |\{e': U\cap e'\ne\emptyset\}|
\]
denote the largest number of hyperedges intersecting a size-$s$ subset of some edge of $H$;
in particular $\Delta_1(H) = \Delta(H)$, $\Delta_r(H)=\Delta_\star(H)+1$, 
and $\Delta_s(H) \le s\Delta(H)$ for every $1\le s\le r-1$. 
Since
$
\ddprime(e) \ge (\dd(e) +2)/r
$,
it follows that for any $r$-graph $H$, 
\begin{equation}	\label{DD.LB}
\DDprime(H) \ge \frac{  \Delta_\star(H) +2}r 
\ge \frac{\Delta(H)+1}{r}
\end{equation}
whereas taking $\Dom(e)={e\choose r-1}\cup\{\emptyset\}$ (which is always a dominating base) shows
\begin{equation}	\label{DD.bound1}
\DDprime(H)\le\Delta(H)+1.
\end{equation}	
Indeed, for each $\dom=e\setminus\{v\}$ we have $\dd(e\setminus\base)=\dd(v)\le\Delta(H)$.
If every pair of edges overlaps in at most $s_0$ vertices, then taking the bases $\Dom(e)=\{\emptyset\}\cup {e\choose s_0}$ 
we obtain the sharper bound
\begin{equation}	\label{DD.bound3}
\DDprime(H) \le \max_{s\in \{r,r-s_0\}} \frac{\Delta_{s}(H) + 1}{s} .
\end{equation}

\begin{example}[Cliques]
\label{ex:cliques}
When $H$ is the $r$-uniform clique on $k$ vertices, we have $\Delta(H)=\binom{k-1}{r-1}$
and $\DDprime(H)=\binom{k-1}{r-1}+1$, so that equality holds in \eqref{DD.bound1}.
Indeed, for each hyperedge $e$ we are forced to take $\Dom(e) = {e\choose r-1}\cup \{\emptyset\}$ to satisfy the domination condition.
\end{example}

\begin{example}[Sunflowers and stars]
\label{ex:sunflower}
When $H$ is a sunflower, with pairwise intersections of all $\Delta$ edges (``petals'') equal to a common ``kernel'' $V_0\subset \Verts(H)$,
for every edge the optimal base is $\Dom(e) = \{\emptyset, V_0\}$, for which we have $\ddprime_\emptyset(e)=\frac{\Delta+1}{r}$ and $\ddprime_{V_0}(e) = \frac2{r-|V_0|}$, and so 
$
\DDprime(H) = \max\{ \frac{\Delta+1}{r} , \frac{2}{r-|V_0|}\}.
$
Thus, sunflowers attain the minimum in \eqref{DD.LB} as long as the kernel is of size $|V_0| \le r\frac{\Delta-1}{\Delta+1}$.
For the $r$-uniform $\Delta$-armed star, with $\Delta\ge2$ (with $|V_0|=1$) we have $\DDprime(H)= \frac{\Delta+1}{r}$ when $r\ge3$ and $\DDprime(H)=\max\{\frac{\Delta+1}2,2\}\le\Delta$ when $r=2$.
\end{example}

\begin{example}[Linear hypergraphs]
\label{ex:linear}
When all pairs of edges share at most one vertex then \eqref{DD.bound3} holds with $s_0=1$.
For instance, for linear cycles (or disjoint unions thereof),
$
\DDprime(H) =\max\{ \frac4r,  \frac{3}{r-1}\}
$
which 
attains the lower bound \eqref{DD.LB} of $4/r$ for all $r\ge 4$. 
For 2-graphs of degree 2 we have$\DDprime(H) \le 3$, and one checks that in fact $\DDprime(H) = 3$. 
For the linear $3$-graph of \Cref{cor:LiZh}, taking $\Base(e) = {e\choose 1}\cup \{\emptyset\}$ shows that $\DDprime(H) = \frac12(\Delta_2(H)+1) = 2$.
For the Fano plane one checks that $\DDprime(H)=\Delta(H)=3$.
\end{example}

\begin{example}[2-graphs]
\label{ex:2graphs}
As was noted in \eqref{DD.bound1} we always have $\DDprime(H)\le \Delta(H)+1$.
Here we verify that $\DDprime(H)\le\Delta(H)$ for any 2-graph $H$ as in \Cref{cor:tails2}(b).
Indeed, for any edge $e=\{u,v\}$ with an end $v$ of maximal degree we take the base $\Base(e)=\{\emptyset,\{v\}\}$, giving $\ddprime(e) = \max\{\frac12(\Delta(H)+1),2\}\le \Delta(H)$. For any other edge, the maximal base $\Base(e)=\{\emptyset, \{u\},\{v\}\}$ verifies that $\ddprime(e)\le \Delta(H)$, and thus $\DDprime(H)=\max_e\ddprime(e)\le \Delta(H)$. 
\end{example}

\nick{\begin{remark}\label{rmk:pLB'}
As in Remark \ref{rmk:pLB}, we believe that the right hand side of \eqref{eq:costD} can be improved to $\kappa n^r \epsilon^{-2}p^{-1}$, which would allow us to replace the left hand side of \eqref{def:WB} with $\min_{\base\in \Base}\left\{n^{r-|\base|}p^{d_*-d_f+1}\right\}$. This would in turn imply that the conclusion of Theorem \ref{thm:tails} holds as long as $p\gg n^{-1/\tilde{\Delta}(H)}$ where $\tilde{\Delta}(H)$ is defined as follows. Given a dominating base $\Base$ over an edge $e\in \Edges(H)$, we set 
\begin{align*}	
\tilde{d}_\Dom(e) &= \max\{ \tilde{d}_\dom(e):\dom\in \Dom\}\,, \\
\tilde{d}_\dom(e) 
&= 
\frac{\dd^H(e) - \dd^H_\base(e)+1}{|e\setminus \base|}
=
\begin{cases}
\frac{\dd^H(e\setminus\dom) }{|e\setminus \dom|}	& \dom\ne \emptyset\\
\frac{\dd^H(e) +1}{r} & \dom = \emptyset 
\end{cases}	
\;.
\end{align*}
(Compare \eqref{def:dprime.dom}.)
Note that $\tilde{d}_\base(e)$ is the average degree of vertices in $e\setminus\base$.
We define
\[
\tilde{d}(e) = \min_{\Dom}\tilde{d}_\Dom(e)\,,\qquad\quad
\tilde{\Delta}(H) =\max_{e\in \Edges(H)}\tilde{d}(e)
\]
where the minimum is taken over all dominating bases $\Dom$ for $e$.
\end{remark}}

\section{\label{sec:count.proof}Proof of \Cref{thm:count} (counting lemma)} 

We will actually prove a more general version, involving a generalization of homomorphism counts that also includes \emph{induced} homomorphism counts as a special case. 
We say a pair $\cH= (H, \xi)$ is a \emph{signed hypergraph}
if $H=(\Verts, \Edges)$ is a hypergraph and $\xi:\Edges\to \{-1,+1\}$ is a labeling of the edges by signs. 
Recall from $\mathsection$\ref{sec:notation} that  $H'= (\Verts',\Edges')\subseteq H$ if $\Verts'\subseteq\Verts$ and $\Edges'\subseteq\Edges$, and $H'\subset H$ if $\Verts'\subseteq \Verts$ and $\Edges'\subset \Edges$.
We say $\cH'=(H',\xi')\subseteq \cH=(H,\xi)$ (resp.\ $\cH'=(H',\xi')\subset \cH$) if $H'\subseteq H$ (resp.\ $H'\subset H$) and $\xi'= \xi|_{\Edges'}$. 
For a signed hypergraph $\cH=(H,\xi)$, the signing induces two subgraphs of $H$ given by
$H_\pm$ with $\Verts(H_\pm) = \Verts(H)$ and $\Edges(H_\pm)= \xi^{-1}(\pm1)$.
We extend the definition of homomorphism counts to signed hypergraphs by defining for any $\cH'\subseteq\cH$ and $\usym=(S^e)_{e\in \Edges(H)}\in \Syms^{\Edges(H)}$,
\begin{equation}	\label{def:genhom}
\hom(\cH', \usym) = \sum_{\phi: \Verts(H')\to [n]} \prod_{e\in \Edges(H'_+)} S^e(\phi_e) \prod_{e\in \Edges(H'_-)} (1-S^e(\phi_e)).
\end{equation}
For compactness, here and in the remainder of the section we write 
\[
\phi_v:=\phi(v),\qquad \phi_e:= \phi(e)= \{\phi_v\}_{v\in e}
\]
and similarly $\phi_U:=\phi(U)$ for general $U\subset \Verts$.
 
We can \nick{alternatively} express this using the functional $\hom(H,\cdot)$ as follows: with $\xi$ fixed, we denote
\begin{equation}	\label{def:Sxi}
\wt{\uS} = \wt{\uS}_\xi = (\wt{S}^e)_{e\in \Edges(H)}\,, \qquad 
\wt{S}^e:= \begin{cases} S^e & \xi(e) = +1\\ \jay-S^e & \xi(e) = -1\,, \end{cases}
\end{equation}
(Recall $\jay$ is the tensor with entries 1 when all indices are distinct and 0 otherwise.)
We have
\begin{equation}	\label{hom:Sxi}
\hom(\cH, \uS) = \hom(H, \wt{\uS}).
\end{equation}

\Cref{thm:count} follows immediately from the next result upon taking the trivial labeling $\xi(e)\equiv 1$.

\begin{thm}[Counting lemma for signed homomorphisms]	\label{thm:count.gen}
Let $p\in (0,1)$ and let $\cH=(H, \xi)$ be a signed hypergraph as above. 
For each $e\in \Edges(H)$ let $\Dom(e)$ be an $H$-dominating base for $e$, 
and define a weighted base $\Basys(e)=(r,e,\iota_e,\Base(e),\dd^{H_+}(e), \{\dd^{H_+}_\dom(e)\}_{\dom\in \Dom(e)})$.
With base system $\BBasys=\{\Basys(e)\}_{e\in\Edges(H)}$, let $\|\cdot\|_{\BBasys}^*$ be the associated
seminorm on $\Tensnr$ 
as defined in \eqref{def:BBnorm}.
Let $\cC\subset\cAnr$ be 
a set of diameter at most $\eps p$ under $\|\cdot\|_\BBasys^*$ for some $\eps\in(0,1]$, and assume further that
there exists $Q_0\in \hull(\cC), L\ge1$ such that 
\begin{equation}	\label{assu:count2}
\hom(\cH', Q_0) \le L n^{\verts(H')}p^{\edges(H'_+)} \qquad \forall \cH'\nick{=(H',\xi')} \subset \cH.
\end{equation}
Then for all 
$\uP= (P^e)_{e\in \Edges(H)}$ and $\uQ=(Q^e)_{e\in\Edges(H)}$ with each $P^e,Q^e\in\hull(\cC)$, 
\begin{equation}	\label{count:goal0}
\big|\hom(\cH, \uP) - \hom(\cH, \uQ)  \big| \ls_{\cH} L\eps n^{\verts(H)} p^{\edges(H_+)}.
\end{equation}
\end{thm}

Note that while the bases for $\Basys(e)$ are $H$-dominating, the weights are taken from the neighborhood structure in the subgraph $H_+$.

\begin{proof}
Fix $\cH$ and $\cC$ as in the statement of the lemma. 
We prove by induction on $m\le \edges(H)$ that for all $\cH'=(H',\xi')\subseteq \cH$ with $\edges(H')\le m$, all $\uA=(A^e)\in \cC^{\Edges(H')}$ and $\uQ=(Q^e)\in \hull(\cC)^{\Edges(H')}$,
we have
\begin{equation}	\label{count:goal1}
\big|\hom(\cH', \uA) - \hom(\cH', \uQ) \big| \le C(m,r) L\eps n^{\verts(H')} p^{\edges(H'_+)}
\end{equation}
for some $C(m,r)<\infty$. 
One can then replace $\uA$ with $\uP$ as in the theorem statement via the triangle inequality.

The base case $m=0$ holds trivially. Assume now \eqref{count:goal1}
holds for all $\cH'\subseteq\cH$ with $\edges(H')\le m-1$.
We fix $\cH'=(H',\xi')\subseteq\cH$ with $\edges(H')=m$ and $\uA$ and $\uQ$ as above.
For brevity we write
\[
\Verts':= \Verts(H'),\qquad \Edges':=\Edges(H'), \qquad \Edges'_\pm := \Edges(H'_\pm).
\]

We first express $\hom(\mathcal{H}',\uQ)$ as a convex combination
of homomorphism counts for Boolean tensors. Labeling the elements of $\cC$ as $B_j$, $1\le j\le |\cC|$, for each $e\in \Edges(H)$ we express $Q^e=\sum_jc_j^eB_j$ for coefficients $c_j^e\in [0,1]$ with $\sum_jc_j^e=1$. 
We have
\begin{align*}
\hom(\mathcal{H}',\uQ) 
&  
=\sum_{\phi:\Verts'\to[n]}\,\,\prod_{e\in \Edges'_+}Q^{e}(\phi_e)\prod_{e\in\Edges'_-}(1-Q^{e}(\phi_e))\\
 &  =\sum_{\phi:\Verts'\to[n]}\,\,\prod_{e\in \Edges'_+}\bigg[\sum_{j}c_{j}^{e}B_j(\phi_e)\bigg]\prod_{e\in\Edges'_-}\bigg[\sum_{j}c_{j}^{e}(1-B_j(\phi_e))\bigg]\\
 & =\sum_{\phi:\Verts'\to[n]}\,\,\sum_{\bs j}\,\,\prod_{e\in \Edges'}c_{j_{e}}^{e}\prod_{e\in \Edges'_+}B_{j_{e}}(\phi_e)\prod_{e\in\Edges'_-}(1-B_{j_{e}}(\phi_e)) \\
 &  =\sum_{\bs j}\,\, c_{\bs j} \hom(\mathcal{H}',\uB_{\bs j})
\end{align*}
where sums over $\bs j$ run over all $\bs j=(j_e)_{e\in \Edges'}\in |\cC|^{\Edges'}$, and we set
\[
c_{\bs j}:= \prod_{e\in \Edges'}c_{j_{e}}^{e}\,,\qquad 
\uB_{\bs j}:= (B_{j_e})_{e\in \Edges'}.
\]
Now noting that $\sum_{\bs j}c_{\bs j}=1$, we have
\begin{align*}
 \left|\hom(\mathcal{H}',\uA) - \hom(\mathcal{H}',\uQ) \right|
 & = \bigg| \sum_{\bs j}\,\, c_{\bs j} \big( \hom(\mathcal{H}',\uA) - \hom(\mathcal{H}',\uB_{\bs j})  \big) \bigg| \\
 &\le \sum_{\bs j}\,\, c_{\bs j} \big|  \hom(\mathcal{H}',\uA) - \hom(\mathcal{H}',\uB_{\bs j}) \big|.
\end{align*}
Thus, fixing collections $\uA= (A^{e})_{e\in \Edges'}$ and $\uB =(B^{e})_{e\in \Edges'}$ of tensors in $\cC$, it suffices to show 
\begin{align}
\big| \hom(\mathcal{H}',\uA) - \hom(\mathcal{H}',\uB) \big|
 & \le C(m,r)L\eps n^{|\Verts'|}p^{|\Edges'_+|}.
 \label{goal:count2}
\end{align}

Label the hyperedges of $\Edges'$ as $e_1,\dots, e_m$.
Recalling the notation \eqref{def:Sxi}, we express the difference of homomorphism counts as a telescoping sum
\begin{align*}
\hom(\mathcal{H}',\uA) &- \hom(\mathcal{H}',\uB)
=\hom(H', \utA) - \hom(H', \utB) \\
&
=\sum_{\phi:\Verts'\to [n]} \sum_{k=1}^m \big[ \tA^{e_k}(\phi_{e_k}) - \tB^{e_k}(\phi_{e_k}) \big] \prod_{j<k} \tB^{e_j}(\phi_{e_j}) \prod_{j>k} \tA^{e_j}(\phi_{e_j}) \\
&
=\sum_{k=1}^m \sum_{\phi: \Verts'\setminus e_k\to [n]} \widetilde \Gamma^{e_k}_{k,\phi}
\prod_{\substack{e\in \Edges': \\e\cap e_k=\emptyset}} \tZ_k^e (\phi_e) 
\end{align*}
where 
\begin{align*}
\cAnr\ni \tZ_k^e & := \begin{cases} \tB^{e} & \text{ for $e= e_j$ with  $j\le k$}\\ 
\tA^{e} & \text{ for $e=e_j$ with  $j>k,$}\end{cases}\\
\widetilde \Gamma^{e_k}_{k,\phi} &:=
\sum_{\psi:e_k\to [n]} \big[ \tA^{e_k}(\psi_{e_k}) - \tB^{e_k}(\psi_{e_k}) \big]
\prod_{e\in \bdy^{H'}(e_k)} \tZ_k^e(\phi_{e\setminus e_k}\cup \psi_{e\cap e_k}),
\end{align*}
and $\tZ_k^e(\phi_{e\setminus e_k}\cup \psi_{e\cap e_k})$ is the value of the symmetric $r$-tensor $\tZ_k^e$ evaluated at an arbitrary ordering of the set $\phi_{e\setminus e_k}\cup \psi_{e\cap e_k}$. 

Now we recognize the expression
\begin{align}
\prod_{e\in \bdy^{H'}(e_k)} \tZ_k^e(\phi_{e\setminus e_k} \cup \psi_{e\cap e_k}) 
&= \prod_{\dom\in \Dom(e_k)} \prod_{e\in \bdy^{H'}_\dom(e_k)} \tZ_k^e (\phi_{e\setminus e_k}\cup \psi_{e\cap e_k}) \notag\\
&=: \prod_{\dom\in \Dom(e_k)} \ttest_\dom\big( (\psi_v)_{v\in \dom}\big) 
=: \test_{e_k, \phi}\big( (\psi_v)_{v\in e_k}\big)	\label{def:Tek.test}
\end{align}
as the output of a test tensor $\test_{e_k,\phi}\in \Test(e_k)$.
Hence we can express
\begin{align*}
\widetilde{\Gamma}^{e_k}_{k,\phi}=
\big\langle \tA^{e_k} - \tB^{e_k}\, , \, \test_{e_k,\phi}\big\rangle \,.
\end{align*}
Noting that $\tA^{e_k}- \tB^{e_k} = \pm (A^{e_k} - B^{e_k})$ for each $k$, we can apply the triangle inequality and our assumption 
on the diameter of $\cC$ to bound
\begin{align}
\big|\hom(\mathcal{H}',\uA)  & - \hom(\mathcal{H}',\uB) \big| \notag \\
&\le  \sum_{k=1}^m \sum_{\phi: \Verts'\setminus e_k\to [n]}
 \big| \big\langle A^{e_k} - B^{e_k}\, , \, \test_{e_k,\phi}\big\rangle \big| 
 \prod_{\substack{e\in \Edges': \\e\cap e_k=\emptyset}} \tZ_k^e (\phi_e) 
\notag \\
&\le \eps p \sum_{k=1}^m \sum_{\phi: \Verts'\setminus e_k\to [n]} \|\test_{e_k,\phi}\|_{\Basys(e_k)}
\prod_{\substack{e\in \Edges': \\e\cap e_k=\emptyset}} \tZ_k^e (\phi_e) \,.	\label{bound:homAB}
\end{align}
Recalling our choice of weights for the weighted base $\Basys(e_k)$, by definition we have
\[
\|\test_{e_k,\phi}\|_{\Basys(e_k)} 
\le \|\test_{e_k,\phi}\|_1 + \sum_{\dom \in \Dom(e_k)} n^{r-|\dom|} p^{\dd^{H_+}(e_k) - \dd^{H_+}_\dom(e_k)}\|\ttest_\dom\|_1.
\]
From \eqref{def:Tek.test},
\[
 \|\test_{e_k,\phi}\|_1
= \sum_{\psi: e_k\to [n]}\prod_{e\in \bdy^{H'}(e_k)} \tZ_k^e(\phi_{e\setminus e_k} \cup \psi_{e\cap e_k}) 
\]
and for $\|\ttest_\dom\|_1$ we have the same expression with $\bdy^{H'}_\dom(e_k)$ in place of $\bdy^{H'}(e_k)$. 
Substituting these bounds in \eqref{bound:homAB} we obtain
\begin{align}
&\big|\hom(\mathcal{H}',\uA) - \hom(\mathcal{H}',\uB) \big|		\notag\\
&
\le \eps p \sum_{k=1}^m \bigg\{ \hom(H^{(k)}, \utZ_k) 
+ \sum_{\dom\in \Dom(e_k)} n^{r-|\dom|} p^{\dd^{H_+}(e_k) - \dd^{H_+}_\dom(e_k)} \hom(H^{(k,\dom)}, \utZ_k) \bigg\}	\label{bound:homAB2}
\end{align}
where $H^{(k)}= (\Verts', \Edges'\setminus\{e_k\})$, and for $H^{(k,\dom)}$, 
\[
\Verts(H^{(k,\dom)}) = \Verts'\setminus (e_k\setminus \dom), \qquad \Edges(H^{(k,\dom)}) = \{ e\in \Edges': e\cap (e_k\setminus \dom) = \emptyset\}.
\]
In particular, 
\begin{align}
\verts(H^{(k,\dom)}) &= |\Verts'| -r + |\dom| 	\label{Vhkf}
\end{align}
and
\begin{align}
\edges(H^{(k,\dom)}_+) & = |\{ e\in \Edges'_+: e\cap (e_k\setminus \dom) = \emptyset\}|	\notag\\
&= |\Edges'_+ \cap\{e_k\}^c \cap \bdy^{H_+}(e_k)\cap \bdy^{H_+}_\dom(e_k)^c|	\notag\\
&= |\Edges'_+|-1_{\xi(e_k)=1} -\dd^{H'_+}(e_k) +  \dd^{H'_+}_\dom(e_k)	\notag\\
&\ge |\Edges'_+|-1_{\xi(e_k)=1} -\dd^{H_+}(e_k) +  \dd^{H_+}_\dom(e_k).	\label{Ehkf}
\end{align}
By restricting $\xi$ to the edge sets of $H^{(k)}$ and $H^{(k,\dom)}$ we obtain signed hypergraphs $\cH^{(k)}\subset \cH'$ and $\cH^{(k,\dom)}\subset\cH'$ for each $k\in [m]$ and $\dom\in \Dom(e_k)$. 
For any $\cH''=(H'',\xi'')$ in this collection of signed hypergraphs we have $\edges(H'')\le m-1$, and by the induction hypothesis and the assumption \eqref{assu:count2}, for any $\uQ\in \hull(\cC)^{\Edges'}$,
\begin{align*}
\hom(H'',\uQ) = \hom(\cH'', \uQ) 
& \le \hom(\cH'', Q_0) + |\hom(\cH'', \uQ) - \hom(\cH'', A_0)| \\
&\le (1+ C(m-1,r) ) L n^{\verts(H'')} p^{\edges(H''_+)} 
\end{align*}
(recalling $\eps\le 1$).
Applying this for each $\cH^{(k)}$ and $\cH^{(k,\dom)}$ with $\uQ= \uZ_k$ and combining with \eqref{Vhkf}, \eqref{Ehkf} we obtain
\begin{align*}
\hom(\cH^{(k)}, \utZ_k) 
&\le (1+ C(m-1,r) ) L n^{|\Verts'|} p^{\edges(H'_+)-1} \\
\hom(\cH^{(k,\dom)}, \utZ_k)
&\le (1+ C(m-1,r) ) L n^{|\Verts'|-r+|\dom|} p^{\edges(H'_+)-1-\dd^{H_+}(e_k)+\dd^{H_+}_\dom(e_k)} .
\end{align*}
Substituting these bounds into \eqref{bound:homAB2} we obtain \eqref{goal:count2} upon taking $C(m,r):=m2^r(1+C(m-1,r))$. This completes the induction step to conclude the proof of \Cref{thm:count.gen}.
\end{proof}

\section{\label{sec:reg.proof}Proof of \Cref{thm:reg} (decomposition lemma) }

Throughout this section we write $\Test:=\Test_\Basys$.
For $A\in \cAn$ we denote the centered tensor
\[
\bar{A}=A-\mathbb{E}{\bA} = A - p\jay.
\]
We refer the reader to \Cref{sec:notation} for our notational conventions for tensors.

\begin{lem}
\label{lem:bad-event-1}
Let $k\ge 1$ and $\test_1,\dots,\test_k\in\Test$.
For $1\le i\le k$ let $W_i$ be the span of $\{\test_1,\dots,\test_{i}\}$ and set
\begin{equation}	\label{def:bartest}
\htest_i :=  P_{W_{i-1}^\perp} (\test_i)
\end{equation}
(with $\htest_1={\test}_1$). 
We have
\begin{align*}
&\mathbb{P}\,\Big(\,\bigwedge_{i\in [k]} \big| \langle\bar{{\bA}} \,,\, \htest_{i} \rangle  \big|
\ge\eps p\|\test_i\|_\Basys\, \Big)
\le2^{k}\exp\Big(-c(r)\eps^{2}p^{2}(1\vee\log(1/p))\sum_{i=1}^{k}\|\test_i\|_\Basys\Big)
\end{align*}
for some $c(r)>0$ depending only on $r$. 
\end{lem}

\begin{proof}
By the union bound, 
\begin{align*}
\mathbb{P}\,\Big(\,\bigwedge_{i\in [k]} \big| \langle\bar{{\bA}} \,,\, \htest_i \rangle  \big|
\ge\eps p\|\test_i\|_\Basys\, \Big)
 \le2^{k}\mathbb{P}\,\Big( \,\bigwedge_{i\in [k]}  \langle\bar{{\bA}},\sigma_{i}\htest_i\rangle \ge\eps p\|\test_i\|_\Basys\,\Big),
\end{align*}
where $\sigma_{i}\in\{+1,-1\}$. Fix a choice of $\sigma_{1},\dots,\sigma_{k}$
and let $\widetilde{\test}_i=\sigma_{i}\htest_i$.
Then 
\[
\mathbb{P}\, \Big( \,\bigwedge_{i\in [k]} \langle\bar{{\bA}},\sigma_{i}\htest_i\rangle \ge\eps p\|\test_i\|_\Basys\,\Big)
\le\mathbb{P}\,\Big( \,\langle\bar{{\bA}},\sum_{i=1}^{k}\wt{\test}_i\rangle \ge\eps p\sum_{i=1}^{k}\|\test_i\|_\Basys\Big).
\]
Since $\wt{\test}_1,\dots,\wt{\test}_k$ are orthogonal we have
\begin{equation}
\|\sum_{i=1}^{k}\wt{\test}_i\|_{2}^{2}=\sum_{i=1}^{k}\|\wt{\test}_i\|_{2}^{2}\le\sum_{i=1}^{k}\|\test_i\|_{2}^{2}=\sum_{i=1}^{k}\|\test_i\|_{1},\label{eq:bound-\test}
\end{equation}
where the last equality uses that the $\test_i$ are Boolean tensors. 

Let $\wt{\test}=\sum_{i=1}^{k}\wt{\test}_i$. Then, for any $\lambda>0$,
\[
\mathbb{P}\, \Big(\,\langle\bar{{\bA}},\sum_{i=1}^{k}\wt{\test}_i\rangle \ge\eps p\sum_{i=1}^{k}\|\test_i\|_\Basys\,\Big)\le\exp\Big(-\lambda\eps p\sum_{i=1}^{k}\|\test_i\|_\Basys\Big)\mathbb{E}\exp\Big(\lambda\langle\bar{{\bA}},\wt{\test}\rangle \Big).
\]
Recall that the entries of $\bar{{\bA}}$ with distinct coordinates
are independent centered $\textrm{Bernoulli}(p)$ random variables,
up to the symmetry constraint. Let $\bar{{\bA}}':\{n\}^{r}/S_{r}\to\mathbb{R}$
be the independent entries of $\bar{{\bA}}$. Notice that 
\[
\langle\bar{{\bA}},\wt{\test}\rangle =\langle\bar{{\bA}}',\wt{\test}'\rangle 
\]
for some tensor $\wt{\test}'$ in which each coordinate is a sum of
at most $r!$ entries of $\wt{\test}$. We thus have $\|\wt{\test}'\|_{2}^{2}\le r!\|\wt{\test}\|_{2}^{2}$.
Hence, 
\begin{align*}
\mathbb{E}\exp(\lambda\langle\bar{{\bA}},\wt{\test}\rangle ) & =\mathbb{E}\exp(\lambda\langle\bar{{\bA}}',\wt{\test}'\rangle )  =\prod_{{\bf i}\in[n]^{r}/S}\left(pe^{\lambda(1-p)\wt{\test}'({\bf i})}+(1-p)e^{-\lambda p\wt{\test}'({\bf i})}\right).
\end{align*}
We claim that 
\[
p^{\lambda(1-p)x}+(1-p)e^{-\lambda px}\le e^{\lambda^{2}x^{2}/\log(1/p)}.
\]
Indeed, we have 
\begin{align*}
pe^{\lambda(1-p)x}+(1-p)e^{-\lambda px} & \le\exp(-\lambda px+p[e^{\lambda x}-1]) \le\exp(\lambda^{2}x^{2}/\log(1/p))
\end{align*}
assuming $|\lambda x|\le\log(1/p)$, since for $|z|\le\log(1/p)$,
we have $\exp(z)\le1+z+z^{2}/(p\log(1/p)^{2})$ by monotonicity of
the function $z\mapsto\frac{\exp(z)-1-z}{z^{2}}$. Otherwise, $|\lambda x|>\log(1/p)$
and we have 
\begin{align*}
pe^{\lambda(1-p)x}+(1-p)e^{-\lambda px} & =e^{\lambda(1-p)x+\log p}+(1-p)e^{-\lambda px} \\
 & \le\exp(\lambda^{2}x^{2}/\log(1/p)).
\end{align*}
Thus, 
\[
\mathbb{E}\exp(\lambda\langle\bar{{\bA}},\wt{\test}\rangle )\le\exp\left(\lambda^{2}\|\wt{\test}'\|_{2}^{2}/\log(1/p)\right)\le\exp\left(c_{1}(r)\lambda^{2}\|\wt{\test}\|_{2}^{2}/\log(1/p)\right).
\]
By choosing $\lambda=c_{2}(r)\eps p\log(1/p)\sum_{i=1}^{k}\|\test_i\|_\Basys/\|\wt{\test}\|_{2}^{2}$,
we obtain 
\begin{align*}
 \mathbb{P}\,\Big( \, \Big\langle\bar{{\bA}},\sum_{i=1}^{k}\wt{\test}_i\Big\rangle & \ge\eps p\sum_{i=1}^{k}\|\test_i\|_\Basys\Big) \\
& \le\exp\Big( -c_{3}(r)\eps^{2}p^{2}\log(1/p)\Big(\sum_{i=1}^{k}\|\test_i\|_\Basys\Big)^{2}/\|\tilde{\test}\|_{2}^{2}\Big)\\
 & \le\exp\Big(-c(r)\eps^{2}p^{2}\log(1/p)\sum_{i=1}^{k}\|\test_i\|_\Basys\Big),
\end{align*}
using (\ref{eq:bound-\test}) and $\sum_{i=1}^{k}\|\test_i\|_{1}\le\sum_{i=1}^{k}\|\test_i\|_\Basys$. 
Moreover, the same bound with $\log(1/p)$ replaced by $1$ follows from Hoeffding's inequality.
\end{proof}

We establish \Cref{thm:reg} by the following iterative procedure.
We initialize $\resid_{0}=\bar{A}$.
If $\|\resid_0\|_{\Basys}^*\le \eps p$ then the claim follows with $k=0$.
Otherwise we proceed to step $k=1$.
At step $k\ge 1$, having obtained $\resid_0,\dots \resid_{k-1}$ and $\test_1,\dots, \test_{k-1}$,
if $\|\resid_{k-1}\|_{\Basys}^*>\eps p$ then there exists $\test_k\in{\Test}$
so that $|\langle \resid_{k-1},\test_k\rangle |>\eps p\|\test_k\|_\Basys$. Taking such a $\test_k$, we set
\[
\resid_{k}=\resid_{k-1}-P_{\Span( \htest_{k})}(\resid_{k-1})= P_{W_k^\perp} (\bar{A})\,.
\]
where for brevity we denote the subspace $W_k:=\Span(\test_1,\dots,\test_k)$.
We stop the process at step $k$ if either 
\begin{equation}	\label{stop}
\|\resid_{k}\|_{\Basys}^*\le\eps p
\qquad\text{ or }\qquad
\sum_{i=1}^{k}\|\test_i\|_\Basys
>\kappa\eps^{-2}n^{r}p^{-2},
\end{equation}
and otherwise proceed to step $k+1$.
Note that the process must stop at step $k$ for some 
\begin{equation}	\label{k.stop}
k\le k_\star := \lf 1+ \kappa\eps^{-2} p^{-\ddstar -2}\rf.		
\end{equation}
Indeed, if the process hasn't stopped after step $k-1$ for some $k\ge 1$, then
$\sum_{i=1}^{k-1}\|\test_i\|_\Basys\le \kappa\eps^{-2}n^{r}p^{-2}$, while on the other hand 
$\|\test_i\|_\Basys \ge \|\test_i\|_\emptyset =n^rp^{\ddstar}$ for each $1\le i\le k-1$, and \eqref{k.stop} follows by combining these bounds.

We take $\sum_{i=1}^k\alpha_i\test_i$ to be the expansion of $P_{W_k}(\bar{A})$ in the basis $\{T_1,\dots, T_k\}$.
Note that for each $1\le j\le k$, since $R_{j-1}$ is orthogonal to $T_1,\dots, T_{j-1}$, 
\[
\eps p \|T_j\|_{\Basys} <|\langle R_{j-1}, T_j\rangle |= |\langle R_{j-1}, \htest_j\rangle | \le \|R_{j-1}\|_2 \|\htest_j\|_2 \le \|A\|_2\|\htest_j\|_2
\]
(recalling the notation \eqref{def:bartest}), and so
the distance of $T_j$ to the span of $\{T_1,\dots,T_{j-1}\}$ is 
\[
\|\htest_j\|_2 \ge \frac{\eps p\|T_j\|_{\Basys}}{\|A\|_2} \ge \frac{\eps p n^rp^{\ddstar}}{n^{r/2}} =\eps p^{1+\ddstar}n^{r/2} 
\]
as claimed.

If the process stops at some $k$ for which 
$\sum_{i=1}^{k}\|\test_i\|_\Basys\le \kappa\eps^{-2}n^{r}p^{-2}$, then the first condition in \eqref{stop} holds, i.e.
\[
\left\| \bar{A}-P_{W_k}(\bar{A})\right\| _{\Basys}^*
\le\eps p
\]
and we obtain the claim.
We take $\badset(\kappa,\eps)$ to be the set of $A\in \cAnr$ for which the process runs until the second condition in \eqref{stop} holds for some $k\le k_\star$.

Thus, it only remains to bound the measure of $\badset(\kappa,\eps)$. 
For the case that the process ends at step $k=1$ we obtained the desired probability bound from the second bound in \eqref{stop} and \Cref{lem:bad-event-1}, so we may henceforth assume $k\ge2$. In particular, from \eqref{k.stop} it follows that $\kappa\eps^{-2}p^{ -\ddstar-2}\ge 1$ in this case.
Denoting the event in Lemma \ref{lem:bad-event-1} by $\mathcal{E}(\test_1,\dots,\test_k)$, we have
\begin{equation}	\label{reg.goal1}
\badset(\kappa,\eps) \subseteq 
\bigcup_{\substack{ \test_1,\dots,\test_k:\\ \sum_{i=1}^{k}\|\test_i\|_\Basys>\kappa\eps^{-2}n^{r}p^{ -2} }}\mathcal{E}(\test_1,\dots,\test_k).
\end{equation}
By Lemma \ref{lem:bad-event-1}, for each fixed sequence $\test_1,\dots,\test_k$,
\begin{align}	\label{E1k.bd}
\mathbb{P}\big( \mathcal{E}(\test_1,\dots,\test_k)\big) & \le2^{k}\exp\Big(-c(r)\eps^{2}p^{2}(1\vee\log(1/p))\sum_{i=1}^{k}\|\test_i\|_\Basys\Big).
\end{align}
We break up the union on the right hand side of \eqref{reg.goal1} into dyadic ranges for $\sum_{i=1}^{k}\|\test_i\|_\Basys$. For each $j\ge 0$ let
\[
\badset_j(\kappa,\eps):= \bigcup_{\substack{ \test_1,\dots,\test_k:\\ \sum_{i=1}^{k}\|\test_i\|_\Basys\in I_j }}\mathcal{E}(\test_1,\dots,\test_k).
\]
where $I_j :=\kappa\eps^{-2}n^{r}p^{ -2} \cdot [2^j, 2^{j+1})$.
Writing $\test_i=\prod_{\base\in \Base} \ttest^{(i)}_\base\circ\pi_\base$ as in \eqref{test.base}
we have that for all $\base\in\Base$, 
\[
\|\ttest_{\base}^{(i)}\|_{1}\le\|\test_i\|_\Basys/\big( n^{r-|\base|}p^{\ddstar-\dd_\base}\big).
\]
The number of choices for the Boolean tensor $\ttest_{\base}^{(i)}$ given $\|\test_i\|_\Basys$
is thus at most 
\begin{equation}
\label{bd:factors}
n^{|\base| \|\ttest_{\base}^{(i)}\|_1}
\le \exp\Big(\frac{|\base|\cdot\|\test_i\|_\Basys(\log n)}{n^{r-|\base|}p^{\ddstar-\dd_\base}}\Big)
\end{equation}
and so the number of choices for $\test_i$ given $\|\test_i\|_\Basys$
is at most 
\[
\exp\Big(r2^{r}(\log n)\cdot\|\test_i\|_\Basys
\cdot \max_{\base\in \Base}\big\{n^{|\base|-r}p^{\dd_\base-\ddstar}\big\}
\Big).
\]
Since each $\|\test_i\|_\Basys$ can take at most $O_r(b)$ different values in an interval $[a,b]$,
the total number of choices of $\test_1,\dots,\test_k$ with
$\sum_{i=1}^{k}\|\test_i\|_\Basys\in I_j$
is at most 
\begin{align*}
 & \sum_{z_{1}+\dots+z_{k}\in I_j}\exp\Big(r2^{r}(\log n)\cdot \big(\sum_{i=1}^{k}z_{i}\big)
 \cdot \max_{\base\in \Base}\big\{n^{|\base|-r}p^{\dd_\base-\ddstar}\big\}
 \Big)\\
  & \le\exp\Big(2^{j+1+2r}\kappa\eps^{-2}(\log n)p^{ -\ddstar-2}
  \max_{\base\in \Base}\big\{n^{|\base|}p^{\dd_\base}\big\}
 +O_r\big(k\log(2^{j+1}\kappa\eps^{-2}n^{r}p^{ -2})\big)\Big)\\
 &=\exp\Big(O_r(2^{j})\cdot \kappa\eps^{-2}(\log n)p^{ -\ddstar-2}
  \cdot \max_{\base\in \Base}\big\{n^{|\base|}p^{\dd_\base}\big\}\Big)\,,
\end{align*}
where we used that 
\[
\max_{\base\in \Base}\big\{n^{|\base|}p^{\dd_\base}\big\}
\ge n^{|\emptyset|}p^{\dd_\emptyset}=1
\]
along with \eqref{k.stop} to absorb the errors depending on $k$ (recall that we reduced to the case $\kappa\eps^{-2}p^{ -\ddstar-2}\ge1$, and note that $\kappa\eps^{-2} =O(n)$ from our assumptions).
Combining with \eqref{E1k.bd}, our assumption \eqref{def:WB}, and taking the constant $C_0=C_0(r)$ there sufficiently large, we obtain
\begin{align*}
& \pr ( \badset_j(\kappa,\eps)) \\
& \le\exp\Big(O_r(2^{j})\cdot \kappa\eps^{-2}(\log n)p^{ -\ddstar-2}
  \cdot \max_{\base\in \Base}\big\{n^{|\base|}p^{\dd_\base}\big\}
-c(r)2^j \kappa n^r (1\vee\log(1/p))\Big)\\
& \le \exp \big( - c(r) 2^j \kappa n^r(1\vee\log(1/p))\big)
\end{align*}
for a modified constant $c(r)>0$. Summing the above bound over $j$ and combining with \eqref{reg.goal1} and the union bound, this
completes the proof of \Cref{thm:reg}.

\section{Proof of \Cref{thm:LDP} (quantitative \abbr{LDP})}
\label{sec:LDP.proof}

\subsection{Proof of the \ldp upper bound}

In this section we prove \Cref{thm:LDP}(a).
We use the following non-asymptotic \ldp upper bound for convex sets, which holds in wide generality, and is a simple consequence of the minimax theorem -- see \cite[Exercise 4.5.5]{dz}.

\begin{lem}
\label{lem:upperLDP.convex}
For a Borel probability measure $\mu$ on a topological vector space $\cV$, and any convex, compact subset $\cB\subset\cV$, we have
\begin{equation}	\label{upperLDP.convex}
\mu(\cB) \le \exp\Big( -\inf_{Q\in \cB} \Lambda^*(Q)\Big)
\end{equation}
where $\Lambda^*:\cV\to\R_+$ is the convex dual of the log-moment generating function $\Lambda(\lambda)=\log\int_{\cV} e^{\lambda(Q)} d\mu(Q)$ on the dual vector space $\cV^*$.
\end{lem}

For the case of the \ER measure $\mu_p$ on the vector space $\Symsnr$ of real symmetric $r$-tensors with zero diagonals one checks that $\Lambda^*(Q) = \eye_p(Q)$. 

We commence with the proof of \Cref{thm:LDP}(a). 
For $t\ge0$ and sequences $\bs{\test}=(\test_1, \dots, \test_k)\in \Test(e)^k$ and $\bs{\lam}=(\lam_1,\dots, \lam_k)\in \R^k$ we let 
$\convex_e(\bs{\test}, \bs{\lam}; t)$ be the convex hull of all $A\in \cAnr$ such that
\begin{equation}	\label{Atball}
\Big\| A-\mathbb{E}{\bA}-\sum_{i=1}^{k}\lambda_{i}\htest_{i}\Big\| _{\Basys(e)}^*\le t\,,
\end{equation}
where we recall from \eqref{def:bartest} the notation $(\htest_1,\dots, \htest_k)$ for the associated orthogonal sequence. 
For each $e\in \Edges$, let $\cover_e$ be the collection of all sets of the form $\convex_e(\bs{\test},\bs{\lam}; 2\eps p)$ for some 
$1\le k\le \lf 1+ \kappa\eps^{-2} p^{-\ddstar(e) -2}\rf$,
some $\bs{\test}=(\test_1,\dots, \test_k)\in \Test(e)^k$ and some
$\bs{\lam}$ in the scaled integer lattice $\Lambda^k:=(\eps p^{1+\ddstar(e)/2}/k)\cdot \Z^k$ such that
\begin{equation}	\label{testlam}
\sum_{i=1}^k \|\test_i\|_{\Basys(e)} \le \kappa\eps^{-2}n^r p^{-2}
\qquad\text{and}\qquad 
\|\bs\lam\|_\infty\le p^{-1-\ddstar(e)}\eps^{-1}.
\end{equation}
We claim that for each $e\in \Edges$, 
\begin{equation}	\label{cover.goal1}
\pr\,\Big\{ \bA\notin \bigcup_{\convex\in \cover_e} \convex \Big\} \le \exp( -c_0\kappa n^r  \log(1/p))
\end{equation}
with $c_0>0$ as in \Cref{thm:reg}.
Indeed, it suffices to show that $\cover_e$ covers the complement in $\cAnr$ of the exceptional set $\badset_e(\kappa,\eps)$ provided by the application of  \Cref{thm:reg} with weighted base $\Basys(e)$. To that end, fix an arbitrary $A\in \cAnr\setminus \badset_e(\kappa,\eps)$. 
From \Cref{thm:reg} we have that $A$ satisfies \eqref{Atball} with $t=\eps p$ for some $\bs \test\in \Test(e)^k$ and  $\bs\lam\in \R^k$, with $\|\htest_j\|_2\ge \eps p^{1+\ddstar(e)}n^{r/2}$ for each $1\le j\le k$, where we henceforth write $\ddstar(e),\dd_\base(e)$ for the weights associated to the weighted base $\Basys(e)$.
It follows from the Cauchy--Schwarz inequality that
\[
|\lambda_{j}|=
\frac{|\langle P_{W_j}(\bar{A}),\htest_j\rangle |}{\|\htest_j\|_{2}^{2}}
\le\frac{\|\bar{A}\|_{2}}{\|\htest_j\|_{2}} \le \eps^{-1} p^{-1-\ddstar(e)} \,,
\]
so $\|\bs\lambda\|_\infty \le \eps^{-1}p^{-1-\ddstar(e)}$.
Now let $\bs\lam'\in \Lambda^k$ be as in \eqref{testlam} with $\|\bs\lam-\bs\lam'\|_\infty \le \eps p^{1+\ddstar(e)/2}/k$.
By an application of the triangle inequality for the $\|\cdot\|_{\Basys(e)}^*$ seminorm, we only need to show
\begin{equation}	\label{cover.goal2}
\|\htest_{i}\|_{\Basys(e)}^* \le p^{-\ddstar(e)/2}
\end{equation}
for each $1\le i\le k$. 
For this, note that 
$\|\htest_{i}\|_2\le \|\test_i\|_2 \le n^{r/2}$ since $\test_i$ is Boolean.
Now for
any $\tens\in\Tensnr$ with $\|\tens\|_{2}\le n^{r/2}$,
\[
\|\tens\| _{\Basys(e)}^*=\sup_{\test\in{\Test}(e)}\frac{|\langle \tens,\test\rangle |}{\|\test\|_{\Basys(e)}}\le\frac{\|\tens\|_{2}\|\test\|_{2}}{\|\test\|_{1}^{1/2}(n^{r}p^{\ddstar(e)})^{1/2}}=\frac{\|\tens\|_{2}}{(n^{r}p^{\ddstar(e)})^{1/2}}\le p^{-\ddstar(e)/2},
\]
where in the second equality we used that $\|\test\|_{2}^{2}=\|\test\|_{1}$ for Boolean $\test$. Thus we obtain \eqref{cover.goal2} and hence \eqref{cover.goal1} as desired.

Now set
\[
\cover_\Edges' = \Big\{ \bigcap_{e\in \Edges} \convex_e: \convex_e\in \cover_e \;\text{ for each } e\in \Edges\Big\}
\]
and let $\cover_\Edges$ be obtained by replacing each $\convex\in \cover_H'$ with the convex hull of $\convex\cap \cAnr$.
We claim
\begin{equation}	\label{claim:upper.RME}
\log|\cover_\Edges| \ls_{\BBasys} 
\kappa n^r\eps^{-2} \growing_{n,p}^{-1}\log n.
\end{equation}
Fixing $e\in \Edges$, it suffices to prove the claimed bound holds for $\log|\cover_e|$ (up to modification of the constant by a factor $|\Edges|$). 
First, recalling the bound \eqref{bd:factors}, the number of $\test\in \Test(e)$ with a given value of $\|\test\|_{\Basys(e)}$ is at most 
\[
\prod_{\base\in \Base(e)} n^{|\base| \|\test\|_{\Basys(e)} / (n^{r-|\base|} p^{\ddstar(e) - \dd_\base(e)})}
\]
so the total number of choices for $\bs\test$ as in \eqref{testlam} is at most
\begin{align}	\label{TT.choices}
 \exp \Big((\log n) \sum_{i=1}^k & \sum_{\base\in \Base(e)} \frac{ |\base| \|\test_i\|_{\Basys(e)} }{ n^{r-|\base|} p^{\ddstar(e) - \dd_\base(e)} } \Big) \notag \\
& \le 
\exp \Big( O_r(1) \growing_{n,p}(\Basys(e))^{-1} \kappa \eps^{-2} n^r  \log n \Big).
\end{align}
The number of choices for $k$, $\lam_1,\dots, \lam_k$ and $\|\test_1\|_{\Basys(e)}, \dots, \|\test_k\|_{\Basys(e)}$ is
\begin{align}	\label{LL.choices}
\sum_{k\le 1+ \kappa\eps^{-2} p^{-\ddstar(e)-2}} & \Big(\frac{2}{\eps^2p^{2+\frac32\ddstar(e)}}\Big)^k O_r(\kappa\eps^{-2}n^rp^{-2})^k \notag \\
& = n^{O_r(1)} n^{O_r (\kappa\eps^{-2} p^{-\ddstar(e)-2})}
\end{align}
where we noted that the bases of the exponentials in $k$ are all $n^{O_r(1)}$ by our assumptions on $n,p,\kappa$ and $\eps$. 
Now since $\growing_{n,p}(\Basys(e)) \le n^{r-|\emptyset|} p^{\ddstar(e) - d_\emptyset+2} = n^rp^{\ddstar(e)+2}$ we see that the second factor in \eqref{LL.choices} is dominated by the right hand side of \eqref{TT.choices}.
We thus obtained the claimed bound on $|\cover_e|$, establishing \eqref{claim:upper.RME}.

Fix $\cE\subseteq\cAnr$. We claim that for any $\cK\in \cI_\Edges$,
\begin{equation}	\label{claim:upper.contain}
\cK\cap\cE\ne\emptyset \;\Longrightarrow\; \cK\subseteq (\cE)_{\BBasys,4\eps p}.
\end{equation}
Indeed, fix arbitrary $\cK\in \cI_\Edges$ with $\cK\cap \cE\ne\emptyset$. It suffices to show that for any fixed $A_1,A_2\in \cK$, we have
\[
\|A_1-A_2\|_{\Basys(e)}^*\le 4\eps p\quad \forall e\in \Edges.
\]
But this is immediate from the definitions: we have $\cK=\bigcap_{e\in \Edges} \cK_e$ for some choices of $\cK_e\in \cI_e$, and each $\cK_e$ is contained in the $2\eps p$-neighborhood of some $A_e'\in \cAnr$ under $\|\cdot\|_{\Basys(e)}^*$, so the above bound follows by the triangle inequality.

Now we are ready to conclude.
For $\cF\subset\cQnr$ we abbreviate
\begin{equation}	\label{def:ip.short}
\eye_p(\cF):= \inf\{ \eye_p(Q): Q\in\cF\}.
\end{equation}
Applying the union bound and \eqref{cover.goal1} we have
\begin{align*}
\P(\bA\in \cE) 
&\le |\Edges| p^{c_0\kappa n^r} + \sum_{\cK\in \cI_\Edges: \cK\cap \cE\ne \emptyset} \P(\bA\in \cK).
\end{align*}
For the latter term we apply \eqref{upperLDP.convex} to bound
\begin{align*}
 \sum_{\cK\in \cI_\Edges: \cK\cap \cE\ne \emptyset}  & \P(\bA\in \cK) \\
 &\le \sum_{\cK\in \cI_\Edges: \cK\cap \cE\ne \emptyset} \exp( - \eye_p(\cK))
\le  |\cI_\Edges| \max_{\cK\in \cI_\Edges: \cK\cap \cE\ne \emptyset} \exp( - \eye_p(\cK))\\
&= |\cI_\Edges| \exp\Big( - \eye_p\big( \bigcup_{\cK\in \cI_\Edges: \cK\cap \cE\ne \emptyset} \cK \big) \Big)
\le |\cI_\Edges| \exp(- \eye_p((\cE)_{\BBasys,4\eps p})) \,,
\end{align*}
where in the final line we used \eqref{claim:upper.contain}.
Applying \eqref{claim:upper.RME}, the claim now follows by using $\log(a+b) \le \max(\log(2a) ,\log(2b))$.

\subsection{Proof of the \ldp lower bound \eqref{LDP.lower}}

The claim quickly follows from the next lemma:

\begin{lem}
\label{lem:tilt.Bstar}
Fix a collection $\BBasys=\{\Basys(e)\}_{e\in \Edges}$ of weighted bases,
let $p\in(0,1)$, $\eps>0$, and assume 
\begin{equation}	\label{Bstar.conc-growing}
\min_{e\in\Edges}\growing_{n,p}(\Basys(e))\ge C(r) \eps^{-2}\log n
\end{equation}
for a sufficiently large constant $C(r)>0$. Then for any $Q\in\cQnr$,
\[
\log\mu_p(\nbhd_{\BBasys}(Q,\eps p)) \ge -\eye_p(Q) -O\Big(1+ n^{r/2}\Lp \Big).
\]
\end{lem}

Indeed, for any $\cE\subseteq\cAnr$ and any $Q\in(\cE)^\circ_{\BBasys,\eps p}$ we have $\P(\bA\in\cE) \ge \mu_p(\nbhd_{\BBasys}(Q,\eps p))$ by monotonicity, and the claim follows from the above lemma and taking the supremum over such $Q$.

It only remains to establish \Cref{lem:tilt.Bstar}. To this end we use the following two lemmas,
starting with a complementary lower bound for \Cref{lem:upperLDP.convex}.
\begin{lem}
\label{lem:tilt}
For $d\ge1$ and $p\in(0,1)$ let $\mu_p$ be the product Bernoulli($p$) measure on $\{0,1\}^d$ and let $\nu$ be any other product measure on $\{0,1\}^d$. For any $\cE\subseteq\{0,1\}^d$,
\[
\log\mu_p(\cE) \ge -\DKL(\nu\|\mu_p) + \log\nu(\cE) - O\bigg( \frac{d^{1/2} \Lp} {\nu(\cE)^{1/2}}\bigg).
\]
\end{lem}
\begin{proof} While \cite[Lemma 6.3]{CoDe} is stated in terms of two adjacency matrices of simple graphs of $d={n \choose 2}$ edges, one drawn from $\mu_p$ for some 
$p \in (0,\frac{1}{2}]$ and the other from some product measure $\nu$ on $\{0,1\}^d$, 
this lemma and it's proof apply to any value of $d$ and to any such product measures of Bernoulli 
variables. Further, with $p \le \frac{1}{2}$ used only for the elementary bound of \cite[(6.7)]{CoDe}, upon changing there $p$ to $(1-p)$, the same proof applies also for $p \in (\frac{1}{2},1)$.
\end{proof}

\begin{lem}
\label{lem:Bstar.conc}
Fix a weighted base $\Basys=(r,e,\iota,\Base,\ddstar,\{\dd_\base\}_{\base\in\Base})$, let $p\in(0,1)$, $\eps>0$. If $\growing_{n,p}(\Basys)\ge C(r) \eps^{-2}\log n$ for a sufficiently large constant $C(r)>0$, then 
\[
\mu_Q(\{ A\in\cAnr: \|A-Q\|_{\Basys}^*\ge \eps p\}) 
\le \exp(-c \eps^2p^{\ddstar+2}n^r).
\]
\end{lem}

\begin{proof}
Let $\P_Q$ be a probability measure under which $\bA$ has distribution $\mu_Q$, and let $\E_Q$ be the associated expectation. 
From Hoeffding's inequality we get that for any $T\in\Test_\Basys$,
\[
\P_Q(|\langle \bA-Q,T\rangle|>\eps p \|T\|_{\Basys}) 
\le 2\exp\bigg( -\frac{c\eps^2p^2\|T\|_{\Basys}^2}{\|T\|_2^2}\bigg)
\le 2\exp( -c'\eps^2p^2\|T\|_{\Basys}).
\]
Recalling \eqref{bd:factors}, by the union bound we have that for any $L>0$,
\begin{align*}
&\P_Q\Big( \exists T\in \Test_\Basys: \|T\|_{\Basys} =L\,, \; |\langle \bA-Q,T\rangle|>\eps p \|T\|_{\Basys} \Big)\\
&\qquad\qquad\qquad\le 2\exp\Big( O_r(\log n) \max_{\base\in\Base} \{ n^{|\base|-r} p^{\dd_\base-\ddstar}\} L  - c'\eps^2p^2L\Big)\\
&\qquad\qquad\qquad\le 2\exp(-(c'/2) \eps^2p^2L)
\end{align*}
assuming $C(r)$ in our assumption \eqref{Bstar.conc-growing} is sufficiently large. 
Since the set of possible values for $\|T\|_{\Basys}$ is contained in a union of $|\Base|+1\le 2^r$ arithmetic progressions in $[n^rp^{\ddstar},n^r]$ of step at least 1, and $n^rp^{\ddstar+2} \ge \growing_{n,p}(\Basys)$, the claim follows from another union bound and summing the at most $2^r$ geometric series (and we can take $C(r)$ larger if necessary to absorb the prefactor $2^r$).
\end{proof}

\begin{proof}[Proof of \Cref{lem:tilt.Bstar}]
From \Cref{lem:tilt}, identifying $\cAnr$ with $\{0,1\}^{{n\choose r}}$, we get that for any $\cE\subseteq\cAnr$ and $Q\in\cQnr$,
\[
\log\mu_p(\cE) \ge -\eye_p(Q) + \log\mu_Q(\cE) - O\bigg( \frac{n^{r/2}\Lp}{\mu_Q(\cE)^{1/2}}\bigg).
\]
Taking $\cE =\bigcap_{e\in \Edges}\cE_e$ with $\cE_e=\{A\in \cA_{n,r}: \|A-Q\|_{\Basys(e)}^*< \eps p\}$, from \Cref{lem:Bstar.conc} we have $\mu_Q(\cE_e)\ge1-1/(2|\Edges|)$ if $C(r)$ is sufficiently large. From the union bound we get that $\mu_Q(\cE) \ge 1/2$, and the claim follows upon substituting this estimate in the above display.
\end{proof}

\section{\label{sec:tails.upper} Proof of \Cref{thm:tails} -- the upper tail}

In this section we establish \eqref{UT.LB}. We also show how \eqref{UT.UB} can be established along similar lines under some alternative assumptions on $p$ and $\uH$. The proof of \eqref{UT.UB} under the assumption $p\gg n^{-1/\Delta}$ is quite different and is given in $\mathsection$\ref{sec:tails.lower}.

To lighten notation, we will first present the proof of \eqref{UT.LB} for the case $m=1$ in $\mathsection$\ref{sec:UTLB1}, and then describe in $\mathsection$\ref{sec:UTLBm} the simple modifications that are needed for the general case.

\subsection{Upper bound for the upper tail probability (case $m=1$)}
\label{sec:UTLB1}

In this subsection we prove the following proposition, yielding \eqref{UT.LB} for the case $m=1$.

\begin{prop}
\label{prop:upper}
For any $r$-graph $H$ and $\delta,\xi>0$, assuming
\begin{equation}\label{UB.np-LB1}
np^{\DDprime(H)}\log(1/p)>C_2\xi^{-3}\log n
\end{equation}
for sufficiently large $C_2(H,\delta)>0$, 
we have 
\begin{equation}	\label{UB.bound1}
\log\P\big(\,t_p(H,\bA)>1+\delta\,\big)
\le -(1-\xi)\Phi_{n,p}(H,\delta-\xi).
\end{equation}
\end{prop}

We need the following lemma, which one obtains by the same lines as in \cite[Theorem 2.2]{LiZh} (for the lower bound they do not use the stated assumption that $p\gg n^{-1/\Delta(H)}$, and their upper bound assumption on $p$ is also not necessary, as one can use \Cref{lem:ipbelow} below in place of their Lemma 4.7 to get sufficient on control $\eye_p$). 

\begin{lem}	\label{lem:Phi.LB}
For any $n\in\N$, $p\in (0,1)$, $u>0$, and $r$-graph $H$ with $\Delta(H)\ge2$,
\begin{equation*}
\Phi_{n,p}(H,u)\gs_{H,u} n^rp^{\Delta(H)}\log(1/p)\,.
\end{equation*}
Moreover, for $u\ge1$,
\begin{equation*}	
\Phi_{n,p}(H,u)\gs_H u^{\Delta(H)/\edges(H)} n^rp^{\Delta(H)}\log(1/p)\,.
\end{equation*}
\end{lem}

In the remainder of this subsection we set $\Delta:=\Delta(H)$ and denote
\begin{equation}	\label{def:Rnp}
R_{n,p}:= n^rp^\Delta\log(1/p).
\end{equation}
The following claim, giving the tail probability for the event under
which we can apply \Cref{thm:count}, will be proved together
with \Cref{prop:upper} by induction on the number of edges
in $H$. 

\begin{claim}
\label{claim:crude-upper}
There exists $C_2'(H)>0$ such that for any $L>C_2'(H)$, if
\begin{equation}	\label{UB.np-LB2}
np^{\Delta'(H)}\log(1/p) \ge C_2'L^3\log n
\end{equation}
then if $\Delta\ge2$, 
\begin{equation}	\label{UB.bound2}
\P\big(\, t_p(H,\bA)\ge L\,\big)
\le\exp\left(-cL^{1/\edges(H)} R_{n,p} \right) 	
\end{equation}
and when $\Delta=1$, 
\begin{equation}	\label{UB.boundD1}
\P\big(\, t_p(H,\bA)\ge L\,\big)
\le\exp\left(-cL^{1/\edges(H)} n^rp\log L \right)\,.
\end{equation}
\end{claim}

The key to the proof of \Cref{prop:upper} is that  \Cref{thm:count} only requires control over \emph{proper} subgraphs of $H$ as given by \Cref{claim:crude-upper}, and thus can be guaranteed inductively.
(Observe the conclusion of \Cref{claim:crude-upper} can be essentially obtained from the conclusion of \Cref{prop:upper} and the crude upper bound in \Cref{lem:Phi.LB}; \nick{however, the constant $C_2'$ in \Cref{claim:crude-upper} is independent of $L=1+\delta$, which will be important for closing the inductive argument}.)

\begin{proof}[Proof of \Cref{prop:upper}]
 We proceed by induction on the number of edges in $H$. 
For the case $\edges(H)=1$,  \Cref{claim:crude-upper} follows from a standard tail bound for the binomial distribution, and  \Cref{prop:upper} 
follows immediately from \Cref{lem:upperLDP.convex} as the super-level set in this case is a (convex) half space.
From the fact that $t_p(H_1\cup H_2, Q) = t_p(H_1,Q)t_p(H_2,Q)$ for $H$ a disjoint union of two graphs $H_1,H_2$ we further obtain the case $\Delta=1$ for \Cref{prop:upper} and \Cref{claim:crude-upper}.

Assume now that $\Delta\ge2$ and that \Cref{prop:upper} and \Cref{claim:crude-upper}
hold when
$\edges(H)\le \ell-1$ for some $\ell\ge2$. Consider now an $r$-graph $H$ with $\edges(H)=\ell$.

Recalling \eqref{def:dprime.dom}, for each $e\in \Edges(H)$ we fix a dominating weighted base $\Basys(e)$ as in \eqref{Deg.usual} with $\ddprime_{\Dom(e)}(e) \le \DDprime(H_k)$, and form the base system $\BBasys=\{\Basys(e)\}_{e\in\Edges(H)}$. 
This implies
\begin{equation}	\label{UB.WLB}
\growing_{n,p}(\BBasys) \ge \min_{e\in \Edges(H)} \min_{\base\in \Base(e)} (np^{\Delta'(H)})^{r-|\base|} \ge np^{\Delta'(H)}
\end{equation}
since by definition we have $|\base|<r$ for any element $\base$ of any base (note that we may alternatively express $\ddprime_\dom(e)$ in \eqref{def:dprime.dom} as $(\dd^H(e) -\dd_\dom^H(e) + 2)(r-|\dom|)$). 

We denote sub-level sets by
\begin{equation}	\label{def:subL}
\cL_H( L) := \{ Q\in\cQnr: 
t_p(H,Q) \le L \}\,,\qquad L>0
\end{equation}
and additionally denote
\begin{equation}	\label{def:subLsub}
\cL_{<H}( L) := \bigcap_{F\subsetneq H} \cL_{F}( L)\,.
\end{equation}
With $C_2,C_2'$ to be determined over the course of the proof, consider for now arbitrary $\delta,\xi>0$, and additional parameters $L_0>C_2'(H)$, $K\ge1$ and $\eps>0$.
We may assume $\xi<\delta/2$.
Taking $C_2'(H)\ge \max_{F\subsetneq H}C_2'(F)$, from the induction hypothesis and the union bound we have
\begin{equation}
\P\Big( \bA \notin \cL_{<H}( L_0) \Big) \ls_H \exp(-c L_0^{1/\edges(H)} \rate_{n,p}).
\end{equation}
(Here we used our assumption $\Delta\ge2$: note that for $F\subset H$ of max-degree 1, while \eqref{UB.boundD1} loses a factor $\log(1/p)$ in the exponent as compared to \eqref{UB.bound2}, we gain a factor $p^{1-\Delta}$.)
Now set
\[
\cE := \cA_{n,r}\cap\cL_H( 1+\delta)^c \cap  \cL_{<H}(L_0).
\]
From the previous bound and \Cref{thm:LDP} we have
\begin{align*}
 \P(\bA\in \cL_H(1+\delta)^c)  
\ls_H & \exp(-cL_0^{1/\edges(H)}\rate_{n,p})
+ \P(\bA\in \cE)\\
\ls_H  &\exp(-cL_0^{1/\edges(H)}\rate_{n,p})
+ \exp( - c_0 K\rate_{n,p}) \\
& + \exp\Big( \frac{C_1 K \rate_{n,p} \log n}{\eps^2 np^{\Delta'(H)}\log(1/p)} - \eye_p\big( (\cE)_{\BBasys, \eps p}\big) \Big)
\end{align*}
(recall the shorthand notation \eqref{def:ip.short}). 
From \Cref{thm:count} it follows that
\[
(\cE)_{\BBasys, \eps p} \subseteq \cL_H( 1+\delta - O_H(L_0\eps))^c.
\]
and hence
\begin{equation}
\eye_p\big((\cE)_{\BBasys,\eps p}\big) \ge \Phi_{n,p}(H, \delta - O_H(\eps L_0)).
\end{equation}
Substituting this bound into the previous bound, we get that for any $K\ge 1, L_0>C_2'(H)$, $\delta\ge 2\xi>0$ and $\eps<c(H)\xi/L_0$ for $c(H)>0$ sufficiently small, 
\begin{align}
\P(\bA\in\cL_H(1+\delta)^c)
&\ls_H \exp(-cL_0^{1/\edges(H)}\rate_{n,p})
+ \exp(- c_0 K\rate_{n,p})	
\notag \\ &
+ \exp\Big( \frac{C_1 K \rate_{n,p} \log n}{\eps^2 np^{\Delta'(H)}\log(1/p)} -\Phi_{n,p}(H, \delta-\xi)  \Big)		\label{up.summary}\\
&=: \text{(I)} + \text{(II)} + \text{(III)}\,.\notag
\end{align}

Now to establish \eqref{UB.bound1} under the assumption \eqref{UB.np-LB1}, from \Cref{lem:Phi.LB} we can take $L_0,K$ sufficiently large depending on $H,\delta$ to make the terms (I) and (II) in \eqref{up.summary} negligible. 
Fixing such $L_0, K$, we can then fix $\eps=c'\xi$ for sufficiently small $c'(H,\delta)>0$, so that, together with \eqref{UB.WLB} and our assumption \eqref{UB.np-LB1}, 
by taking $C_2(H,\delta)$ sufficiently large we make the first term in the exponential of (III) at most
$
\frac{\xi}2\Phi_{n,p}(H, \delta-\xi)
$,
and \eqref{UB.bound1} follows.

For \eqref{UB.bound2}, in \eqref{up.summary} we take $L_0=2L\ge C_2'(H)$, $\delta = L-1$, $\xi=1/2$ (say), $K=2L^{1/\edges(H)}$, $\eps=c''/10L$ for sufficiently small $c''(H)>0$, and combining with \eqref{UB.WLB} and \eqref{UB.np-LB2}, the claim then follows upon taking $C_2'(H)$ sufficiently large. 
\end{proof}

\subsection{Upper bound for the upper tail probability (general case)}
\label{sec:UTLBm}

For the case of general $m\in\N$ we follow similar lines as in $\mathsection$\ref{sec:UTLB1} with some minor modifications.
The proof is now by induction on $\ell:=\max_k\edges(H_k)$. The case $\ell=1$ is handled exactly as before (the half-spaces being intersected have parallel boundaries). For $\ell\ge2$, we fix a dominating weighted base $\Basys(e)$ for each $k\in[m]$ and $e\in \Edges(H_k)$, and \eqref{UB.WLB} now holds with the minimum now taken over all edges in all graphs $H_k$.
In place of \eqref{def:Rnp} we now take
\begin{equation}	\label{def:Rnp.m}
R_{n,p}:=n^r p^{\min_k\Delta(H_k)}\log(1/p).
\end{equation}
The bounds in \Cref{lem:Phi.LB} extend to $\Phi_{n,p}(\uH,\udelta)$ by restricting the infimum to a single superlevel set $\{t_p(H_k,\cdot)\ge 1+\delta_k\}$ for which $\Delta(H_k)= \min_\ell\Delta(H_\ell)$. 
We apply the upper-LDP bound of \Cref{thm:LDP}(a) with 
\[
\cE := \cA_{n,r}\cap\cL_{\uH}( \1+\udelta)^c \cap  \cL_{<\uH}(L_0).
\]
where for $\uL\in\R_+^m$,
\begin{equation}	\label{def:subLint}
\cL_{\uH}(\uL):= \bigcap_{k\in[m]} \cL_{H_k}(L_k)\,,
\qquad
\cL_{<\uH}( L_0) := \bigcap_{k\in[m]}\bigcap_{F\subsetneq H_k} \cL_{F}( L_0)\,.
\end{equation}
For this choice of $\cE$, \Cref{thm:count} implies $(\cE)_{\BBasys, \eps p} \subseteq \cL_{\uH}(\1+\udelta-O_{\uH}(L_0\eps))^c$, and the rest of the argument proceeds as before. \qed

\subsection{Lower bound for the upper tail probability}
\label{sec:UTUB}

In this subsection we show how the bound \eqref{UT.UB} easily follows from Theorems \ref{thm:LDP} and \ref{thm:count} under some alternative assumptions.
As in the proof of \eqref{UT.LB} we present only the case $m=1$ to lighten notation, but the argument extends to general $m$ in a straightforward way, following similar modifications as in $\mathsection$\ref{sec:UTLBm}.
The proof assuming $p\gg n^{-1/\Delta(H)}$ is quite different and is given in \Cref{sec:tails.lower}.

\begin{prop}	\label{prop:UTUB}
In the setting of \Cref{thm:tails}, the bound \eqref{UT.UB} holds if $p\gg n^{-\min(\frac1{\DDprime(H)},\frac{r}{2\Delta(H)})}\log n$ and if we further assume either (a) $p\ge n^{-1/\Delta(H)}$ or (b) $H$ is regular.
\end{prop}

This is an immediate consequence of \Cref{lem:Phi.LB} and the following, together with the assumption $p\gg n^{-1/\DDprime(H)}\log n$ (via \eqref{growing.LB}).

\begin{lem}	\label{lem:UTUB}
Let $H$ be an $r$-graph and let $\delta>\xi>0$. 
Suppose there exists an $H$-dominating base system $\BBasys=\{\Basys(e)\}_{e\in\Edges(H)}$ 
with weights as in \eqref{Deg.usual}
such that $\min_{e\in\Edges(H)} \growing_{n,p}(\Basys(e))\ge C(H,\delta)\xi^{-2}\log n$ for a sufficiently large constant $C(H,\delta)>0$. Then if $p\ge n^{-1/\Delta(H)}$, or if $p\ge n^{-r/\Delta(H)}$ and $H$ is regular, we have
\[
\log\P(t_p(H,\bA)\ge1+\delta) \ge -\Phi_{n,p}(H,\delta+\xi) - O\Big(1+ n^{r/2}\Lp\Big).
\]
\end{lem}

(To deduce case (b) in \Cref{prop:UTUB} note that $n^{-1/\DDprime(H)}\ge n^{-r/\Delta(H)}$ from \eqref{DDprime.range}, and the assumption $p\gg n^{-r/(2\Delta(H))}$ implies the error term in \Cref{lem:UTUB} is $o(\rate_{n,p})=o(\Phi_{n,p}(H,\delta))$.) 

\begin{proof}
Setting $\cE:=\cAnr\cap \cL_H(1+\delta)^c$ and letting $\eps=c\xi$ for $c=c(H,\delta)>0$ to be taken sufficiently small, from \Cref{thm:LDP}(b) we have
\begin{equation}
\P(t_p(H,\bA)\ge1+\delta) \ge -\eye_p ((\cE)_{\BBasys,\eps p}^\circ) - O\Big(1+ n^{r/2} \Lp\Big)
\end{equation}
assuming $C(H,\delta)$ is sufficiently large depending on $c$ and the constant $C_0'(r)$ from \Cref{thm:LDP}(b). It only remains to show
\begin{equation}	\label{UTLB1}
\eye_p ((\cE)_{\BBasys,\eps p}^\circ)\le \Phi_{n,p}(H,\delta+\xi)
\end{equation}
if $c$ is sufficiently small. 
Letting $1\le K=O_{H,\delta}(1)$ to be chosen later, we first argue that
\begin{equation}	\label{UTLB2}
(\cE)_{\BBasys,\eps p}^\circ\supseteq \cL_H(1+\delta+\xi)^c \cap \cL_{<H}(K).
\end{equation}
Indeed, letting $Q_0$ be an arbitrary element of the \abbr{lhs}, for any $A\in\cU_{\BBasys}(Q_0,\eps p)$ we have from \Cref{thm:count} applied to $\cC=\cU_{\BBasys}(Q_0,\eps p)$ that 
\[
t_p(H,A) \ge t_p(H,Q_0) - O_{H,\delta}(\eps) \ge 1+\delta+\xi-O_{H,\delta}(\eps) \ge 1+\delta
\]
if $c$ is sufficiently small, and \eqref{UTLB2} follows. 
Next, by considering $Q\in\cQnr$ of the form $Q(I)=1$ whenever $I\cap[\lf L p^{\Delta(H)} n\rf]\ne \emptyset$ and otherwise $Q(I)=p$, one easily checks that 
\begin{equation}	\label{eq:upper-Phi}
\Phi_{n,p}(H,\gamma) \ls_{H,\gamma}n^rp^{\Delta(H)}\log(1/p)
\end{equation}
for any fixed $\gamma>0$, assuming $p\ge n^{-1/\Delta(H)}$ and taking $L=L(H,\gamma)$ sufficiently large. If $H$ is regular we get the same bound for $p\ge n^{-r/\Delta(H)}$ by considering $Q$ of the form $Q(I)=1$ for $I\subset[\lf Lp^{\Delta(H)/r}n\rf]$ and otherwise $Q(I)=p$. 
From \Cref{lem:Phi.LB}, we can hence fix $K$ sufficiently large depending on $H$ and $\delta$ so that
\[
\eye_p(\cL_F(K)^c) = \Phi_{n,p}(F,K-1) > \Phi_{n,p}(H,2\delta)
\]
for every subgraph $F$ of $H$. Since $ \Phi_{n,p}(H,2\delta)\ge \Phi_{n,p}(H,\delta+\xi)$, it follows that
\[
 \Phi_{n,p}(H,\delta+\xi) = \eye_p(\cL_H(1+\delta+\xi)^c) = \eye_p(\cL_H(1+\delta+\xi)^c\cap \cL_{<H}(K))
\]
and \eqref{UTLB1} follows upon combining the above with \eqref{UTLB2}.
\end{proof}

\section{Proof of \Cref{thm:tails} -- lower tail}
\label{sec:LT}

Recall our notation $\cL_H(u)=\{ Q\in\cQnr: t_p(H,Q)\le u\}$, and $\cL_{<H}(u)=\bigcap_{F\subset H} \cL_F(u)$. We will make frequent use of the shorthand notation \eqref{def:ip.short}.
We need the following elementary estimate for the lower tail optimization problem \eqref{def:Psi}; for a proof see \cite[Lemma 22]{KoSa}.

\begin{lem}\label{lem:Psi.size}
For any $r$-graph $H$ and $\delta\in(0,1)$, 
$
 \Psi_{n,p}(H,\delta) =\eye_p(\cL_H(1-\delta)) 
\gs _{H,\delta} n^r p\,.
$
\end{lem}

Since $\Psi_{n,p}(\uH,\udelta)\ge \Psi_{n,p}(H_k,\delta_k)$ for any $k\in [m]$ it immediately follows that $\Psi_{n,p}(\uH,\udelta) \gs_{\uH,\udelta}n^rp$ for any $\uH$ and $\udelta\in(0,1)^m$.

\subsection{Upper bound on the lower tail probability}

As in the proof of the upper bound for the upper tail, we will combine the upper-\abbr{ldp} of \Cref{thm:LDP} with the counting lemma (\Cref{thm:count}). 
However, whereas for the upper tail we needed to argue by induction on the number of edges in $H$ in order to satisfy the crude bound assumption \eqref{assu:count2.0} for subgraphs, here we can reduce to this assumption with a simple application of the FKG inequality:

\begin{lem}	\label{lem:LT.FKG}
Recalling the notation \eqref{def:subLint}, there exists $C=C(\uH)>0$ such that 
\[
\P( \bA\in\cL_{\uH}(\1-\udelta)) \le 2\P(\bA\in \cL_{\uH}(\1-\udelta)\cap \cL_{<\uH}(C)).
\]
\end{lem}

\begin{proof}
Since the sets $\cL_{\uH}(\1-\udelta)\}$ and $\cL_{<\uH}(C)\}$ are monotone subsets of the hypercube, we have by the FKG inequality that
\begin{align*}
\P(\bA\in & \cL_{\uH}(\1-\udelta)) \\
&= \P(\bA\in\cL_{\uH}(\1-\udelta)\cap\cL_{<\uH}(C)) + \P(\bA\in\cL_{\uH}(\1-\udelta), \bA\notin\cL_{<\uH}(C))\\
&\le \P(\bA\in\cL_{\uH}(\1-\udelta)\cap\cL_{<\uH}(C)) + \P(\bA\in\cL_{\uH}(\-\udelta))\P( \bA\notin\cL_{<\uH}(C)).
\end{align*}
Rearranging yields
\[
\P(\bA\in\cL_H(1-\delta)) 
\le \frac{ \P(\bA\in\cL_{\uH}(\1-\udelta)\cap\cL_{<\uH}(C)) }{1- \P( \bA\notin\cL_{<\uH}(C))}.
\]
Now for any $k\in[m]$ and subgraph $F\subset H_k$, from Markov's inequality we get that $\P(\bA\notin \cL_F(C)) \ls 1/C$. The claim follows by applying this together with the union bound over $F\subset H_k, k\in [m]$ and taking $C$ sufficiently large depending on $\uH$.
\end{proof}

We proceed with the proof of the upper bound in \eqref{LT.UBLB}.
Let $c=c(\uH)>0$ to be taken sufficiently small and set $\eps:=c\xi$.
By our lower-bound assumption on $p$ and \eqref{growing.LB}, we can select a collection $\BBasys=\{\Basys(e)\}_{e\in\bigcup_{k=1}^m\Edges(H_k)}$ of dominating weighted bases over the edges of each $H_k$ with weights as in \eqref{Deg.usual}, such that for any fixed $W_0=W_0(\uH,\delta,\xi)$ and all $n$ sufficiently large,
\begin{equation}
\growing_{n,p}(\Basys(e)) \ge W_0\eps^{-2}\log n.
\end{equation} 
With $C=C(\uH)$ as in \Cref{lem:LT.FKG}, we denote $\cE:= \cAnr\cap \cL_{\uH}(\1-\udelta)\cap\cL_{<\uH}(C)$.
Taking $W_0$ larger than the constant $C_0(r)$ from \Cref{thm:reg}, we can apply \Cref{thm:LDP}(a) with $\Delta=1$ and some $1\le K=O_H(1)$ to be chosen later to bound
\begin{align}	\label{LTUB.applyLDP}
\P(\bA\in \cE) \le -\min\big\{ & c_0Kn^rp\log(1/p) - O_{\BBasys}(1)\,,\; \nonumber \\
& \eye_p((\cE)_{\BBasys,\eps p}) - O_{\BBasys}( KW_0^{-1} n^rp\log(1/p)) \big\}\,.
\end{align}
Now we claim that
\begin{equation}	\label{LTUB.claim}
(\cE)_{\BBasys,\eps p} \subseteq \cL_{\uH}(\1-\udelta+\xi\1)
\end{equation}
if $c$ is sufficiently small. Indeed, letting $A_0\in\cE$ be arbitrary, applying \Cref{thm:count} with $\cC=\cU_{\BBasys}(Q,\eps p)$ we have that for each $1\le k\le m$, 
\[
t_p(H_k,Q) \le t_p(A_0) + O_{\uH}(\eps) \le 1-\delta_k+\xi
\]
if $c$ is sufficiently small, and \eqref{LTUB.claim} follows.
Thus, $\eye_p((\cE)_{\BBasys,\eps p}) \ge \Psi_{n,p}(\uH, \udelta-\xi\1)$. From \Cref{lem:Psi.size} (and the remark that follows it) we can fix $K$ sufficiently large and then fix $W_0$ sufficiently large such that the minimum in \eqref{LTUB.applyLDP} is attained by the second argument, and such that the error term in the second argument is at most $\xi\Psi_{n,p}(\uH,\udelta-\xi\1)$. The lower bound on $\LT_{n,p}(\uH,\udelta)$ in \eqref{LT.UBLB} follows. \qed

\subsection{Lower bound for the lower tail probability}

The upper bound on $\LT_{n,p}(H,\delta)$ in \eqref{LT.UBLB} is a consequence of \Cref{lem:Psi.size} and the following, together with our assumption on $p$ (via \eqref{growing.LB}).

\begin{prop}	\label{prop:LT.LB}
Let $H$ be an $r$-graph and let $\delta\in(0,1)$, $\xi\in(0,1-\delta)$. Assume there exists a collection $\BBasys=\{\Basys(e)\}_{e\in\Edges(H)}$ of weighted bases such that $\min_{e\in\Edges(H)} \growing_{n,p}(\Basys(e))\ge C(r)\xi^{-2}\log n$ for a sufficiently large constant $C(r)>0$. Then
\[
\log\P(t_p(H,\bA)\le 1-\delta) \ge -\Psi_{n,p}(H,\delta+\xi) - O\Big(1+ n^{r/2}\Lp\Big).
\]
\end{prop}

To establish the proposition we need two lemmas.
In the following we denote by $\cQ_{n,r}^{\le p}$ the set of all $Q\in\cQnr$ with entries uniformly bounded by $p$.

\begin{lem}	\label{lem:LT.p-threshold}
For any sequence of $r$-graphs $\uH=(H_1,\dots, H_m)$ and $\udelta\in(0,1)^m$, we have
$
\eye_p(\cL_{\uH}(\1-\udelta)) = \eye_p(\cL_{\uH}(\1-\udelta)\cap \cQ_{n,r}^{\le p}).
$
\end{lem}

\begin{proof}
Clearly the left hand side is bounded by the right hand side. For the reverse inequality, let $Q\in\cL_{\uH}(\1-\udelta)$ be arbitrary. Since $t_p(H_k,\cdot)$ is monotone increasing in each coordinate we have $t_p(H_k, Q\wedge p)\le t_p(H_k,Q)$ for each $k$, so $Q\wedge p\in \cL_{\uH}(\1-\udelta)$. Furthermore, since $\eye_p$ is monotone increasing under increasing $Q(I)\in [p,1]$ for any $I\in {[n]\choose r}$ with all other coordinates held fixed, we have that $\eye_p(Q)\ge \eye_p(Q\wedge p)$, and the claim follows. 
\end{proof}

\begin{lem}	\label{lem:LT.cont}
There exists $c=c(\uH)>0$ such that for any $\xi>0$,
\[
\cL_{\uH}(\1-\udelta-\xi\1)\cap \cQ_{n,r}^{\le p}\subseteq(\cA\cap\cL_{\uH}(\1-\udelta))^\circ_{\BBasys,c\xi p}.
\]
\end{lem}

\begin{proof}
We write $\cE:=\cA\cap\cL_{\uH}(\1-\udelta)$. Let $c>0$ to be taken sufficiently small depending on $H$ and put $\eps=c\xi$.
Fixing an arbitrary $Q\in \cL_{\uH}(\1-\udelta-\xi\1)\cap \cQ_{n,r}^{\le p}$,
our aim is to show that $\nbhd_{\BBasys}(Q,\eps p)\subseteq\cL_{\uH}(\1-\udelta)$. 
Since $Q\in\cQ_{n,r}^{\le p}$ we have by monotonicity that $t_p(F,Q)\le t_p(F,p)\le 1$ for every $r$-graph $F$. 
In particular, $Q\in \cL_{<\uH}(1)$. Applying \Cref{thm:count} with $\cC=\nbhd_{\BBasys}(Q,\eps p)$ and $Q_0=Q$ we have 
\[
t_p(H_k,Q) - t_p(H_k, A) \ls_H \eps
\]
for each $k\in [m]$ and every $A\in\nbhd_{\BBasys}(Q,\eps p)$. From the triangle inequality it follows that
$t_p(H_k,A)\le 1-\delta_k-\xi+O_{\uH}(\eps)$ for every such $k$ and $A$, and the claim follows by 
taking $\eps=c \xi$ for some $c=c(\xi)>0$ sufficiently small.
\end{proof}

\begin{proof}[Proof of \Cref{prop:LT.LB}]
Setting $\eps = c\xi$ for $c(\uH)>0$ sufficiently small and denoting $\cE:=\cA\cap \cL_{\uH}(\1-\udelta)$, from \Cref{thm:LDP}(b) it suffices to show that 
\[
\eye_p((\cE)^\circ_{\BBasys,\eps p}) \le \Psi_{n,p}(\uH, \udelta+\xi\1) = \eye_p(\cL_{\uH}(\1-\udelta-\xi\1)).
\]
From \Cref{lem:LT.p-threshold} the right hand side is equal to $\eye_p(\cL_{\uH}(\1-\udelta-\xi\1)\cap \cQ_{n,r}^{\le p})$, and from \Cref{lem:LT.cont} this is bounded below by the left hand side above as long as $c$ is sufficiently small.
\end{proof}

\section{\label{sec:tails.lower} Alternative lower bound argument: Concentration under the tilted law }
\label{sec:UTUB.tilt}

In this section we establish the bound \eqref{UT.UB} of \Cref{thm:tails} under the stated assumption $p\gg n^{-1/\Delta_{\min}}$. We have already shown in $\mathsection$\ref{sec:UTUB} how the lower bound can be established using the general \ldp lower bound of \Cref{thm:LDP}(b) together with the counting lemma. \Cref{thm:LDP}(b) in turn is based on concentration of the $\Basys^*$-norms under the tilted \ER laws $\mu_Q$ (see \Cref{lem:Bstar.conc}). 
Here we instead show the homomorphism counts themselves are concentrated under the tilted law. In fact we show this for the more general signed homomorphism counts of \Cref{thm:count}, which also includes \emph{induced} homomorphism counts as a special case. 

The bound  \eqref{UT.UB} under the assumption $1>p\gg n^{-1/\Delta(H)}$ is a consequence of the following.
For the upper bound assumption on $p$ recall \Cref{rmk:WLOG.p}.

\begin{prop}
\label{prop:lower}
In the setting of \Cref{thm:tails}, if $n^{-1/\Delta_{\max}}\ll p\le\min_k(1+\delta_k)^{-1/\edges(H_k)}$,
then for any fixed $\xi>0$ and all $n$ sufficiently large, 
\[
\mathbb{P}  \big(\, t_p(H_k,\bA) >1+\delta_k\,,\; 1\le k\le m\, \big) 
\ge\frac{1}{2}\exp\left(-(1+\xi)\Phi_{n,p}(\uH,\udelta+\xi\1)\right).
\]
\end{prop}

The main step is to show that for any $r$-graph $H$ and any $Q\in\cQ_n$ satisfying
\[
\eye_p(Q) \ls n^rp^{\Delta(H)} \log(1/p) \,,
\]
we have that if $p\ge n^{-1/\Delta(H)}$, then the homomorphism counting functional concentrates under $\mu_Q$, in the sense that
\[
\Var_Q\big( t_p(H, \bs A)\big) \ll \big( \E_Q t_p(H,\bs A) \big)^2.
\]
We prove the following more general statement for signed homomorphisms (recall the definitions from \eqref{def:genhom}).

\begin{prop}\label{prop:tilt.concentration}
Let $n^{-1/\Delta(H_+)}\le p\le 1-\eps_0$ for some $\eps_0>0$ and let $Q\in\cQnr$ be such that
\[
\eye_p(Q) \le K n^rp^{\Delta(H_+)}\log(1/p)
\]
for some $K\in(0,\infty)$. 
Then 
\[
\Var_{Q}(\hom(\mathcal{H},{\bA})) \ls_{K,\cH,\eps_0} \frac{\hom(\cH,Q)^2}{np^{\Delta(H_+)}}.
\]
\end{prop}

Taking \Cref{prop:tilt.concentration} as given, we now establish \Cref{prop:lower} by a tilting argument. 

\begin{proof}[Proof of \Cref{prop:lower}]

Fix $\xi>0$. We may assume $\xi\le 1$. 
The argument is slightly simpler when either $p$ is fixed or all of the $H_k$ have the same max-degree.
We thus introduce a cutoff parameter $p_0(\xi,\uH,\udelta)>0$ to be taken sufficiently small, and let $U\subseteq [m]$ be defined to be $U:=\{k\in[m]:\Delta(H_k)=\Delta_{\min}\}$ when $p<p_0$, and $U:=[m]$ when $p\ge p_0$, where we write $\Delta_{\min}:=\min_k\Delta(H_k)$. 
We write $\uH_U,\udelta_U$ for the restriction of the sequences $\uH,\udelta$ to the indices in $U$. 

For the main step of the argument, we find an element $Q_\star\in\cQnr$ for which
\begin{equation}	\label{eyep.Qstar}
\eye_p(Q_\star)\le (1+\xi/2 ) \Phi_{n,p}(\uH_U,\udelta_U+\xi\1)
\end{equation}
and such that the joint upper tail event is likely under the tilted law $\mu_{Q_\star}$, specifically:
\begin{equation}	\label{tilt:goal1}
\P_{Q_\star}\bigg( \bA\in \bigcap_{k=1}^m \cL_{\nick{H_k}}(1+\delta_k)^c\bigg) \ge 3/4
\end{equation}
(recalling the notation \eqref{def:subLint}). 
We begin by fixing an arbitrary element $Q\in\cQnr$ such that
\[
t_p(H_k,Q)\ge 1+\delta_k+\xi \qquad \forall k\in U
\]
and 
\begin{equation}	\label{tilt:assuQ}
\eye_p(Q)\le (1+\xi/4) \Phi_{n,p}(\uH_U,\udelta_U+\xi\1).
\end{equation}
By considering $Q_0\in\cQ_{n,r}$ of the form $Q_0(I) = p+ (1-p)1_{I\cap [n_0]\ne\emptyset}$ for $n_0=\lf Cnp^{\Delta_{\min}}\rf$, one easily verifies that $t_p(H_k,Q_0)\ge 2+\delta_k\ge 1+\delta_k+\xi$ for all $k\in U$ if $C=C(\uH,\udelta)>0$ is sufficiently large, and thus we see  that the \abbr{rhs} above is $O_{\uH,\udelta}(n^rp^{\Delta_{\min}} \log(1/p))$.
Hence,
\begin{equation}	\label{ipQ.bd}
\eye_p(Q) \ls_{\uH,\udelta} n^r p^{\Delta_{\min}} \log(1/p).
\end{equation}

Note that for each $k\in U$, 
\begin{equation}	\label{E_QhA-hQ}
\E_Q\hom(\nick{H_k},\bA) \ge \hom(\nick{H_k},Q).
\end{equation}
Combining this with \Cref{prop:tilt.concentration} applied with $\mathcal{H} = H_k$ (all edges receiving the positive sign) and $\eps_0=\eps_0(\udelta,\cH) = 1-\min_k(1+\delta_k)^{-1/\edges(H_k)}>0$ we have  
\[
\Var_{Q}(t_p(H_k,{\bA}))\ll (\mathbb{E}_{Q}t_p(H_k,{\bA}))^{2}
\]
for each $k\in U$.
In particular, from Chebyshev's inequality and the union bound we get 
\[
\P_{Q}\bigg( \bA\in \bigcap_{k\in U} \cL_{\nick{H_k}}(1+\delta_k)^c\bigg) \ge 1-o(1)\,.
\]
For the case that $U=[m]$ this yields \eqref{tilt:goal1} with $Q_\star=Q$. 

For the case that $p<p_0$ and not all $\Delta(H_k)$ are equal we modify $Q$ slightly.
First, setting $n_1=\lf C'np^{\Delta_{\min}+1}\rf$, 
let $Q_0'$ have entries $Q_0'(I)=p$ for all $I\subset[n_1+1,n]$, and otherwise $Q_0'(I)=1$.
By the same computation as for $Q_0$ above, we have that $t_p(H_k, Q_0') \ge 1+\delta_k+\xi$ for each $k\in [m]\setminus U$ if $C'(\uH,\udelta)$ is sufficiently large. Now let $Q_\star:= Q\vee Q_0'$ be the entrywise maximum of $Q$ and $Q_0'$. By monotonicity we have that \eqref{tilt:goal1} holds. 
Moreover, since $Q_\star$ only differs from $Q$ on $O_{\uH,\udelta}(n^rp^{\Delta_{\min}+1})$ entries we have
\[
\eye_p(Q_\star) \le \eye_p(Q) + O_{\uH,\udelta}(n^rp^{\Delta_{\min}+1}\log (1/p))\,.
\]
On the other hand, for any $k\in U$ we can lower bound $\Phi_{n,p}(\uH_U,\udelta_U+\xi\1)\ge \Phi_{n,p}(H_k,\delta_k+\xi)$, and combining with \Cref{lem:Phi.LB}, the above and \eqref{tilt:assuQ} we get
\begin{align*}
\eye_p(Q_\star) 
\le (1+\xi/4 + O_{\uH,\udelta}(p)) & \Phi_{n,p}(\uH_U,\udelta_U+\xi\1) \\
 \le (1+\xi/2 ) & \Phi_{n,p}(\uH_U,\udelta_U+\xi\1)
\end{align*}
taking $p_0$ sufficiently small depending on $\xi,\uH$ and $\udelta$, giving \eqref{eyep.Qstar} as desired. 

Setting $\cE:= \bigcap_{k=1}^m \cL_H(1+\delta_k)^c$, we next observe that
\begin{align*}
 & \mathbb{E}\,\ind({\bA}\in \cE)  =\mathbb{E}_{Q_\star}\ind({\bA}\in \cE)\exp\left(-W({\bA})\right)
\end{align*}
where 
\[
W({\bA})=\sum_{I\in{[n]\choose r}}\Big\{ {\bA}(I)\log\frac{Q(I)}{p}+(1-{\bA}(I))\log\frac{1-Q(I)}{1-p}\Big\} .
\]
By \Cref{lem:Phi.LB}, the random variable $W({\bA})$ has expectation
\[
\mathbb{E}_{Q_\star}[W({\bA})]=\eye_p(Q_\star)\gs_{\uH,\udelta}n^{r}p^{\Delta_{\min}}\log(1/p)
\]
and variance 
\begin{align*}
  \Var_{Q_\star}(W({\bA}))
 & =\sum_{I}\left(\log\frac{p}{Q_\star (I)}-\log\frac{1-p}{1-Q_\star (I)}\right)^{2}Q_\star (I)(1-Q_\star (I))\\
 & \le4\sum_{I}\left((\log p)^{2}+(\log(1-p))^{2}\right)Q_\star (I)(1-Q_\star (I))\\
 & \quad+4\sum_{I}\left((\log Q_\star (I))^{2}+(\log(1-Q_\star (I)))^{2}\right)Q_\star (I)(1-Q_\star (I))\\
 & \ls_{\uH,\udelta} n^r(\log p)^2
\end{align*}
since we assumed $1-p\gs_{\uH,\udelta} 1$.
Thus, 
\[
\Var_{Q_\star}(W({\bA})) \ll (\mathbb{E}_{Q_\star }[W({\bA})])^{2}.
\]
In particular, for sufficiently large $n$, we have a subset $\cE'$
of $\cE$ with $\mathbb{P}_{Q_\star }(\cL')\ge\mathbb{P}_{Q_\star }(\cE)-1/4\ge1/2$
such that for all $A\in \cE'$, $W(A)\le (1+\xi/10)\eye_p(Q_\star )$. 
Therefore, 
\[
\P({\bA}\in \cE)\ge\mathbb{E}_{Q_\star }\ind({\bA}\in \cE')\exp(-W({\bA}))\ge\frac{1}{2}\exp(-(1+\xi/10)\eye_p(Q_\star )),
\]
and hence
\begin{align*}
\mathbb{P}(\bA\in \cE)
&\ge\frac{1}{2}\exp(-(1+\xi/10)(1+\xi/2)\Phi_{n,p}(\uH_U,\udelta_U+\xi\1))\\
&\ge\frac{1}{2}\exp(-(1+\xi)\Phi_{n,p}(\uH_U,\udelta_U+\xi\1))\\
&\ge\frac{1}{2}\exp(-(1+\xi)\Phi_{n,p}(\uH,\udelta+\xi\1))
\end{align*}
as desired.
\end{proof}

\begin{remark}
\nick{As was pointed out to us by a referee,} for $r\ge 3$ and for $r=2$ under the stronger assumption that $(np^{\Delta_{\min}})^2 \gg \log(1/p)$, one can prove \cref{prop:lower} using much weaker control of the event $\bA \in \bigcap_{k=1}^{m} \mathcal{L}_{H_k}(1+\delta_k)^c$ (for example, via Markov's inequality), together with stronger control of the random variable $W(\bA)$, which is a sum of independent random variables. We have decided to proceed via \cref{prop:tilt.concentration} as it 
\nick{may be of independent interest.}
\end{remark}

It remains to establish \Cref{prop:tilt.concentration}. We begin by collecting a few lemmas. 
The first is a Brascamp--Lieb-type 
generalization of H\"older's inequality that has been applied extensively in
previous works analyzing the upper tail optimization problem.

\begin{lem}[Finner's inequality \cite{Finner} (see also {\cite[Theorem 3.1]{LuZh:sparse})}]
\label{lem:finner}For each $i\in[n]$, let $\Omega_{i}$ be
a probability space with measure $\mu_{i}$. Let $\mu=\bigotimes_{i=1}^{n}\mu_{i}$.
Let $A_{1},A_{2},\dots,A_{n}$ be nonempty subsets of $[n]=\{1,2,\dots,n\}$
and for $A\subseteq[n]$ let $\mu_{A}=\bigotimes_{i\in A}\mu_{i}$ and $\Omega_{A}=\prod_{i\in A}\Omega_{i}$.
Let $f_{i}\in L^{q_i}(\Omega_{A_{i}},\mu_{A_{i}})$ for each $i\le m$.
Assume that $\sum_{i:j\in A_{i}}q_i^{-1}\le1$ for all $j\le n$.
Then we have
\[
\int \prod_{i=1}^{m}f_{i}\,d\mu\le\prod_{i=1}^{m}\left(\int|f_{i}|^{q_i}d\mu_{A_{i}}\right)^{1/q_i}.
\]
\end{lem}

\begin{lem}
\label{lem:ipbelow}
For any $p\in (0,1)$ and $0\le x\le 1-p$ we have $\eye_p(p+x) \gs x^2\log(1/p)$.
\end{lem}

\begin{proof}
The claim is trivial for $p\in [c,1)$ for any fixed constant $c>0$ by the uniform convexity of $\eye_p$, so we may assume $p$ is sufficiently small.
Since $\eye_p(p)=\eye_p'(p)=0$ and $\eye_p''(x) \ge 1/x$, the claimed bound holds for $x\le 1/\log(1/p)$. 
For larger $x$ one simply notes that $\eye_p'(x) \gs \log(1/p)$ for $x\ge \sqrt{p}$ (say).
\end{proof}

\begin{lem}
\label{lem:countGip} Let $\Delta\ge2$, $K>0$, and 
let $\HG$ be an $r$-graph with maximum degree at most
$\Delta$. Suppose $Q\in\cQnr$ has all entries in $[p,1]$, and
\[
\eye_p(Q)\le Kn^{r}p^{\Delta}\log(1/p).
\]
Then
\[
 t_p(G,Q) \ls_{K,G} 1 .
\]
\end{lem}

\begin{proof}	
For economy of notation we write $\bs i= (i_{1},\dots,i_{r})$. 
Let $\tens(\bs i)=Q(\bs i)-p\in[0,1-p]$. We
have 
\[
\hom(\HG,Q)=\sum_{\HG'\subseteq \HG}p^{\edges(\HG)-\edges(\HG')}\hom(\HG',\tens),
\]
where $\HG'$ ranges over subgraphs of $\HG$ with the same
vertex set. 

Let $\HG'\subseteq \HG$. Then as the maximum degree of $\HG$ (and hence
$\HG'$) is at most $\Delta$, we have $\sum_{e\in \Edges(\HG'):v\in e}\frac{1}{\Delta}\le1$.
Thus, by Lemma \ref{lem:finner},
\begin{align*}
n^{-\verts(\HG')} & \hom(\HG',\tens)  =n^{-\verts(\HG')}\sum_{\psi:\Verts(\HG)\to[n]}\prod_{e\in \Edges(\HG')}\tens(\psi(e))\\
 & \le\prod_{e\in \Edges(\HG')}\Big(n^{-r}\sum_{\psi:e\to[n]}\tens(\psi(e))^{\Delta}\Big)^{1/\Delta} 
  \le\Big(n^{-r}\sum_{\bs i\in[n]^{r}}\tens(\bs i)^{2}\Big)^{\edges(\HG')/\Delta}.
\end{align*}
By Lemma \ref{lem:ipbelow}, 
\[
\sum_{\bs i \in[n]^{r}}\tens(\bs i)^{2}\ls \frac{\eye_p(Q)}{\log(1/p)}.
\]
Thus, 
\[
\hom(\HG',\tens)\le n^{\verts(\HG')}O\Big(\frac{n^{-r}\eye_p(Q)}{\log(1/p)}\Big)^{\edges(\HG')/\Delta}\le O(K+1)^{\edges(\HG')/\Delta} n^{\verts(\HG)}p^{\edges(\HG')},
\]
and hence
\begin{align*}
\hom(\HG,Q) & \le O(K+1)^{\edges(\HG)/\Delta}n^{\verts(\HG)}\sum_{\HG'\subseteq \HG}p^{\edges(\HG)-\edges(\HG')+\edges(\HG')}
 \ls_{K,\HG} n^{\verts(\HG)}p^{\edges(\HG)},
\end{align*}
\nick{as claimed.}
\end{proof}

Below we will need to compare $\E_Q\hom(\HG,\bA)$ with $\hom(\HG,Q)$.  By expanding the polynomial $\hom(\HG,\bA)$ and taking expectations, one sees that since the entries of $Q$ lie in $[0,1]$ we have $\E_Q\hom(\HG,\bA)\ge\hom(\HG,Q)$. The next lemma shows that a nearly matching upper bound holds.
For a hypergraph $\HG$, let $\mathcal{S}(\HG)$ be the collection of hypergraphs
$\HG'$ such that there exists a surjective map $\Psi$ from $\Verts(\HG)$ to
$\Verts(\HG')$ such that $\Edges(\HG')=\{\Psi(e): e\in \Edges(\HG)\}$ (without repeated elements). In particular, $\HG\in\cS(\HG)$. We have
\begin{equation}	\label{EQin}
\mathbb{E}_{Q} \hom(\HG,{\bA}) \le\sum_{\HG'\in\mathcal{S}(\HG)}\hom(\HG',Q)\,.
\end{equation}
The next lemma shows that the contribution from $\HG'\in\mathcal{S}(\HG)$ with $\HG'\ne\HG$ is negligible.

\begin{lem}
\label{lem:hom-control}
Let $\HG$ be an $r$-graph with
maximum degree at most $\Delta$ and let $\HG'\in\mathcal{S}(\HG)$.
Assume that $\verts(\HG')<\verts(\HG)$ and 
$p\ge n^{-1/\Delta}$.
If $Q\in\cQnr$ has all entries at least $p$ and satisfies
\[
\eye_p(Q)\le K n^{r}p^{\Delta}\log(1/p)
\]
then 
\begin{equation}\label{eq:little-o}
\hom(\HG',Q) 
\ls_{K,G} \frac{n^{\verts(\HG)}p^{\edges(\HG)}}{np^\Delta}. 
\end{equation}
Furthermore, the same conclusion holds if all entries of $Q$ are
equal to some $q\le K p$.
\end{lem}

\begin{proof}
Since $\HG'\in\mathcal{S}(\HG)$, we have a surjection $\Psi:\Verts(\HG)\to \Verts(\HG')$
such that $\Edges(\HG')=\{\Psi(e):e\in \Edges(\HG)\}$.
Let $\Psi':\Verts(\HG')\to \Verts(\HG)$ be any map such that $\Psi(\Psi'(v))=v$
for all $v\in \Verts(\HG')$. Let $\overline{\HG}$ be the induced subgraph of $\HG$
on $(\Verts(\HG'))$. Note that $\Psi'$ gives a bijection between $\Verts(\HG')$ and
$\Verts(\overline{\HG})$ such that $\Psi'^{-1}(\Edges(\overline{\HG}))\subseteq \Edges(\HG')$. 
For the first claim, observe that
\begin{align*}
\hom(\HG',Q) & =\sum_{\psi:\Verts(\HG')\to[n]}\,\,\prod_{e\in \Edges(\HG')}Q(\psi(e))
\le\sum_{\psi:\Verts(\HG')\to[n]}\,\,\prod_{e\in \Edges(\overline{\HG})}Q(\psi(\Psi'^{-1}(e)))\\
 & =\sum_{\phi:\Verts(\overline{\HG})\to[n]}\,\,\prod_{e\in \Edges(\overline{\HG})}Q(\phi(e))
  =\hom(\overline{\HG},Q),
\end{align*}
where in the second equality we let $\phi=\psi\circ \Psi'^{-1}$. 
Note that $\overline{\HG}$ has maximum degree at most $\Delta$ as it is
an induced subgraph of $\HG$. By the above and Lemma \ref{lem:countGip}, we
have 
\[
\hom(G,Q) \le 
\hom(\overline{\HG},Q)\ls_{K,\overline{\HG}} n^{\verts(\overline{\HG})}p^{\edges(\overline{\HG})}.
\]
In the case all entries of $Q$ are equal to $q\le K' p$, we trivially
have 
\[
\hom(\HG',Q)\ls_{\HG,K' }n^{\verts(\overline{\HG})}p^{\edges(\overline{\HG})}\,.
\]
Thus, for both cases, to obtain (\ref{eq:little-o}) it suffices
to show that 
\[
n^{\verts(\overline{\HG})-\verts(\HG)}p^{\edges(\overline{\HG})-\edges(\HG)}
\le (np^\Delta)^{-1}.
\]
By our assumptions that $p\ge n^{-1/\Delta}$ and $\verts(\HG')=\verts(\overline{\HG})<\verts(\HG)$,
it suffices to show that 
\[
\Delta(\verts(\HG)-\verts(\overline{\HG}))\ge\edges(\HG)-\edges(\overline{\HG}).
\]
Noting that 
\[
\Delta(\verts(\HG)-\verts(\overline{\HG}))=\Delta|\Verts(\HG)\setminus \Psi'(\Verts(\HG'))|\ge|\{e\in \Edges(\HG):e\not\subset g(\Verts(\HG'))\}|,
\]
and 
\[
\edges(\HG)-\edges(\overline{\HG})=|\{e\in \Edges(\HG):e\not\subset \Psi'(\Verts(\HG'))\}|,
\]
we obtain the desired conclusion. 
\end{proof}

\begin{proof}[Proof of \Cref{prop:tilt.concentration}]
From the Efron--Stein inequality we have that
\[
\Var_{Q}[\hom(\mathcal{H},{\bA}) ] \le\frac{1}{2}\sum_{I\in{[n] \choose r}}\mathbb{E}_{Q}\left[(\hom(\mathcal{H},{\bA})-\hom(\mathcal{H},{\bA}_{I}))^{2}\right],
\]
where $\bA_{I}, \bA_{I}^0, {\bA}_{I}^1 \in \cAn$ are equal to $\bA$
except (up to the symmetry constraint) for ${\bA}_{I}(I)$ being a $\textrm{Bernoulli}(Q(I))$ random variable
independent of ${\bA}$, whereas $\bA_{I}^1(I)=1$ and $\bA_{I}^0(I)=0$.
Taking expectation over the variables $\bA_{I}(I)$, we have the bound
\begin{align}	\label{conc.1}
\Var_{Q}[\hom(\mathcal{H},{\bA}) ] & \le
\sum_{I} Q(I)(1-Q(I))\mathbb{E}_{Q}\Delta_{I}(\cH,\bA)^2 \nonumber \\
& \le \sum_{I} Q(I)\mathbb{E}_{Q}\Delta_{I}(\cH,\bA)^2
\end{align}
where
\[
\Delta_{I}(\cH,\bA) := \hom(\mathcal{H},{\bA}_{I}^{1})-\hom(\mathcal{H},{\bA}_{I}^{0})\,.
\]
Recalling the notation
\[
A_\xi^e =\begin{cases}  A & e\in \Edges(H_+)\\ 1- A &e\in \Edges(H_-)\,.\end{cases}
\]
for $A\in\cA_n$,
we have
\begin{align*}
&\Delta_{I}(\cH,\bA)  =\sum_{E_0\subset\Edges(H)} \sum_{\substack{\phi:\Verts(H)\to [n]\\\phi^{-1}(I)=E_0}}
\prod_{e\in\Edges(H)} (\bA_{I}^1)_\xi^e(\phi_e) - \prod_{e\in\Edges(H)} (\bA_{I}^0)_\xi^e(\phi_e) \\
&=\sum_{E_0\subset\Edges(H)} \sum_{\substack{\phi:\Verts(H)\to [n]\\\phi^{-1}(I)=E_0}}
\big[1_{E_0\subseteq \Edges(H_+)} - 1_{E_0\subseteq \Edges(H_-)}\big]
\!\!\!\!\!\! 
\prod_{e\in\Edges(H_+)\setminus E_0} \!\!\!\!\!\!\!\! \bA(\phi_e) \!\!\!\!\!\! \prod_{e\in\Edges(H_-)\setminus E_0} \!\!\!\!\!\! (1-\bA(\phi_e)) .
\end{align*}
Squaring, dropping the negative summands, and bounding terms $1-\bA(\phi_e)$ by 1, we obtain
\begin{align*}
\Delta_{I}(\cH,\bA) ^2  
&\le \!\!\!\!
\sum_{\substack{ E_0,E_1\subset \Edges(H_+)\\ \text{ or } E_0,E_1\subset \Edges(H_-)}}
\sum_{\substack{\phi^0,\phi^1:\Verts(H)\to [n]\\ \phi^a(e)=I\,\forall e\in E_a, a=1,2}}
\!\!\!\!\!
\prod_{e\in \Edges(H_+)\setminus E_0}\!\!\!
\bA(\phi^0_e)\!\!\!\! \prod_{e\in \Edges(H_+)\setminus E_1} \!\!\!\!\!\bA(\phi^1_e)\,.
\end{align*}
Combining the above with the FKG inequality, we get that \eqref{conc.1} is bounded by
\begin{align*}
\E_Q
\sum_{I} \bA(I)
 \sum_{\substack{ E_0,E_1\subset \Edges(H_+)\\ \text{ or } E_0,E_1\subset \Edges(H_-)}}
\sum_{\substack{\phi^0,\phi^1:\Verts(H)\to [n]\\ \phi^a(e)=I\,\forall e\in E_a, a=1,2}}
\!\!\!\!
\prod_{e\in \Edges(H_+)\setminus E_0}\bA(\phi^0_e) \!\!\!\!\! \prod_{e\in \Edges(H_+)\setminus E_1} \!\!\!\!\bA(\phi^1_e)\,.
\end{align*}

We can interpret the above sum as counting homomorphisms of graphs $G_+$ obtained from $\cH$ as follows.
Let $H^0,H^1$ be disjoint copies of $H$. For $a=0,1$ let $E_a\subset \Edges(H^a)$ and write $V_a$ for the union of the edges in $E_a$; further, let $H_+^a$ be a subgraph of $H^a$ isomorphic to $H_+$ (with $\Verts(H_+^a)=\Verts(H^a)$ and edges corresponding to the positively labeled edges of $H$).  
Let $\wt H$ be the union of $H^0\cup H^1$ with a disjoint edge $e_\star$, and define a surjection $\Psi$ from $\Verts(\wt H) = \Verts(H^0)\cup \Verts(H^1)\cup e_\star$ to $V':=(\Verts(H^0)\setminus V_0)\cup (\Verts(H^1)\setminus V_1) \cup e_*$ as follows: 
\begin{itemize}
\item On $V'$ we take $\Psi$ to be the identity map.
\item On each $V_a$ we take $\Psi$ to be any surjection to $e_\star$ such that $\Psi(e)=e_\star$ for every $e\in E_a$.
\end{itemize}
Let $G$ be the graph on $V'$ that is the image of $\wt H$ under $\Psi$ -- that is, its edge set consists of the images $\Psi(e)$ of $e\in \Edges(\wt H)$, removing any repetitions. We further let $G_+$ be the subgraph of $G$ on the same vertex set $V'$ with edges that are the images of the edges of $H_+^0\cup H_+^1$, together with the edge $e_\star$. 
Letting $\cG_{\cH}$ be the collection of graphs $G_+$ over $V'$ that can be obtained in this way, we have
\[
\sum_{I} \bA(I) \Delta_{I}(\cH,\bA)^2 \le \sum_{G_+\in \cG_{\cH}} \hom(G_+, \bA). 
\]
Combining the previous displays and applying \eqref{EQin}, we have shown
\begin{small}
\begin{equation}	\label{VarQ.above}
\Var_Q[\hom(\cH,\bA)] \le \E_Q\sum_{G_+\in \cG_{\cH}} \hom(G_+,\bA) \le \sum_{G_+\in \cG_{\cH}} \sum_{F\in \cS(G_+)}\hom(F,Q)\,.
\end{equation}
\end{small}
Now consider an arbitrary $F\in \cS(G_+)$ for some $G_+\in \cG_{\cH}$. 
By definition, $F$ is the image of $G_+$ under a surjection $\Psi':V'\to \Verts(F)$ that maps edges to edges. Let $F'$ be the subgraph over $\Verts(F)$ obtained by removing the edge $\Psi'(e_*)$, and observe that $F'$ is the image under the restriction $\Psi''$ of $\Psi'\circ\Psi$ to $\Verts(H^0\cup H^1)$ of a subgraph $H'$ of $H_+^0\cup H_+^1$ over the same vertex set. (Specifically, it is the image of the subgraph of $H_+^0\cup H_-^0$ obtained by removing the edges $E_0\cup E_1$.)
Since $\Psi''$ is a surjection from $\Verts(H^0\cup H^1)$ to $\Verts(F)$ we have that $F'\in \cS(H_+^0\cup H_-^0)$. Moreover, 
\[
\verts(F') = \verts(F) \le |V'| = 2\verts(H)-|V_0|-|V_1|+ r < 2\verts(H)= \verts(H_+^0\cup H_+^1)
\] 
where the strict inequality follows from the fact that each of $V_0$ and $V_1$ contain at least one edge by assumption. 
Let $Q'$ be defined by the coordinate-wise maximum between $Q$ and $p$. Then $\eye_p(Q') \le \eye_p(Q) \le K n^rp^{\Delta(H_+)}\log(1/p)$ and
\[
\hom(F,Q)\le \hom(F',Q)\le \hom(F',Q').
\]
From \Cref{lem:hom-control} we have
\[
\hom(F',Q') \ls_{K,\cH} \frac{n^{\verts(H_+^0\cup H_+^1)} p^{\edges(H_+^0\cup H_+^1)}}{np^{\Delta(H_+)}}
= \frac{n^{2\verts(H)} p^{2\edges(H_+)}}{np^{\Delta(H_+)}}\,.
\]
Combining these bounds with \eqref{VarQ.above} yields the claim.
\end{proof}

\section{Other applications}
\label{sec:other}

\subsection{Upper tails for induced homomorphism counts }

For an $r$-graph $H$ and $Q\in\cQnr$, the induced homomorphism
count of $H$ in $Q$ is defined as 
\[
\indhom(H,Q)=\sum_{\phi:\Verts(H)\to [n]}\prod_{e\in \Edges(H)}Q(\phi(e))\prod_{e\in\binom{\Verts(H)}{r}\setminus \Edges(H)}(1-Q(\phi(e))).
\]
As before, this definition extends to symmetric $r$-tensors. 

For simplicity we only consider the analogue of \eqref{UT.LB}--\eqref{UT.UB} for the case $m=1$. 
Define 
\[
\UT_{n,p}^{\indhom}(H,\delta)=-\log\mathbb{P}\left(\indhom(H,{\bG})\ge(1+\delta)n^{\verts(H)}p^{\edges(H)}(1-p)^{\binom{\verts(H)}{r}-\edges(H)}\right)
\]
and the corresponding upper-tail optimization problem
\begin{align*}
\Phi_{n,p}^{\indhom}(H,\delta)=\inf\Big\{ \eye_p(Q): & Q\in\cQnr,\\
& \indhom(H,Q)\ge(1+\delta)n^{\verts(H)}p^{\edges(H)}(1-p)^{\binom{\verts(H)}{r}-\edges(H)}\Big\} .
\end{align*}

\begin{thm}
\label{thm:induced-main}
The bounds \eqref{UT.UB} and \eqref{UT.LB} hold with $\UT_{n,p}^{\indhom}(H,\delta)$, $\Phi_{n,p}^{\indhom}(H,\delta)$ in place of $\UT_{n,p}(H,\delta)$, $\Phi_{n,p}(H,\delta)$ (in the case $m=1$), under the same lower bound assumptions on $p$, and also assuming that $p\le p_0$ for an arbitrary fixed $p_0\in(0,1)$.
\end{thm}

For the lower bound on $\UT_{n,p}^{\indhom}(H,\delta)$, we follow the proof of \Cref{prop:upper}.
The only difference is that in place of \Cref{thm:count} we apply the generalized
counting lemma \Cref{thm:count.gen} with the signed hypergraph $\mathcal{K}=(K^r_{\verts(H)},\xi)$ where $\xi(e) = +1$ if $e\in \Edges(H)$ and $\xi(e) = -1$ if $e\in \binom{\Verts(H)}{r} \setminus \Edges(H)$. The subgraphs $K_{\pm}$ induced by $\xi$ are then defined by $\Verts(K_\pm)=V(H)$, $\Edges(K_+)=\Edges(H)$ and $\Edges(K_-)=\binom{\Verts(H)}{r} \setminus \Edges(H)$. We also use in the proof the fact that 
\[
\Phi_{n,p}^{\indhom}(H,\delta) \gs_{H,\delta}n^{r}p^{\Delta(H)}\log(1/p)),
\]
which follows from the argument of \cite[Theorem 2.2]{LiZh}. 
For $p=\omega(n^{-1/\Delta(H)})$, we obtain a matching upper bound 
\[
\Phi_{n,p}^{\indhom}(H,\delta) \ls_{H,\delta}n^{r}p^{\Delta(H)}\log(1/p),
\]
by fixing a subset $J_0$ of $[n]$ of size $\Theta_{H,\delta}(np^{\Delta(H)})$,
and let $Q$ be so that $Q$ takes value $1/2$ on hyperedges which
intersect $J_0$ and $Q$ takes value $p$ elsewhere. We can verify
that $\eye_p(Q)=\Theta_{H,\delta}(n^{r}p^{\Delta(H)}\log(1/p))$.

Next, we show the upper bound on $\UT_{n,p}^{\indhom}(H,\delta)$, following the proof of \Cref{prop:lower}. We highlight the main changes and additional steps. 
Let $Q\in\cQnr$
be such that
\[
\hom(\mathcal{K},Q)\ge(1+\delta+\xi)n^{\verts(K)}p^{\edges(K_+)}(1-p)^{\edges(K_-)} \gs n^{\verts(K)}p^{\edges(K_+)}
\]
and 
\[
\eye_p(Q)=\sum_{I\in{[n]\choose r}}\eye_{p}(Q(I))\le(1+\xi/4)\Phi_{n,p}^{\indhom}(H,\delta+\xi) \ls n^r p^{\Delta(K_+)}\log(1/p)\,.
\]
In order to show that
\begin{equation}	\label{ind.goal0}
\Var_{Q}(\hom(\cK,{\bA})) \ll (\mathbb{E}_{Q}\hom(\cK,{\bA}))^{2},
\end{equation}
by \Cref{prop:tilt.concentration} it suffices to show
\begin{equation}	\label{ind.goal1}
\hom(\cK,Q)\ls \E_Q \hom(\cK,\bA).
\end{equation}
The right hand side is bounded below by $\E_Q\injo(\cK,\bA)=\injo(\cK,Q)$, where we write $\injo(\cK,\cdot)$ for the count of injective signed homomorphisms. On the other hand, the count of non-injective signed homomorphisms of $\cK$ in $Q$ is at most the count of non-injective homomorphisms of $K_+=H$ in $Q$, for which \Cref{lem:hom-control} gives
\[
\hom(H,Q)-\injo(H,Q) = o(n^{\verts(H)}p^{\edges(H)}).
\]
Since $1-p \ge 1-(1+\delta)^{-1/\edges(H)} \gs_{\delta,H}1$, the right hand side is $o(\hom(\cK,Q))$. Thus, $\hom(\cK,Q)=(1+o(1))\injo(\cK,Q)$, which establishes \eqref{ind.goal1} and hence \eqref{ind.goal0}.

We can easily show that 
\[
W({\bA})=\sum_{I}\Big\{ {\bA}(I)\log\frac{Q(I)}{p}+(1-{\bA}(I))\log\frac{1-Q(I)}{1-p}\Big\}
\]
concentrates around its expectation under $\mathbb{P}_{Q}$. We then obtain the desired upper bound on $\UT_{n,p}^{\indhom}(H,\delta)$ as in \Cref{prop:lower}.

\subsection{Lower tails for Sidorenko hypergraphs}

We \nick{call} $H$ a \textit{Sidorenko hypergraph} if 
\[
\frac{\hom(H,Q)}{n^{\verts(H)}}\ge\hom(K_{r}^{r},Q)^{\edges(H)}\qquad\forall Q \in \cQnr,
\]
where $K_{r}^{r}$ is simply one hyperedge. 
For the case $r=2$, a famous conjecture
in extremal combinatorics by Erd\H{o}s and Simonovits \cite{ErSi}
and Sidorenko \cite{Sidorenko} states that all bipartite graphs are Sidorenko.
This conjecture has been verified for a large family of bipartite
graphs, including trees, even cycles, paths, hypercubes, and bipartite graphs with one vertex complete to
the other side -- see \cite{CFS,CKLL,Szegedy} and references therein. 
While a natural generalization
of Sidorenko's conjecture to hypergraphs is false, it is known that
many families of hypergraphs satisfy the Sidorenko property \cite{Szegedy}. 

Let $H$ be a graph (so $r=2$). Let $\hat{q}$ be so that $\hat{q}^{\edges(H)}\le(1-\delta)p^{\edges(H)}$
and let $q=\hat{q}\frac{n}{n-1}$. It is established in \cite{CoDe}
that 
\[
\LT_{n,p}(H,\delta)\le(1+o(1))\binom{n}{2}\eye_{p}(q),
\]
as long as $p=\omega(n^{-1/(2\Delta_{2}(H)-1)})$. Furthermore, if
$H$ is a Sidorenko graph, then the following non-asymptotic bound
holds:
\[
\LT_{n,p}(H,\delta)\ge\binom{n}{2}\eye_{p}(q).
\]
Our next
theorem generalizes this result to $r$-uniform Sidorenko hypergraphs,
and improves on the range of $p$ where the lower tail asymptotics
hold. We remark that in \cite{KoSa} the reduction of the lower tail asymptotics to the corresponding variational problem has been shown in an optimal range of sparsity for graphs ($r=2$), where the reduction is also shown for hypergraphs satisfying appropriate degree conditions.

Denote by $\mathbb{E}_{q}$ the expectation with respect to the
random $r$-graph where each hyperedge is independently included with probability
$q$.
\begin{thm}
\label{thm:LT}Let $H$ be an $r$-graph. Assume $p=\omega(n^{-1/\Delta(H)})$. Fix $\delta\in(0,1)$,  let $\hat{q}=(1-\delta)^{1/\edges(H)}p$ and let $q=\hat{q}\frac{n^{r}}{n\cdots(n-r+1)}$.
Then 
\begin{equation}
\LT_{n,p}(H,\delta)\le(1+o(1))\binom{n}{r}\eye_p\left((1-o(1))q\right).\label{eq:LT-upper}
\end{equation}
Furthermore, if $H$ is a Sidorenko hypergraph, then we have 
\begin{equation}
\LT_{n,p}(H,\delta)\ge\binom{n}{r}\eye_p\left(q\right).\label{eq:LT-lower}
\end{equation}
\end{thm}

Our result thus yields the lower tail asymptotics as long as $H$
is Sidorenko and $p=\omega(n^{-1/\Delta(H)})$. We remark that in
the regime $p=\omega(n^{-1/\Delta(H)})$, we can verify that $\hat{q}=\Theta(p)$
so $q=\Theta(p)$. In the case $r=2$, this improves the threshold
in \cite{CoDe}. 

We turn to the proof of \Cref{thm:LT}.
We first give the proof of (\ref{eq:LT-upper}) following
the proof of \Cref{prop:lower}. Let $\xi>0$ be any sufficiently small real number. We choose $\tilde{q}=q(1-\xi)$.
We write $\mathbb{E}_{\tilde{q}}$ and $\Var_{\tilde{q}}$
for expectation and variance under the distribution of a random tensor ${\bA}$ whose entries are 
i.i.d. $\textrm{Bernoulli}(\tilde{q})$ variables. We first establish an analogue of \Cref{prop:tilt.concentration} showing the concentration of $\hom(H,{\bA})$ where ${\bA}$
has independent $\textrm{Bernoulli}(\tilde{q})$ entries. In particular, we
show that 
\begin{equation}
\Var_{\tilde{q}}(\hom(H,{\bA}))\ll(\mathbb{E}_{\tilde{q}}\hom(H,{\bA}))^{2}.\label{eq:hom-var}
\end{equation}

First, notice that $cp\le \tilde{q}\le p$ for some constant $c\in (0,1)$ depending only on $\delta$. We have that 
\begin{equation}
\mathbb{E}_{\tilde{q}}\hom(H,{\bA})\ge(1+o(1))n^{\verts(H)}\tilde{q}^{\edges(H)}.\label{eq:Eqhom}
\end{equation}
Indeed, by summing over the injective homomorphisms, we obtain 
\[
\mathbb{E}_{\tilde{q}}\hom(H,{\bA})\ge(1-o(1))n^{\verts(H)}\tilde{q}^{\edges(H)}.
\]

Recall that we denote by $\jay \in \cAnr$ the tensor with all ``off-diagonal'' elements equal to $1$. Following identically the proof of \Cref{prop:tilt.concentration}, we obtain that 
\[
\Var_{\tilde{q}}(\hom(H,{\bf A})) \ls_H \sum_{\HG\in \mathcal{S}(\tilde{H}):\verts(\HG)<2\verts(H)} \hom(\HG,\tilde{q}\jay),
\]where $\tilde{H}$ is the hypergraph obtained from two disjoint copies of $H$, and recall that $\mathcal{S}(\tilde{H})$ is the collection of hypergraphs $\HG$ such that there exists a surjective map $f$ from $\Verts(\tilde{H})$ to $\Verts(\HG)$ such that $\Edges(\HG)=f(\Edges(\tilde{H}))$. 
By Lemma \ref{lem:hom-control} applied with $Q = \tilde{q}\jay$, we obtain that for each $\HG\in \mathcal{S}(\tilde{H})$, $$\hom(\HG,\tilde{q}\jay) = o(n^{2\verts(H)}p^{2\edges(H)}).$$Thus, $$\Var_{\tilde{q}}(\hom(H,{\bf A})) = o(n^{2\verts(H)}p^{2\edges(H)}) = o((\mathbb{E}_{\tilde{q}}\hom(H,{\bA}))^2),$$using (\ref{eq:Eqhom}), yielding (\ref{eq:hom-var}). 
The concentration of
\[
W({\bA})=\sum_{\bs i}\Big\{ {\bA}(\bs i)\log\frac{\tilde{q}}{p}+(1-{\bA}(\bs i))\log\frac{1-\tilde{q}}{1-p}\Big\} 
\]
easily follows noting that 
\[
\mathbb{E}_{\tilde{q}}W({\bA})\ge cn^{r}\Big(\tilde{q}\log\frac{\tilde{q}}{p}+(1-\tilde{q})\log\frac{1-\tilde{q}}{1-p}\Big),
\]
and 
\[
\Var_{\tilde{q}}(W({\bA}))\le Cn^{r}\tilde{q}(1-\tilde{q})\Big(\log\frac{p}{\tilde{q}}-\log\frac{1-p}{1-\tilde{q}}\Big)^{2},
\]
so as $\tilde{q}\in [cp,p]$, we have
$
\Var_{\tilde{q}}(W({\bA}))=o((\mathbb{E}_{\tilde{q}}W({\bA}))^{2}).\label{eq:W-var}
$
Combining 
this with (\ref{eq:hom-var}),
we obtain (\ref{eq:LT-upper})
as in the proof of \Cref{prop:lower}. 

To establish (\ref{eq:LT-lower}) under the additional assumption
that $H$ is Sidorenko, we note that if $\hom(H,{\bA})\le(1-\delta)p^{\edges(H)}n^{\verts(H)}$,
then by the Sidorenko property, 
$
\hom(K_{r}^{r},{\bA})\le\hat{q}.
$
Noting that 
\[
\hom(K_{r}^{r},{\bA})=n^{-r}\sum_{\bs i \in [n]^r}{\bA}(\bs i)=\frac{n\cdots(n-r+1)}{n^{r}}\binom{n}{r}^{-1}\sum_{\bs i}{\bA}(\bs i),
\]
(\ref{eq:LT-lower}) follows from basic properties of the
binomial distribution.
\qedhere

\subsection*{Acknowledgments}
We thank Bhaswar Bhattacharya, Ronen Eldan, Jacob Fox, Shirshendu Ganguly, Eyal Lubetzky, Wojciech Samotij and Yufei Zhao for helpful comments on an earlier version of the paper, as well as Yang P.\ Liu and Yufei Zhao for their permission to 
use \Cref{fig:LiZh}.
\nick{We additionally thank the anonymous referees for their feedback and helpful suggestions to improve the exposition.}

\bibliographystyle{abbrv}
\bibliography{hgtails.bib}

\end{document}